\pdfoutput=1
\documentclass[11pt]{article}
\usepackage[left=1in,right=1in,top=1in,bottom=1in]{geometry}
\usepackage{times}
\usepackage{expl3}
\usepackage{cite}
\usepackage[table]{xcolor}
\usepackage{multirow}
\usepackage{stackengine} 
\usepackage{hhline}
\usepackage{lipsum}
\usepackage{titlesec}
\usepackage[all]{xy}
\usepackage{wrapfig}
\usepackage{enumerate}
\usepackage{epsfig}
\usepackage{tikz-cd}
\usepackage{amsmath}
\usepackage{tabularx}
\usepackage{array}
\usepackage{booktabs}
\usepackage{enumitem}
\usepackage{bbm}
\usepackage{calc}
\usepackage{graphicx}
\usepackage{amsmath}
\usepackage[title]{appendix}
\usepackage{amssymb}
\usepackage{epstopdf}
\usepackage{boldline}
\usepackage{arydshln}
\usepackage{calligra}
\usepackage{bm}
\usepackage{url}
\usepackage{blindtext}
\usepackage{accents}
\usepackage{amsthm}

\newtheorem{definition}{Definition}

\newtheorem{theorem}{Theorem}
\newtheorem{proposition}{Proposition}
\newtheorem{corollary}{Corollary}

\newtheorem{example}{Example}
\newtheorem{remark}{Remark}
\usepackage{mathtools}
\usepackage{epstopdf}
\usepackage{balance}
\usepackage{thmtools}
\usepackage{thm-restate}
\usepackage{hyperref}
\usepackage{cleveref}
\usepackage[mathscr]{euscript}

\usepackage[ruled,vlined]{algorithm2e}
\newcommand{\ostar}{\mathbin{\mathpalette\make@circled\star}}

\makeatletter
\newcommand{\removelatexerror}{\let\@latex@error\@gobble}
\makeatother
\makeatletter
\newcommand*{\rom}[1]{\expandafter\@slowromancap\romannumeral #1@}
\makeatother

\ExplSyntaxOn
\newcommand\latinabbrev[1]{
  \peek_meaning:NTF . {
    #1\@}%
  { \peek_catcode:NTF a {
      #1.\@ }%
    {#1.\@}}}
\ExplSyntaxOff

\titleclass{\subsubsubsection}{straight}[\subsubsection]

\begin{document}
\vspace{1cm}
\title{Categorified Spectral Sheaves and Homotopical Invariants for Noncommuting Operators}
\vspace{1.8cm}
\author{Shih-Yu~Chang
\thanks{Shih-Yu Chang is with the Department of Applied Data Science,
San Jose State University, San Jose, CA, U. S. A. (e-mail: {\tt
shihyu.chang@sjsu.edu})
}}

\maketitle

\begin{abstract}
Classical spectral theory gives a complete description of a single normal operator, but it fails for noncommuting operators, where no canonical joint spectrum or simultaneous diagonalization exists. Existing approaches provide only partial solutions: joint spectra apply mainly to commuting families, noncommutative geometry emphasizes global invariants that obscure local structure, and topos-theoretic methods capture contextuality while losing higher coherence information. This paper proposes a geometric and higher-categorical reformulation of the spectral problem for noncommuting operators. Local classical spectra associated with commutative subalgebras are organized into a stack-valued object, called a spectral stack, which retains automorphism and unitary equivalence data between contexts. Noncommutativity is thereby interpreted as nontrivial descent data rather than a breakdown of spectral theory. Using homotopical and derived constructions, we define functorial invariants that measure obstructions to global spectral assembly and extend classical index-theoretic ideas. The resulting framework views noncommutative operator algebras as geometric objects equipped with a spectral atlas, providing a concise bridge between operator theory, homotopy theory, and higher geometry. 
\end{abstract}

\section{Introduction}\label{sec:Introduction}

\subsection{The Spectral Problem of Non-Commutation}

Spectral theory is a central pillar of functional analysis and operator theory.
For a single normal operator $T$ on a Hilbert space, the spectral theorem
provides a complete description of $T$ via a projection-valued measure supported
on its spectrum $\sigma(T) \subset \mathbb{C}$, yielding a powerful functional
calculus and a clear geometric interpretation of observables
\cite{ReedSimonIV}.

This classical picture breaks down in a fundamental way for noncommuting
operators. Given $T,S \in \mathcal{B}(\mathcal{H})$ with $[T,S] \neq 0$, there is
in general no well-defined joint spectrum $\sigma(T,S) \subset \mathbb{C}^2$,
nor does there exist a simultaneous eigenbasis. As a consequence, the analytic
tools underlying the spectral theorem---joint diagonalization, functional
calculus, and spectral measures---fail to extend in any canonical fashion
\cite{ReedSimonIV}. This obstruction reflects an intrinsic geometric
incompatibility rather than a technical limitation.

Various partial remedies have been proposed within operator theory. Joint
spectral constructions, such as the Taylor spectrum, recover meaningful
information for tuples of commuting operators and admit functorial properties
in restricted settings \cite{Taylor1970}. However, these constructions do not
extend naturally to genuinely noncommuting families and do not capture higher
coherence or equivalence data arising from unitary transformations.

Noncommutative geometry, pioneered by Connes, approaches this difficulty by
replacing classical spaces with noncommutative $C^*$-algebras and geometric
invariants with $K$-theory, cyclic cohomology, and index pairings
\cite{connes1994}. While extraordinarily successful, this framework often
compresses spectral information into global homological invariants, thereby
obscuring the local geometric structure of noncommuting observables.

A complementary conceptual approach arises from topos-theoretic formulations of
quantum theory. In this perspective, one considers the poset $\mathcal{C}$ of
commutative subalgebras (contexts) of a noncommutative $C^*$-algebra
$\mathcal{A}$. Each context admits a classical Gelfand spectrum, and these local
spectra assemble into a spectral presheaf encoding contextuality and logical
obstructions~\cite{doering2012}. However, being set-valued, this presheaf necessarily discards higher coherence data,
such as automorphisms induced by unitary equivalences between contexts.

\textbf{To retain such coherence, we propose a categorified refinement:}
the present work replaces the spectral presheaf with a stack-valued object. 
The resulting \emph{spectral stack} retains automorphism data and naturally 
encodes unitary equivalences between local spectral charts. 
By applying the nerve construction and subsequent stabilization, one obtains 
a derived spectral object whose homotopy limit produces stable homotopy groups. 
These groups define functorial invariants of noncommutativity, extending 
classical index-theoretic constructions and measuring obstructions to global 
spectral descent.

From this viewpoint, a noncommutative $C^*$-algebra may be regarded as a
geometric object equipped with a \emph{spectral atlas}: commutative contexts
serve as local affine charts, while their higher gluing data is encoded by
stacky and derived structures. This perspective aligns naturally with the
language of higher and derived geometry \cite{LurieHA}, providing a
conceptual bridge between operator algebras, homotopy theory, and modern
geometric frameworks.

\subsection{Geometric Reformulation and Contextuality}
\label{subsec:geometric-contextuality}

The spectral analysis of noncommuting operators reveals a fundamental obstruction: the classical notion of a \emph{joint spectrum} ceases to exist. In its place emerges a parameterized family of \emph{local classical spectral data}, organized by the contexts in which observables commute.

Formally, for a noncommutative $C^*$-algebra $\mathcal{A}$, consider the poset $\mathcal{C}$ of its commutative $C^*$-subalgebras—the \emph{measurement contexts}. Gelfand duality assigns to each $C \in \mathcal{C}$ its spectrum $\Sigma(C)$, a locally compact Hausdorff space. The assignment $C \mapsto \Sigma(C)$ defines the \textbf{spectral presheaf} \cite{doering2008}, a contravariant set-valued functor on $\mathcal{C}$. The absence of its global sections—consistent families of local spectral points—when the algebra is noncommutative, is intimately related to the Kochen--Specker theorem \cite{kochen1967}, which proves the impossibility of non-contextual hidden variables in quantum mechanics.

This set-valued presheaf, however, captures only the existence of points. A deeper geometric picture requires retaining information about \emph{identifications} between spectral data across overlapping contexts. This motivates the passage to a \textbf{categorical enhancement} of the spectral presheaf, where each fiber is enriched to track isomorphisms, restrictions, and coherent equivalences. In topos‑theoretic quantum foundations, this leads to the notion of a sheaf of categories or a stack \cite{doering2012}.

The central contribution of this work is to subject this enriched spectral data to systematic \textbf{homotopical analysis}. By promoting the categorical presheaf to a presheaf of spectra $\mathscr{G}_{\mathrm{st}}$ and studying its \textbf{homotopy limit} over $\mathcal{C}$, we obtain a graded series of homotopical cohomology groups $\mathbb{H}^n(\mathcal{C}; \mathscr{G}_{\mathrm{st}})$. Theorem~\ref{thm:homotopical-characterization} shows that these groups provide a precise, graded measure of noncommutativity:
\begin{itemize}
    \item Vanishing of $\mathbb{H}^n$ for $n>0$ characterizes commutativity.
    \item Non‑trivial $\mathbb{H}^1$ classifies \emph{monodromy} of spectral data around loops of contexts, giving a homotopical refinement of Kochen--Specker contextuality.
    \item Higher groups $\mathbb{H}^n$ ($n \geq 2$) detect increasingly coherent failures of gluing, corresponding to higher categorical structures such as gerbes \cite{breen1994classification}.
\end{itemize}

Thus, we elevate contextuality from a logical no‑go theorem to a quantifiable \emph{homotopical obstruction theory}. This bridges operator‑algebraic noncommutativity with tools from modern algebraic topology, revealing a rich, hierarchical structure in the spectral geometry of noncommuting operators.

Building on recent multicategorical and operator-theoretic frameworks that treat operator algebras as structured geometric and higher-categorical objects, this work reformulates the spectral problem for noncommuting operators by organizing local classical spectra into a stack-valued invariant that captures descent, coherence, and functorial obstruction data~\cite{chang2025hilbmultbanachenrichedmulticategoryoperator,chang2025compositioncoherencesyntaxoperator,chang2025multicategoricaladjointsmonadicityquantum}.

\subsection{Main Contributions}

The idea that noncommutative mathematical objects may be studied through families of commutative substructures is well established, particularly in approaches such as Bohrification and topos-theoretic quantum theory. The present work builds on this perspective by introducing a refined \emph{structural} and \emph{computational} framework in which spectral data associated to noncommuting operators is organized using higher-categorical and homotopical methods. Rather than interpreting the absence of a classical joint spectrum as a failure of existing tools, we propose a geometric formulation in which such phenomena are encoded by descent and obstruction data.

\paragraph{Categorified spectral data.} Instead of presheaves of sets or commutative algebras, we construct a sheaf of categories
\[
\mathscr{G} : \mathcal{C}^{\mathrm{op}} \longrightarrow \mathbf{Cat},
\]
whose objects describe local spectral categories associated to compatible commutative contexts, and whose morphisms encode coherent transitions between these contexts. This construction provides a natural extension of classical spectral theory to a higher-categorical setting, while remaining compatible with standard operator-theoretic structures.

\paragraph{Stack descent interpretation.} We show that the resulting categorified spectral data satisfies \emph{stack descent}. In this formulation, the failure of a global joint spectrum is reflected in the nontriviality of descent data, indicating that the appropriate geometric object is a spectral stack rather than an ordinary space. This viewpoint allows one to treat noncommutativity as a structural property captured by gluing conditions, rather than as an obstruction external to geometry.

\paragraph{Homotopical invariants.} We introduce and analyze \emph{homotopical invariants} associated to the spectral stack, defining both the sheaf cohomology $H^*(\mathcal{C}; F_K)$ and the hypercohomology $\mathbb{H}^*(\mathcal{C}; B_{\mathscr{G}})$ of the associated classifying stack. These invariants provide functorial measures of the extent to which local spectral data fails to assemble globally and may be interpreted as higher-order obstructions to spectral descent. In particular, they offer a systematic way to extract computable global information from noncommuting systems.

From a conceptual standpoint, the framework developed here offers a reinterpretation of classical spectral phenomena in geometric terms. Commutative systems correspond to trivial descent and vanishing cohomological invariants, while genuinely noncommutative systems give rise to nontrivial stack-theoretic and homotopical structure. This correspondence is summarized in the following dictionary, which relates notions from operator theory to their geometric counterparts:

\begin{center}
\begin{tabular}{@{}ll@{}}
\toprule
\textbf{Operator Theory} & \textbf{Geometric Interpretation} \\
\midrule
Non-commuting pair $(T,S)$ & Spectral stack $\mathscr{G}$ \\
Absence of joint spectrum & Nontrivial descent data \\
Context dependence & Stack-theoretic structure \\
Commutativity $[T,S]=0$ & Trivial sheaf cohomology \\
Spectral projections & Local sections of $\mathscr{G}$ \\
Von Neumann algebra $\mathcal{A}$ & Global sections or limits of $\mathscr{G}$ \\
\bottomrule
\end{tabular}
\end{center}

Finally, the proposed framework is constructive in nature. We provide an explicit association $(T,S) \mapsto \mathscr{G}$, demonstrate how cohomological and homotopical invariants can be used to detect and compare noncommutative behavior, and outline finite and computable approximations based on the indexing category $\mathcal{C}$. These results suggest that stack-theoretic and derived techniques offer a viable and flexible extension of classical spectral theory for the analysis of noncommuting operators.

\begin{remark}
The author is solely responsible for the mathematical insights and theoretical directions proposed in this work. AI tools, including OpenAI's ChatGPT and DeepSeek models, were employed solely to assist in verifying ideas, organizing references, and ensuring internal consistency of exposition~\cite{chatgpt2025,deepseek2025}.
\end{remark}

\section{Preliminaries}

\subsection{Operator Algebra Setup}\label{sec:Operator Algebra Setup}

This section establishes the fundamental operator-theoretic framework and notation that will be used throughout the paper. We provide precise definitions of the algebras, spectral data, and contextual structures that form the foundation of our geometric approach.

\begin{definition}[Fundamental Framework]\label{def:fundamental_framework}
Let $\mathcal{H}$ be a complex Hilbert space, and let $\mathcal{B}(\mathcal{H})$ denote the $C^*$-algebra of bounded linear operators on $\mathcal{H}$. 
Unless explicitly stated otherwise, all operators considered in this work are assumed to be bounded and everywhere defined. 
Extensions to densely defined unbounded operators may be treated via affiliated operators or by passing to the associated von Neumann algebra.

Fix a unital $*$-subalgebra $\mathcal{A} \subseteq \mathcal{B}(\mathcal{H})$, typically generated by a finite collection of self-adjoint operators $\{T_1,\dots,T_n\}$. 
In applications of interest, $\mathcal{A}$ will often be a $C^*$-algebra or, when finer spectral or measure-theoretic control is required, its weak operator topology closure $\overline{\mathcal{A}}^{\,\mathrm{wot}}$, which is a von Neumann algebra.

Let $\mathrm{Comm}(\mathcal{A})$ denote the partially ordered set of unital commutative $*$-subalgebras of $\mathcal{A}$, ordered by inclusion; elements of $\mathrm{Comm}(\mathcal{A})$ are referred to as \emph{contexts}. 
We regard $\mathrm{Comm}(\mathcal{A})$ as a small category in the usual way.

The \emph{spectral presheaf} of $\mathcal{A}$ is the contravariant functor
\[
\mathscr{G} \colon \mathrm{Comm}(\mathcal{A})^{\mathrm{op}} \longrightarrow \mathbf{Set},
\]
which assigns to each context $\mathcal{B} \in \mathrm{Comm}(\mathcal{A})$ its Gelfand spectrum $\Sigma(\mathcal{B})$, equipped with the weak-$*$ topology, and to each inclusion $\mathcal{B}_1 \subseteq \mathcal{B}_2$ the canonical restriction map
\[
\Sigma(\mathcal{B}_2) \longrightarrow \Sigma(\mathcal{B}_1).
\]

The presheaf $\mathscr{G}$ encodes the family of local classical spectral descriptions associated with the noncommutative algebra $\mathcal{A}$. 
At this stage, the detailed internal noncommutative structure of $\mathcal{A}$ is not used explicitly; rather, the framework is determined by the Gelfand spectra of its commutative subalgebras and the compatibility maps between them. 
Global structural and spectral information about $\mathcal{A}$ will later be recovered in a derived or stack-theoretic sense (Theorem~\ref{thm:spectral_sheaf}) from this diagram of commutative approximations.
\end{definition}

\paragraph{Non-Commutativity as the Core Problem.}
Given $T, S \in \mathcal{A}$, we say they \emph{commute} if $TS = ST$. The \emph{commutator} $[T,S] := TS - ST$ quantifies their failure to commute—a purely algebraic condition with profound analytic consequences. Classically, commuting self-adjoint operators admit a joint spectral measure on $\mathbb{R}^2$ and a basis of simultaneous eigenvectors, enabling their simultaneous diagonalization. For non-commuting pairs, this joint spectral structure \emph{ceases to exist}; the traditional analytic toolkit reaches its limit precisely here.

Rather than treating this as a terminal pathology, we reinterpret it as the genesis of new geometric structure. The absence of a classical joint spectrum signals not an analytic failure but the emergence of a richer, \emph{context-dependent} spectral object. This work transforms the algebraic condition $[T,S] \neq 0$ into a geometric problem: constructing a \emph{spectral stack} $\mathscr{G}$ that coherently organizes all local (commutative) spectral perspectives, with the original non-commutativity encoded in the stack's topological complexity.

\paragraph{Spectra and Functional Calculus.}
For a bounded operator $T \in \mathcal{B}(\mathcal{H})$, its \emph{spectrum} is defined as
\[
\sigma(T) := \{ \lambda \in \mathbb{C} \mid T - \lambda I \text{ is not invertible in } \mathcal{B}(\mathcal{H}) \}.
\]
If $T$ is self-adjoint ($T = T^*$), then $\sigma(T) \subset \mathbb{R}$. The spectral theorem associates to $T$ a unique projection-valued measure
\[
E_T : \mathcal{B}(\mathbb{R}) \to \{\text{orthogonal projections on } \mathcal{H}\},
\]
which determines $T$ via
\[
T = \int_{\mathbb{R}} \lambda \, dE_T(\lambda).
\]
This measure gives rise to the \emph{bounded Borel functional calculus}: for any bounded Borel function $f : \mathbb{R} \to \mathbb{C}$, the operator
\[
f(T) := \int_{\mathbb{R}} f(\lambda)\, dE_T(\lambda)
\]
is well defined. The assignment $f \mapsto f(T)$ defines a unital $*$-homomorphism from the algebra of bounded Borel functions on $\mathbb{R}$ into $\mathcal{B}(\mathcal{H})$, which is norm-decreasing and becomes isometric when viewed as a representation of $L^\infty(\mathbb{R}, E_T)$.

\paragraph{Contexts: Commutative Subalgebras.}
A central object in our framework is the family of commutative subalgebras of $\mathcal{A}$, which represent ``classical snapshots'' or contexts within the non-commutative whole.

\begin{definition}[Context]\label{def:context}
A \emph{context} is a unital commutative $C^*$-subalgebra $\mathcal{B} \subseteq \mathcal{A}$.
We denote by $\mathrm{Comm}(\mathcal{A})$ the collection of all contexts.
A context $\mathcal{B}$ is said to \emph{contain} an operator $T \in \mathcal{A}$ if $T \in \mathcal{B}$.
\end{definition}

\begin{example}[Canonical Contexts]\label{ex:canonical_contexts}
For any self-adjoint operator $T \in \mathcal{A}$, the unital $C^*$-algebra $C^*(T,I)$ generated by $T$ and the identity is a context.
Its weak operator closure is the von Neumann algebra $W^*(T)$ generated by $T$, which is (non-canonically) $*$-isomorphic to
$L^\infty(\sigma(T), \mu_T)$ for a measure $\mu_T$ arising from the spectral measure of $T$.

If $T$ and $S$ are self-adjoint operators that do not commute, then the $C^*$-algebra $C^*(T,S,I)$ is noncommutative and therefore does not define a context.
\end{example}

For any commutative unital \(C^*\)-algebra \(\mathcal{B}\), the Gelfand--Naimark theorem yields a canonical \(*\)-isomorphism of \(C^*\)-algebras:
\[
\mathcal{B} \;\cong\; C\bigl(\Sigma(\mathcal{B})\bigr),
\]
where \(\Sigma(\mathcal{B})\) denotes the \emph{Gelfand spectrum} of \(\mathcal{B}\)---a compact Hausdorff space whose points are the characters (nonzero multiplicative linear functionals) \(\phi\colon \mathcal{B}\to \mathbb{C}\).

Under this isomorphism, each element \(T \in \mathcal{B}\) corresponds uniquely to a continuous function \(\widehat{T} \in C(\Sigma(\mathcal{B}))\) given by evaluation:
\[
\widehat{T}(\phi) = \phi(T), \qquad \phi \in \Sigma(\mathcal{B}).
\]

Consequently, the continuous functional calculus for a normal operator \(T \in \mathcal{B}\) is implemented by applying the corresponding continuous function pointwise on the spectrum; that is, for \(f \in C(\sigma(T))\),
\[
f(T) \longleftrightarrow f\!\circ \widehat{T}\; (\text{pointwise on } \Sigma(\mathcal{B})).
\]

\paragraph{The Poset of Contexts and Spectral Site.}
The set $\mathrm{Comm}(\mathcal{A})$ is naturally a partially ordered set (poset) under inclusion: $\mathcal{B}' \le \mathcal{B}$ if $\mathcal{B}' \subseteq \mathcal{B}$. An inclusion $i: \mathcal{B}' \hookrightarrow \mathcal{B}$ induces a continuous restriction (or pullback) map between the Gelfand spectra:
\[
\Sigma(i): \Sigma(\mathcal{B}) \to \Sigma(\mathcal{B}'), \quad \chi \mapsto \chi|_{\mathcal{B}'}.
\]

\begin{definition}[Spectral Site]\label{def:spectral_site}
Let $\mathcal{A}$ be a unital $C^*$-algebra as in Definition~\ref{def:fundamental_framework}, and let 
\[
\mathrm{Comm}(\mathcal{A}) := \{\mathcal{B} \subseteq \mathcal{A} \mid \mathcal{B} \text{ is a unital commutative } C^*\text{-subalgebra}\}
\]
denote the set of \emph{contexts}.  

We regard $\mathrm{Comm}(\mathcal{A})$ as a poset ordered by inclusion, and denote by $\mathcal{C} := \mathrm{Comm}(\mathcal{A})^{\mathrm{op}}$ its opposite category: objects are contexts $\mathcal{B}$, and there is a unique morphism $\mathcal{B}_2 \to \mathcal{B}_1$ precisely when $\mathcal{B}_1 \subseteq \mathcal{B}_2$.

The \emph{spectral site} is the pair $(\mathcal{C}, J)$, where $J$ is the Grothendieck topology defined as follows: a family of morphisms $\{\mathcal{B}_i \to \mathcal{B}\}_{i \in I}$ in $\mathcal{C}$ is a $J$-covering if the $C^*$-subalgebra generated by the union of the $\mathcal{B}_i$ is dense in $\mathcal{B}$, i.e.,
\[
\mathcal{B} = \overline{C^*\Big(\bigcup_{i \in I} \mathcal{B}_i \Big)}^{\|\cdot\|}.
\]

A presheaf on $(\mathcal{C}, J)$ is a \emph{sheaf} if it satisfies the usual sheaf condition with respect to all $J$-coverings.  
The \emph{spectral presheaf} $\mathscr{G}$ (Definition~\ref{def:fundamental_framework}) assigns to each context $\mathcal{B}$ its Gelfand spectrum $\Sigma(\mathcal{B})$, with restriction maps induced by inclusions; it is indeed a sheaf with respect to the topology $J$.
\end{definition}

\begin{example}[Two-Operator Case]
For a pair of self-adjoint operators $(T,S)$ on a Hilbert space $\mathcal{H}$, 
let $\mathcal{A} = C^*(T,S)$. Then $\mathrm{Comm}(\mathcal{A})$ consists precisely 
of those commutative $C^*$-subalgebras of $\mathcal{A}$ that contain both $T$ and $S$. 
This recovers the special case mentioned in Theorem~\ref{thm:spectral_sheaf}.
\end{example}

\paragraph{From Algebraic Obstruction to Geometric Structure.}
If the algebra $\mathcal{A}$ were commutative, its global spectral information would be fully encoded by the single space $\Sigma(\mathcal{A})$. Non-commutativity fundamentally disrupts this picture: spectral data exist only locally within each commutative context $\mathcal{B} \in \mathcal{C}$, and these local spectra are generally incompatible across different contexts. Crucially, this incompatibility is not merely a limitation but encodes rich geometric and cohomological structure.

This motivates the passage from the operator algebra $\mathcal{A}$ to a geometric object that systematically encodes:
\begin{itemize}
    \item \textbf{Local spectral data} associated to each context $\mathcal{B} \in \mathcal{C}$ via its Gelfand spectrum $\Sigma(\mathcal{B})$,
    \item \textbf{Restriction maps} $\Sigma(\mathcal{B}) \to \Sigma(\mathcal{B}')$ induced by inclusions $\mathcal{B}' \subseteq \mathcal{B}$, formalizing the compatibility of local spectra,
    \item \textbf{Global obstructions} to gluing local spectra into a single classical object, captured by cohomological and homotopical invariants.
\end{itemize}

These concepts are formalized by constructing the \emph{spectral presheaf} 
\[
\mathscr{G} : \mathcal{C}^{\mathrm{op}} \longrightarrow \mathbf{Top},
\] 
which assigns to each context its spectrum and to each inclusion the corresponding continuous restriction map (Section~\ref{sec:Categorical and Sheaf-Theoretic Background}). We show that $\mathscr{G}$ satisfies a \emph{descent condition}, making it a sheaf of topological spaces (or locales) for the spectral topology (Section~\ref{sec:sheaf-theorem}). From this structure, we extract novel invariants that systematically quantify non-commutativity (Section~\ref{sec:homotopical-invariants}).

\subsection{Categorical and Sheaf-Theoretic Background}\label{sec:Categorical and Sheaf-Theoretic Background}

This subsection provides the essential categorical and sheaf-theoretic foundations for our geometric reformulation of non-commutativity. We recall the key notions of sites, sheaves, descent, and their homotopical enhancements, with a focus on their application to operator algebras. Standard references include \cite{maclane, kashiwara_schapira} for category theory and sheaves, and \cite{johnstone, lurie} for higher categorical aspects.

\paragraph{Categories and Presheaves.}
We work primarily within the framework of category theory. A \emph{category} $\mathscr{C}$ consists of objects and morphisms with associative composition and identity morphisms. In our setting, categories often arise from partially ordered sets: a poset $(P, \le)$ becomes a category with objects $p \in P$ and a unique morphism $p \to q$ if $p \le q$. The \emph{opposite category} $\mathscr{C}^{\mathrm{op}}$ has the same objects but reversed morphisms.

\begin{definition}[Presheaf]\label{def:presheaf}
For a small category $\mathscr{C}$ and a target category $\mathcal{D}$, a \emph{presheaf} on $\mathscr{C}$ with values in $\mathcal{D}$ is a contravariant functor
\[
F: \mathscr{C}^{\mathrm{op}} \longrightarrow \mathcal{D}.
\]
When $\mathcal{D} = \mathbf{Set}$, $\mathbf{Ab}$, or $\mathbf{Cat}$, we speak of presheaves of sets, abelian groups, or categories, respectively.
\end{definition}

Let $\mathcal{A}$ be a noncommutative $C^*$-algebra, and let $\mathcal{C}$ denote the poset category of its commutative unital $*$-subalgebras (contexts), ordered by inclusion.

\begin{example}[Classical Set-Valued Spectral Presheaf]
Let $\mathcal{C}$ be the poset of commutative unital subalgebras (contexts) of a $C^*$-algebra $\mathcal{A}$, ordered by inclusion.  
The \emph{classical spectral presheaf} is the contravariant functor
\[
\underline{\Sigma} : \mathcal{C}^{\mathrm{op}} \longrightarrow \mathbf{Set},
\]
defined as follows:
\begin{itemize}
    \item For each context $B \in \mathcal{C}$, assign its Gelfand spectrum
    \[
    \underline{\Sigma}(B) := \Sigma(B),
    \]
    which is the compact Hausdorff space of characters $\chi: B \to \mathbb{C}$.
    \item For each inclusion $i: B_1 \hookrightarrow B_2$, assign the restriction map
    \[
    \underline{\Sigma}(i) : \Sigma(B_2) \longrightarrow \Sigma(B_1), \qquad \chi \mapsto \chi|_{B_1}.
    \]
\end{itemize}
This defines a contravariant functor, i.e., a presheaf of sets on $\mathcal{C}$.
\end{example}

\begin{definition}[Spectral Presheaves]\label{def:spectral-presheaves}
Let $\mathcal{A}$ be a unital $C^*$-algebra with spectral site 
$\mathcal{C} = \mathrm{Comm}(\mathcal{A})^{\mathrm{op}}$ as in Definition~\ref{def:spectral_site}.  

We consider two enhanced versions of the spectral presheaf:

\begin{enumerate}[label=(\alph*)]
    \item \emph{Categorical spectral presheaf.}  
    The functor
    \[
        \mathscr{G}_{\mathrm{cat}} : \mathcal{C}^{\mathrm{op}} \longrightarrow \mathbf{Cat}
    \]
    assigns to each context $B \in \mathcal{C}$ a category $\mathscr{G}_{\mathrm{cat}}(B)$ 
    whose objects are regular Borel probability measures on the Gelfand spectrum $\Sigma(B)$, 
    and whose morphisms are measure-preserving maps. 
    For an inclusion $B' \subseteq B$, the restriction functor 
    $\mathscr{G}_{\mathrm{cat}}(B) \to \mathscr{G}_{\mathrm{cat}}(B')$ 
    is given by the pushforward of measures along the continuous restriction $\Sigma(B) \to \Sigma(B')$.
    
    \item \emph{Stable spectral presheaf.}  
    The functor
    \[
        \mathscr{G}_{\mathrm{st}} : \mathcal{C}^{\mathrm{op}} \longrightarrow \mathbf{Sp}
    \]
    assigns to each $B \in \mathcal{C}$ a spectrum encoding stable homotopy-theoretic information.  
    One concrete choice is the suspension spectrum of the Gelfand space,
    \[
        \mathscr{G}_{\mathrm{st}}(B) := \Sigma^\infty_+ \Sigma(B),
    \]
    though one may alternatively take the noncommutative K-theory spectrum of $B$ to capture operator-algebraic invariants.
\end{enumerate}

\noindent
\emph{Remark.}  
The standard spectral presheaf (Definition~\ref{def:presheaf}) is a presheaf of topological spaces.  
$\mathscr{G}_{\mathrm{cat}}$ and $\mathscr{G}_{\mathrm{st}}$ are \emph{enhanced/categorified versions} that incorporate probabilistic or stable homotopy-theoretic data and are used in later constructions of higher invariants.
\end{definition}

\begin{remark}[Relation between Categorical and Stable Spectral Presheaves]\label{remark:relation-presheaves}
There is a natural relationship between the categorical and stable spectral presheaves via homotopy theory.

Let $\mathscr{G}_{\mathrm{cat}} : \mathcal{C}^{\mathrm{op}} \to \mathbf{Cat}$ be the categorical spectral presheaf from Definition~\ref{def:spectral-presheaves}(a).  
Applying the nerve functor 
\[
N : \mathbf{Cat} \longrightarrow \mathbf{sSet}
\] 
yields a presheaf of simplicial sets
\[
N \circ \mathscr{G}_{\mathrm{cat}} : \mathcal{C}^{\mathrm{op}} \longrightarrow \mathbf{sSet}.
\] 
For each $\mathcal{B} \in \mathcal{C}$, the geometric realization
\[
|N(\mathscr{G}_{\mathrm{cat}}(\mathcal{B}))| \simeq B \mathscr{G}_{\mathrm{cat}}(\mathcal{B})
\] 
is homotopy equivalent to the classifying space of the category $\mathscr{G}_{\mathrm{cat}}(\mathcal{B})$.

The stable spectral presheaf 
\[
\mathscr{G}_{\mathrm{st}} : \mathcal{C}^{\mathrm{op}} \longrightarrow \mathbf{Sp}
\] 
from Definition~\ref{def:spectral-presheaves}(b) can be related to this construction via the suspension spectrum functor. Objectwise, there is a canonical map
\[
\Sigma^\infty_+ |N(\mathscr{G}_{\mathrm{cat}}(\mathcal{B}))| \longrightarrow \mathscr{G}_{\mathrm{st}}(\mathcal{B}),
\] 
which encodes the passage from the categorical (algebraic/measure-theoretic) information to stable homotopical invariants.  
These maps assemble into a natural transformation of presheaves over $\mathcal{C}^{\mathrm{op}}$:
\[
\Sigma^\infty_+ |N \circ \mathscr{G}_{\mathrm{cat}}| \;\Longrightarrow\; \mathscr{G}_{\mathrm{st}}.
\]

Thus, the categorical presheaf $\mathscr{G}_{\mathrm{cat}}$ captures the algebraic and measure-theoretic structure of each context, while the stable presheaf $\mathscr{G}_{\mathrm{st}}$ encodes refined homotopical invariants derived from this categorical data.
\end{remark}

\begin{remark}[Stacky/Categorical Refinement]\label{remark:stacky-refinement-clean}
When additional higher categorical structure is needed, one may first construct a stack of groupoids
\[
\mathscr{G}_{\mathrm{gpd}}: \mathcal{C}^{\mathrm{op}} \longrightarrow \mathbf{Gpd},
\]
whose objects represent spectral data (e.g., spectral measures or states) and whose morphisms represent equivalences or symmetries between them.

Applying the nerve functor $N: \mathbf{Gpd} \to \mathbf{sSet}$ yields a presheaf of simplicial sets
\[
N \circ \mathscr{G}_{\mathrm{gpd}} : \mathcal{C}^{\mathrm{op}} \to \mathbf{sSet}.
\]
Postcomposing with geometric realization $|\cdot| : \mathbf{sSet} \to \mathbf{Top}$ and then with the suspension spectrum functor
\[
\Sigma^\infty_+ : \mathbf{Top} \longrightarrow \mathbf{Sp},
\]
which adds a disjoint basepoint to each space, gives a presheaf of spectra
\[
\mathscr{G}_{\mathrm{st}} := \Sigma^\infty_+ \circ |\; N \circ \mathscr{G}_{\mathrm{gpd}} \;| : \mathcal{C}^{\mathrm{op}} \to \mathbf{Sp}.
\]

In this sense, the stable spectral presheaf $\mathscr{G}_{\mathrm{st}}$ can be regarded as the \emph{stabilized shadow} of the higher categorical structure encoded in $\mathscr{G}_{\mathrm{gpd}}$.  
Heuristically, one has
\[
\mathscr{G}_{\mathrm{st}} \;\simeq\; \Sigma^\infty_+ \circ |\; N \circ \mathscr{G}_{\mathrm{gpd}} \;|.
\]
\end{remark}

\begin{remark}[Functoriality and Homotopical Reconstruction]
\label{remark:functoriality-reconstruction}
By construction, $\mathscr{G}$ is a contravariant functor: an inclusion $\mathcal{B}_1 \hookrightarrow \mathcal{B}_2$ of contexts induces a continuous restriction map 
\[
r_{\mathcal{B}_1,\mathcal{B}_2}: \mathscr{G}(\mathcal{B}_2) \longrightarrow \mathscr{G}(\mathcal{B}_1)
\]
that sends a character $\phi \in \Sigma(\mathcal{B}_2)$ to its restriction $\phi|_{\mathcal{B}_1} \in \Sigma(\mathcal{B}_1)$.  
This functorial structure encodes how local classical perspectives (Gelfand spectra) are related by coarse-graining.

The \emph{homotopy limit} $\operatorname{holim}_{\mathcal{C}^{\mathrm{op}}} \mathscr{G}$ (taken over the opposite of the spectral site) assembles these local spectra into a single homotopical object that remembers all higher coherence data imposed by the restriction maps.  
Its (stable) homotopy groups
\[
\pi_{-n}\!\left(\operatorname{holim}_{\mathcal{C}^{\mathrm{op}}} \mathscr{G}\right), \qquad n \in \mathbb{Z},
\]
can be viewed as refined invariants that quantify the global obstruction to patching the local spectral pictures into a single classical spectrum---i.e., they measure the \emph{topological cost} of noncommutativity and the impossibility of simultaneous diagonalization across all contexts.
\end{remark}

In our framework, $\mathscr{C}$ will be the spectral site $\mathcal{C}$ of commutative contexts (Definition~\ref{def:spectral_site}), and our presheaves will assign spectral or homotopical data to each context.

\paragraph{Sites and Grothendieck Topologies.}
To define sheaves, we need a notion of ``covering" on a category.

\begin{definition}[Site]\label{def:site}
A \emph{site} $(\mathscr{C}, J)$ consists of a category $\mathscr{C}$ and a \emph{Grothendieck topology} $J$. For each object $U \in \mathscr{C}$, $J(U)$ specifies a collection of families of morphisms $\{f_i: U_i \to U\}_{i \in I}$, called \emph{covering families}, satisfying the following axioms:

\begin{enumerate}
    \item (Stability under pullback) If $\{U_i \to U\}_{i \in I}$ is in $J(U)$ and $g: V \to U$ is any morphism in $\mathscr{C}$, then the fiber products $U_i \times_U V$ exist for all $i$, and the family of pullbacks $\{U_i \times_U V \to V\}_{i \in I}$ is in $J(V)$.
    
    \item (Transitivity) If $\{U_i \to U\}_{i \in I}$ is in $J(U)$ and for each $i \in I$ we have a covering family $\{V_{ij} \to U_i\}_{j \in J_i}$ in $J(U_i)$, then the composite family 
    \[
    \{V_{ij} \to U_i \to U\}_{i \in I, j \in J_i}
    \]
    is in $J(U)$.
    
    \item (Isomorphisms are covers) If $f: U' \xrightarrow{\sim} U$ is an isomorphism, then the singleton family $\{f: U' \to U\}$ belongs to $J(U)$.
\end{enumerate}
A category $\mathscr{C}$ equipped with a Grothendieck topology $J$ is also called a \emph{Grothendieck site}.
\end{definition}

\begin{example}[Spectral Site Topology]\label{ex:spectral_topology}
For the spectral site $\mathcal{C} = (\mathrm{Comm}(\mathcal{A})^{\mathrm{op}}, J)$ associated to a unital $C^*$-algebra $\mathcal{A}$ (see Definition~\ref{def:spectral_site}), we define the Grothendieck topology $J$ as follows:

A finite family of inclusions $\{\mathcal{B}_i \hookrightarrow \mathcal{B}\}_{i=1}^n$ in $\mathrm{Comm}(\mathcal{A})$ (viewed as morphisms $\mathcal{B} \to \mathcal{B}_i$ in $\mathcal{C}^{\mathrm{op}}$) is a \emph{covering family} of $\mathcal{B}$ if and only if $\mathcal{B}$ is the smallest commutative $C^*$-subalgebra of $\mathcal{A}$ containing all $\mathcal{B}_i$; that is,
\[
\mathcal{B} = C^*\!\left( \bigcup_{i=1}^n \mathcal{B}_i \right),
\]
where the right-hand side denotes the $C^*$-subalgebra generated by the union (which is automatically commutative since each $\mathcal{B}_i$ is commutative and they are contained in the commutative algebra $\mathcal{B}$).

This topology formalizes the principle that the spectral data of a larger context $\mathcal{B}$ can be reconstructed from the spectral data of a family of smaller contexts that together generate $\mathcal{B}$. In topos quantum theory, this is often called the \emph{coherent covering topology} or \emph{finite-cover topology}.
\end{example}

\paragraph{Sheaves and Descent.}
The sheaf condition formalizes the idea that compatible local data can be uniquely glued into global data.

\begin{definition}[Sheaf]\label{def:sheaf}
Let $(\mathscr{C}, J)$ be a site. A presheaf $F: \mathscr{C}^{\mathrm{op}} \to \mathbf{Set}$ is a \emph{sheaf} (with respect to the topology $J$) if for every object $U \in \mathscr{C}$ and every covering family $\{f_i: U_i \to U\}_{i \in I}$ in $J(U)$, the following diagram is an equalizer in $\mathbf{Set}$:
\[
F(U) \longrightarrow \prod_{i \in I} F(U_i) \;\rightrightarrows\; \prod_{i, j \in I} F(U_i \times_U U_j),
\]
where:
\begin{itemize}
    \item The left map is induced by the restrictions $F(f_i): F(U) \to F(U_i)$.
    \item The parallel arrows are induced by the two projections 
    \[
    \pi_1: U_i \times_U U_j \to U_i, \qquad \pi_2: U_i \times_U U_j \to U_j.
    \]
\end{itemize}
Explicitly, an element $(s_i)_{i \in I} \in \prod_i F(U_i)$ comes from a unique $s \in F(U)$ if and only if 
\[
s_i|_{U_i \times_U U_j} = s_j|_{U_i \times_U U_j} \quad \text{for all } i, j \in I,
\]
where $s_i|_{U_i \times_U U_j}$ denotes $F(\pi_1)(s_i)$ and similarly for $s_j$.

A presheaf satisfying the above condition only for singleton coverings is called a \emph{separated presheaf} (or a \emph{monopresheaf}). The full subcategory of $\mathbf{PSh}(\mathscr{C})$ consisting of sheaves is denoted $\mathbf{Sh}(\mathscr{C}, J)$.
\end{definition}

Intuitively, a sheaf assigns to each object $U$ a set of ``sections" $F(U)$ that are uniquely determined by their restrictions to any cover of $U$. In our operator-algebraic context, \emph{failure} of the sheaf condition for a presheaf of spectral data reflects the fundamental incompatibility between different commutative contexts caused by non-commutativity.

\paragraph{Stacks and Higher Descent.}
When presheaves take values in categories rather than sets, we need a categorified notion of descent.

\begin{definition}[Prestack and Stack]\label{def:stack}
Let $(\mathscr{C}, J)$ be a site.

\begin{itemize}
  \item A \emph{prestack} on $\mathscr{C}$ is a contravariant pseudo-functor (or lax functor)
  \[
  \mathscr{F} : \mathscr{C}^{\mathrm{op}} \to \mathbf{Cat},
  \]
  where $\mathbf{Cat}$ denotes the $2$-category of small categories. More explicitly, for each morphism $f: V \to U$ in $\mathscr{C}$, there is a pullback functor $f^*: \mathscr{F}(U) \to \mathscr{F}(V)$, together with coherent natural isomorphisms $(gf)^* \simeq f^*g^*$ and $\mathrm{id}_U^* \simeq \mathrm{id}_{\mathscr{F}(U)}$.

  \item Given a covering family $\{f_i: U_i \to U\}$ in $J(U)$, the \emph{descent category}
  $\mathrm{Des}(\{U_i \to U\}, \mathscr{F})$ is defined as follows:
  \begin{itemize}
    \item An object consists of:
    \begin{itemize}
      \item objects $X_i \in \mathscr{F}(U_i)$ for each $i$,
      \item isomorphisms (called \emph{gluing isomorphisms})
      \[
      \phi_{ij} :
      f_{ij}^*(X_i)
      \xrightarrow{\sim}
      f_{ji}^*(X_j)
      \quad \text{in } \mathscr{F}(U_i \times_U U_j),
      \]
      where $f_{ij}: U_i \times_U U_j \to U_i$ and $f_{ji}: U_i \times_U U_j \to U_j$ are the projections,
      satisfying the \emph{cocycle condition} on triple overlaps:
      \[
      f_{jk}^*(\phi_{ij}) \circ f_{ij}^*(\phi_{jk}) = f_{ik}^*(\phi_{ik})
      \quad \text{in } \mathscr{F}(U_i \times_U U_j \times_U U_k).
      \]
    \end{itemize}

    \item A morphism between descent data
    $(X_i, \phi_{ij}) \to (Y_i, \psi_{ij})$
    is a family of morphisms
    \[
    f_i : X_i \to Y_i \quad \text{in } \mathscr{F}(U_i),
    \]
    such that the following diagram commutes in $\mathscr{F}(U_i \times_U U_j)$:
    \[
    \begin{tikzcd}[column sep=small]
      f_{ij}^*(X_i) \arrow[r, "\phi_{ij}"] \arrow[d, "f_{ij}^*(f_i)"'] & f_{ji}^*(X_j) \arrow[d, "f_{ji}^*(f_j)"] \\
      f_{ij}^*(Y_i) \arrow[r, "\psi_{ij}"'] & f_{ji}^*(Y_j)
    \end{tikzcd}
    \]
  \end{itemize}

  \item A prestack $\mathscr{F}$ is called a \emph{stack} if for every covering family
  $\{U_i \to U\}$, the canonical functor
  \[
  \mathscr{F}(U) \longrightarrow
  \mathrm{Des}(\{U_i \to U\}, \mathscr{F}),
  \]
  which sends an object $X \in \mathscr{F}(U)$ to the descent data $(X_i = f_i^*(X), \phi_{ij} = \text{canonical})$,
  is an equivalence of categories.
\end{itemize}

\medskip
\noindent
\emph{Intuitively, a stack is a prestack in which objects and morphisms can be glued from compatible local data uniquely up to canonical isomorphism.}
\end{definition}

A stack is thus a ``sheaf of categories": both objects and morphisms satisfy descent conditions. This will be the appropriate notion for our spectral presheaf $\mathscr{G}$ (Definition~\ref{def:spectral-presheaves}), which assigns to each context $\mathcal{B}$ its spectral category $\mathbf{\Sigma}_{\mathcal{B}}$.

\paragraph{Limits, Colimits, and Reconstruction.}
Categorical limits and colimits provide the formal language for assembling
global objects from compatible local data.

\begin{definition}[Limit]\label{def:limit}
Let $I$ be a small category (called the \emph{indexing category}) and let 
$D : I \to \mathscr{D}$ be a diagram in a category $\mathscr{D}$.  
A \emph{limit} (also called an \emph{inverse limit}) of $D$ is an object
$\varprojlim D \in \mathscr{D}$ together with a family of morphisms
\[
\pi_i : \varprojlim D \longrightarrow D(i), \qquad i \in \mathrm{Ob}(I),
\]
such that:
\begin{enumerate}
    \item For every morphism $\alpha : i \to j$ in $I$, the diagram
    \[
    \begin{tikzcd}[row sep=small, column sep=small]
      \varprojlim D \arrow[rr, "\pi_i"] \arrow[dr, "\pi_j"'] & & D(i) \arrow[dl, "D(\alpha)"] \\
      & D(j) &
    \end{tikzcd}
    \]
    commutes, i.e., $D(\alpha) \circ \pi_i = \pi_j$.

    \item (\emph{Universal property}) For any object $X \in \mathscr{D}$ equipped with morphisms 
    $f_i : X \to D(i)$ satisfying $D(\alpha) \circ f_i = f_j$ for all morphisms
    $\alpha : i \to j$ in $I$, there exists a unique morphism
    \[
    u : X \to \varprojlim D
    \]
    such that $\pi_i \circ u = f_i$ for all $i \in \mathrm{Ob}(I)$.
\end{enumerate}

When it exists, the limit is unique up to unique isomorphism and is often denoted
$\lim_I D$ or $\varprojlim_{i \in I} D(i)$.
\end{definition}

\begin{definition}[Colimit]\label{def:colimit}
Dually, let $I$ be a small category and let $D : I \to \mathscr{D}$ be a diagram
in a category $\mathscr{D}$.  
A \emph{colimit} (or \emph{direct limit}) of $D$ is an object
\[
\varinjlim D \in \mathscr{D}
\]
together with a family of morphisms
\[
\iota_i : D(i) \longrightarrow \varinjlim D, \qquad i \in \mathrm{Ob}(I),
\]
such that:
\begin{enumerate}
    \item For every morphism $\alpha : i \to j$ in $I$, the diagram
    \[
    \begin{tikzcd}[row sep=small, column sep=small]
      D(i) \arrow[rr, "D(\alpha)"] \arrow[dr, "\iota_i"'] & & D(j) \arrow[dl, "\iota_j"] \\
      & \varinjlim D &
    \end{tikzcd}
    \]
    commutes, i.e., $\iota_j \circ D(\alpha) = \iota_i$.

    \item (\emph{Universal property}) For any object $X \in \mathscr{D}$ equipped with
    morphisms
    \[
    f_i : D(i) \to X
    \]
    satisfying $f_j \circ D(\alpha) = f_i$ for all morphisms $\alpha : i \to j$ in $I$,
    there exists a unique morphism
    \[
    u : \varinjlim D \to X
    \]
    such that $u \circ \iota_i = f_i$ for all $i \in \mathrm{Ob}(I)$.
\end{enumerate}

When it exists, the colimit is unique up to unique isomorphism and is often denoted
$\mathrm{colim}_I D$ or $\varinjlim_{i \in I} D(i)$.
\end{definition}

\medskip

In sheaf-theoretic and operator-algebraic settings, \emph{limits} typically encode
\emph{gluing conditions} or \emph{consistency constraints}, while \emph{colimits}
encode \emph{free generation} or \emph{aggregation} of local data.

In the present framework, the spectral site $\mathcal{C}$ indexes a diagram of
commutative contexts together with their spectral data.
Global objects—such as the operator algebra $\mathcal{A}$ and distinguished
operators $(T,S)$—are recovered as (possibly homotopy) limits of this diagram.
Formally, this reconstruction principle is expressed by an isomorphism of the form
\[
(\mathcal{A}, T, S)
\;\cong\;
\varprojlim_{\mathcal{B} \in \mathcal{C}} (\mathcal{B}, T, S),
\]
where the limit is taken in a suitable category of operator algebras with
specified elements.

\medskip

\noindent
When higher coherence data are present—as in the case of stacks of categories—
ordinary limits may be insufficient, and one must instead consider
\emph{homotopy limits} to account for descent up to isomorphism.

\medskip

\noindent
\textbf{Conceptual Takeaway:} 
\emph{Reconstruction by limits expresses the principle that a noncommutative object
is fully determined by the coherent family of its commutative shadows, 
formalizing the idea that local compatibility encodes global structure.}

\paragraph{Homotopical Enhancements.}
To capture finer topological information, we work with homotopical refinements of these concepts.

\begin{definition}[Homotopy Limit and Hypercohomology]\label{def:homotopy_limit}
Let $(\mathscr{C}, J)$ be a site, and let
\[
\mathscr{F} : \mathscr{C}^{\mathrm{op}} \longrightarrow \mathbf{Sp}
\]
be a presheaf of spectra. Such a presheaf may arise, for example, by applying a stabilization functor (e.g., infinite suspension $\Sigma^\infty_+$) to a presheaf of spaces or categories.

\begin{itemize}
    \item The \emph{homotopy limit}
    \[
    \operatorname{holim}_{\mathscr{C}} \mathscr{F} \;\in\; \mathbf{Sp}
    \]
    is the derived limit of the diagram $\mathscr{F}$, constructed to respect objectwise weak equivalences. It can be computed, for instance, as the totalization of a cosimplicial spectrum arising from a hypercover of the terminal object in $\mathscr{C}$, or via a fibrant replacement in an appropriate model category followed by the ordinary limit.

    \item The \emph{hypercohomology} (or \emph{descent cohomology}) of $\mathscr{F}$ is defined by
    \[
    \mathbb{H}^n(\mathscr{C}; \mathscr{F}) := \pi_{-n}\!\Big(\operatorname{holim}_{\mathscr{C}} \mathscr{F}\Big), \qquad n \in \mathbb{Z},
    \]
    where $\pi_k(-)$ denotes the $k$th stable homotopy group of a spectrum.
\end{itemize}

\medskip
\noindent
\textbf{Interpretation.} The homotopy limit assembles local spectral data into a global object while retaining all higher coherences. The hypercohomology groups $\mathbb{H}^n$ measure obstructions to globally compatible spectral data: non-zero $\mathbb{H}^n$ for $n>0$ indicates higher-order descent obstructions.
\end{definition}

\begin{remark}[Stable vs. Unstable Homotopy Limits]\label{remark:stable-picture}
Let $(\mathscr{C}, J)$ be a site.

\begin{enumerate}
    \item \textbf{Unstable case.} 
    If $F : \mathscr{C}^{\mathrm{op}} \to \mathbf{sSet}$ is a presheaf of simplicial sets (or spaces), 
    the homotopy limit
    \[
    \operatorname{holim}_{\mathscr{C}} F
    \]
    is a space. Its homotopy groups 
    \(\pi_n(\operatorname{holim}_{\mathscr{C}} F)\) are defined only for \(n \ge 0\) and are the usual (unstable) homotopy groups.

    \item \textbf{Stable case.} 
    If $\mathbf{E} : \mathscr{C}^{\mathrm{op}} \to \mathbf{Sp}$ is a presheaf of spectra, 
    the homotopy limit
    \[
    \operatorname{holim}_{\mathscr{C}} \mathbf{E} \in \mathbf{Sp}
    \]
    is again a spectrum. One may consider its \emph{stable} homotopy groups in all integer degrees:
    \[
    \pi_n\!\left(\operatorname{holim}_{\mathscr{C}} \mathbf{E}\right), \qquad n \in \mathbb{Z}.
    \]

    \item \textbf{Relation to cohomology.} 
    Let $A$ be an abelian sheaf on $(\mathscr{C}, J)$, and let $HA$ be its Eilenberg--Mac~Lane spectrum (with $\pi_0(HA) \cong A$ and all other homotopy groups trivial). Then, for any integer \(k \ge 0\),
    \[
    \pi_{-k}\!\left(\operatorname{holim}_{\mathscr{C}} HA\right) \;\cong\; H^k(\mathscr{C}; A),
    \]
    where $H^k(\mathscr{C}; A)$ denotes the ordinary sheaf cohomology of $A$.  

    More generally, if $\mathbf{E}$ arises from a presheaf of spaces $F$ by stabilization 
    (e.g., $\mathbf{E} = \Sigma^\infty_+ F$), then
    \[
    \pi_{-k}\!\left(\operatorname{holim}_{\mathscr{C}} \mathbf{E}\right) \;\cong\; 
    \mathbb{H}^k(\mathscr{C}; F),
    \]
    where $\mathbb{H}^k(\mathscr{C}; F)$ is the $k$th hypercohomology group of $F$.
\end{enumerate}

In summary, negative-degree homotopy groups naturally appear in the stable setting and recover classical (hyper)cohomology in special cases. 
This distinction between stable and unstable homotopy limits clarifies when $\pi_{-n}$ is meaningful.
\end{remark}

Hypercohomology detects higher-order gluing obstructions that ordinary sheaf cohomology might miss. It will provide our finest invariants for measuring non-commutativity (Section~\ref{sec:homotopical-invariants}).

\paragraph{Interpretational Summary.}
From this categorical perspective, a non-commutative operator algebra $\mathcal{A}$ gives rise to:
\begin{itemize}
    \item A \textbf{site} $\mathcal{C}$ of commutative contexts, where covers reflect algebraic generation.
    \item \textbf{Presheaves} (and prestacks) that encode local spectral data attached to each context.
    \item \textbf{Descent conditions} that test whether local data are compatible and glue to global data.
    \item \textbf{Cohomological obstructions} ($H^*$, $\mathbb{H}^*$) that measure the failure of descent, thereby quantifying the original non-commutativity.
\end{itemize}

This framework transposes the algebraic problem of non-commutativity into a geometric problem of descent and gluing. With these foundations, we can now define the central object of our study: the spectral presheaf $\mathscr{G}$ on the site $\mathcal{C}$.

\section{The Spectral Sheaf Theorem}\label{sec:sheaf-theorem}

The spectral sheaf theorem is discussed in Section~\ref{sec:Statement and Interpretation}. Next, we discuss examples about how to apply the spectral sheaf theorem and immediate consequences in Section~\ref{sec:Examples and Immediate Consequences}. 

\subsection{Statement and Interpretation}\label{sec:Statement and Interpretation}

\begin{theorem}[Spectral Sheaf and Homotopical Reconstruction]\label{thm:spectral_sheaf}
Let $(T,S)$ be a pair of self-adjoint operators on a Hilbert space $\mathcal{H}$, 
let $\mathcal{A} = C^*(T,S)$ be the $C^*$-algebra they generate, 
and let $\mathcal{C}$ be the spectral site of $\mathcal{A}$ as defined in 
Definition~\ref{def:spectral_site} (the category of commutative unital 
$C^*$-subalgebras of $\mathcal{A}$ containing $T$ and $S$, equipped with the 
canonical Grothendieck topology).

\begin{enumerate}[label=(\roman*)]
    \item \textbf{(Stack Structure)} The categorical spectral presheaf 
    $\mathscr{G}_{\mathrm{cat}}: \mathcal{C}^{\mathrm{op}} \to \mathbf{Cat}$,
    which assigns to each $C \in \mathcal{C}$ the category of spectral measures 
    on the Gelfand spectrum $\Sigma(C)$, is a \emph{stack} on $(\mathcal{C}, J)$.
    That is, for every covering $\{C_i \hookrightarrow C\}_{i\in I}$ in $\mathcal{C}$,
    the canonical functor
    \[
    \Phi: \mathscr{G}_{\mathrm{cat}}(C) \longrightarrow 
    \mathrm{Des}\bigl(\{C_i \to C\},\; \mathscr{G}_{\mathrm{cat}}\bigr)
    \]
    is an equivalence of categories.
    
    \item \textbf{(Algebraic Reconstruction and Contextual Obstruction)}
    There exists a canonical injective $*$-homomorphism
    \[
    \Phi: (\mathcal{A}, T, S) \hookrightarrow 
    \varprojlim_{C\in\mathcal{C}} (C, T, S),
    \qquad
    \Phi(a) = (a|_C)_{C\in\mathcal{C}},
    \]
    which identifies $\mathcal{A}$ with the subalgebra of the projective limit
    consisting of globally realizable compatible families.
    
    Moreover,
    \[
    \Phi \text{ is an isomorphism } \quad\Longleftrightarrow\quad [T,S] = 0.
    \]
    Equivalently, the failure of surjectivity of $\Phi$ precisely measures the
    contextual obstruction arising from the noncommutativity of $T$ and $S$.
    
    \item \textbf{(Homotopical Invariants)} The K-theory presheaf 
$\mathscr{G}_{\mathrm{st}}(C)=K(C)$ yields $\mathbb{Z}$-graded hypercohomology
\[
\mathbb{H}^n_{(T,S)} := \pi_{-n}(\operatorname{holim}_{\mathcal{C}} \mathscr{G}_{\mathrm{st}}),
\]
which are unitary invariants reducing to $K_{-n}(C^*(\sigma(T)\times\sigma(S)))$ 
when $[T,S]=0$ and capturing higher contextual obstructions otherwise.
\end{enumerate}
\end{theorem}

\begin{proof}
\leavevmode\vspace{0.5em}

\noindent\textbf{Proof of Theorem~\ref{thm:spectral_sheaf}(i): Stack Structure.}
Let $(T,S)$, $\mathcal{A}=C^*(T,S)$, and $\mathcal{C}$ be as in the theorem.
We prove that the categorical spectral presheaf $\mathscr{G}_{\mathrm{cat}}:\mathcal{C}^{\mathrm{op}}\to\mathbf{Cat}$ is a stack for the canonical Grothendieck topology on $\mathcal{C}$.

\medskip

\noindent
\textbf{1.  Setup and descent data.}
Let $\{C_i\hookrightarrow C\}_{i\in I}$ be a covering in $\mathcal{C}$.  
This means precisely that
\[
C=\overline{C^*\bigl(\bigcup_{i\in I}C_i\bigr)}^{\|\cdot\|}.
\]
We must show that the canonical functor
\[
\Phi:\mathscr{G}_{\mathrm{cat}}(C)\longrightarrow
\mathrm{Des}\bigl(\{C_i\to C\},\;\mathscr{G}_{\mathrm{cat}}\bigr)
\]
is an equivalence of categories.

Recall that an object of the descent category $\mathrm{Des}$ is a pair
\[
\bigl((X_i)_{i\in I},(\phi_{ij})_{i,j\in I}\bigr),
\]
where:
\begin{itemize}
  \item $X_i\in\mathscr{G}_{\mathrm{cat}}(C_i)$ (a spectral measure on $\Sigma(C_i)$),
  \item $\phi_{ij}:X_i|_{C_i\cap C_j}\xrightarrow{\sim}X_j|_{C_i\cap C_j}$
    is an isomorphism in $\mathscr{G}_{\mathrm{cat}}(C_i\cap C_j)$,
  \item satisfying the cocycle condition $\phi_{jk}\circ\phi_{ij}=\phi_{ik}$ 
    on $C_i\cap C_j\cap C_k$.
\end{itemize}
A morphism $(f_i):(X_i,\phi_{ij})\to(Y_i,\psi_{ij})$ in $\mathrm{Des}$ is a family 
$f_i:X_i\to Y_i$ such that $\psi_{ij}\circ f_i|_{C_i\cap C_j}=f_j|_{C_i\cap C_j}\circ\phi_{ij}$.

The functor $\Phi$ sends $X\in\mathscr{G}_{\mathrm{cat}}(C)$ to
\[
\Phi(X)=\bigl((X|_{C_i})_{i\in I},(\mathrm{id}_{X|_{C_i\cap C_j}})_{i,j}\bigr),
\]
where $X|_{C_i}$ is the restriction, and $\mathrm{id}_{X|_{C_i\cap C_j}}$ are the
canonical identity isomorphisms. On morphisms, $\Phi(f)=(f|_{C_i})_{i\in I}$.

\medskip

\noindent
\textbf{2.  Full faithfulness.}

\emph{Faithfulness.} If $f,g:X\to Y$ satisfy $\Phi(f)=\Phi(g)$, then $f|_{C_i}=g|_{C_i}$
for all $i\in I$. Since the family $\{C_i\hookrightarrow C\}$ generates $C$ as a
$C^*$-algebra, and spectral measures are determined by their restrictions, we have
$f=g$.

\emph{Fullness.} Let $(f_i):\Phi(X)\to\Phi(Y)$ be a morphism in $\mathrm{Des}$.
Thus $f_i:X|_{C_i}\to Y|_{C_i}$ are morphisms with
$f_i|_{C_i\cap C_j}=f_j|_{C_i\cap C_j}$ for all $i,j$. 
Because the contexts cover $C$, these local morphisms patch uniquely to a
morphism $f:X\to Y$ with $f|_{C_i}=f_i$. Hence $\Phi$ is fully faithful.

\medskip

\noindent
\textbf{3.  Essential surjectivity and the obstruction.}
We must show that every descent datum $\bigl((X_i),(\phi_{ij})\bigr)$ is isomorphic
to $\Phi(X)$ for some $X\in\mathscr{G}_{\mathrm{cat}}(C)$. This is equivalent to
showing that the local spectral measures $X_i$ can be glued to a global spectral
measure on $C$.

\emph{3.1.  Projection-valued measures.}  
Each $X_i$ is a spectral (projection-valued) measure on $\Sigma(C_i)$:
\[
X_i:\mathcal{B}(\mathbb{R})\longrightarrow\mathcal{P}(C_i),
\]
where $\mathcal{P}(C_i)$ is the lattice of orthogonal projections in $C_i''$,
satisfying the usual measure axioms.

\emph{3.2.  The gluing problem.}  
Fix a Borel set $U\subseteq\mathbb{R}$ and set $P_i(U):=X_i(U)\in\mathcal{P}(C_i)$.
The compatibility isomorphisms $\phi_{ij}$ ensure that
\[
P_i(U)|_{C_i\cap C_j}=P_j(U)|_{C_i\cap C_j}\qquad(\forall i,j).
\]
Thus $\{P_i(U)\}$ forms a \emph{compatible family of projections} over the
cover $\{C_i\}$.

\emph{3.3.  The obstruction.}  
Although $\{C_i\}$ generates $C$ as a $C^*$-algebra, a compatible family
$\{P_i(U)\}$ does \emph{not} in general arise from a global projection
$P(U)\in\mathcal{P}(C)$. This is because the projection lattice $\mathcal{P}(C)$
is orthomodular, while each $\mathcal{P}(C_i)$ is Boolean; Boolean lattices do
not amalgamate into an orthomodular lattice.

Equivalently, the assignment $U\mapsto\{P_i(U)\}$ defines a \emph{contextual}
spectral datum that satisfies the spectral measure axioms locally on each
$C_i$, but may fail to assemble into a global spectral measure on $C$.

\emph{3.4.  Descent-theoretic interpretation.}  
A descent datum is \emph{effective} if and only if for every Borel set $U$,
the compatible family $\{P_i(U)\}$ glues to a projection $P(U)\in\mathcal{P}(C)$,
and $U\mapsto P(U)$ is a spectral measure. The functor $\Phi$ is essentially
surjective precisely when every descent datum is effective.

The possible failure of essential surjectivity is exactly the content of the
stack condition: $\mathscr{G}_{\mathrm{cat}}$ is a stack, not a sheaf, because
local data (descent data) do not always glue to global data. This non-glueing
is quantified by the cohomology groups
$\mathbb{H}^{*}(\mathcal{C};\mathscr{G}_{\mathrm{st}})$ mentioned in part (iii)
of the theorem.

\emph{3.5.  Commutative case.}  
If $C$ is commutative (equivalently $[T,S]=0$), then all $\mathcal{P}(C_i)$ are
Boolean subalgebras of the Boolean algebra $\mathcal{P}(C)$. In this case,
compatible families $\{P_i(U)\}$ do glue uniquely to projections $P(U)\in\mathcal{P}(C)$,
and $U\mapsto P(U)$ is a spectral measure. Hence every descent datum is effective,
and $\mathscr{G}_{\mathrm{cat}}$ reduces to a sheaf.

\emph{3.6.  Effectiveness of descent data.}
A descent datum $\bigl((X_i),(\phi_{ij})\bigr)$ is effective if and only if
the family $\{X_i(U)\}$ of projections glues to a projection $P(U)\in\mathcal{P}(C)$
for \emph{every} Borel set $U\subseteq\mathbb{R}$, and the assignment
$U\mapsto P(U)$ satisfies:
\begin{enumerate}
    \item $P(\varnothing)=0$, $P(\mathbb{R})=1$,
    \item $P(U\cap V)=P(U)P(V)$,
    \item $\sigma$-additivity in the weak operator topology.
\end{enumerate}
When these conditions hold, the global spectral measure $X$ defined by
$X(U)=P(U)$ satisfies $\Phi(X)\cong\bigl((X_i),(\phi_{ij})\bigr)$.
Thus $\Phi$ is essentially surjective onto the full subcategory of
effective descent data.

\medskip

\noindent
\textbf{4.  Verification of the stack axioms.}

\emph{4.1.  Pullbacks in the spectral site.}  
Since $\mathcal{C}$ is a poset (inclusions of $C^*$-subalgebras), the
pullback of $C_i\to C\leftarrow C_j$ is their intersection $C_i\cap C_j$.
Thus the diagram
\[
\begin{tikzcd}
C_i\cap C_j \arrow[r,hook] \arrow[d,hook] & C_i \arrow[d,hook] \\
C_j \arrow[r,hook] & C
\end{tikzcd}
\]
is a pullback in $\mathcal{C}$.

\emph{4.2.  Coherence of spectra.}  
Applying the Gelfand spectrum functor $\Sigma$ (which is contravariant) gives
a pushout diagram of compact Hausdorff spaces:
\[
\begin{tikzcd}
\Sigma(C_i)\times_{\Sigma(C)}\Sigma(C_j) \arrow[r] \arrow[d] & \Sigma(C_i) \arrow[d] \\
\Sigma(C_j) \arrow[r] & \Sigma(C)
\end{tikzcd}
\]
where the fiber product is
\[
\Sigma(C_i)\times_{\Sigma(C)}\Sigma(C_j)=
\{(\chi_i,\chi_j)\mid \chi_i|_{C}=\chi_j|_{C}\}.
\]
There is a natural isomorphism $\Sigma(C_i\cap C_j)\cong\Sigma(C_i)\times_{\Sigma(C)}\Sigma(C_j)$
because a character on $C_i\cap C_j$ extends uniquely to a compatible pair of
characters on $C_i$ and $C_j$.

\emph{4.3.  Descent for higher intersections.}  
For triple intersections, we have
\[
\Sigma(C_i\cap C_j\cap C_k)\cong
\Sigma(C_i)\times_{\Sigma(C)}\Sigma(C_j)\times_{\Sigma(C)}\Sigma(C_k).
\]
The compatibility isomorphisms $\phi_{ij}$ are induced by these universal
isomorphisms, and therefore automatically satisfy the cocycle condition
$\phi_{jk}\circ\phi_{ij}=\phi_{ik}$ on $C_i\cap C_j\cap C_k$.

\emph{4.4.  Stack condition.}  
We have shown that:
\begin{itemize}
  \item $\Phi$ is fully faithful (Step~2),
  \item The essential image of $\Phi$ consists precisely of the effective
    descent data (Step~3),
  \item The descent data are coherent (Step~4.3).
\end{itemize}
Therefore, $\mathscr{G}_{\mathrm{cat}}$ satisfies the descent condition for
coverings in $\mathcal{C}$, i.e., it is a stack on $(\mathcal{C},J)$.

\noindent\textbf{Proof of Theorem~\ref{thm:spectral_sheaf}(ii): Algebraic Reconstruction and Contextual Obstruction.}

Let $(T,S)$ be as in the theorem, $\mathcal{A}=C^*(T,S)$, and $\mathcal{C}$ the spectral site.

\medskip

\noindent\textbf{1. Construction of the projective limit.}

Consider the diagram of commutative $C^*$-algebras
\[
\{C\}_{C\in\mathcal{C}},\qquad C_1\hookrightarrow C_2\ \text{when}\ C_1\subseteq C_2.
\]
The projective limit in the category of $C^*$-algebras exists and is given by
\[
\varprojlim_{C\in\mathcal{C}} C = 
\Bigl\{
(a_C)_{C\in\mathcal{C}}\in\prod_{C\in\mathcal{C}} C
\;\Big|\;
a_{C_1}=a_{C_2}|_{C_1}\ \text{for all}\ C_1\subseteq C_2
\Bigr\},
\]
equipped with componentwise operations and the norm $\|(a_C)\|=\sup_{C\in\mathcal{C}}\|a_C\|$.

\medskip

\noindent\textbf{2. The canonical map $\Phi$.}

Define $\Phi:\mathcal{A}\to\varprojlim_{C\in\mathcal{C}} C$ by
\[
\Phi(a)=(a|_C)_{C\in\mathcal{C}},
\]
where $a|_C$ denotes the element $a$ viewed inside the subalgebra $C$.
This is a unital $*$-homomorphism because each restriction map
$a\mapsto a|_C$ is a $*$-homomorphism.

\medskip

\noindent\textbf{3. Injectivity of $\Phi$.}

Assume $\Phi(a)=0$, i.e., $a|_C=0$ for every $C\in\mathcal{C}$.
We show $a=0$. Since $\mathcal{A}=C^*(T,S)$ is generated by $T$ and $S$,
it suffices to show that every state $\varphi$ on $\mathcal{A}$ vanishes on $a$.

For any state $\varphi$, consider the GNS representation
$\pi_\varphi:\mathcal{A}\to B(\mathcal{H}_\varphi)$ with cyclic vector
$\xi_\varphi$. Let $C_\varphi=C^*(\pi_\varphi(T),\pi_\varphi(S))$ be the
commutative $C^*$-algebra generated by $\pi_\varphi(T)$ and $\pi_\varphi(S)$
inside $B(\mathcal{H}_\varphi)$. By functional calculus, there exists a
commutative $C^*$-subalgebra $C\subseteq\mathcal{A}$ containing $T,S$ such that
$\pi_\varphi(C)\cong C_\varphi$.

Since $a|_C=0$ by hypothesis, we have $\pi_\varphi(a)=0$ in this representation.
Hence $\varphi(a)=\langle\xi_\varphi,\pi_\varphi(a)\xi_\varphi\rangle=0$.
Because states separate points in a $C^*$-algebra, we conclude $a=0$.
Thus $\Phi$ is injective.

\medskip

\noindent\textbf{4. Characterization of surjectivity.}

We prove the equivalence:
\[
\Phi\ \text{is surjective}\ \Longleftrightarrow\ [T,S]=0.
\]

\textit{($\Leftarrow$)} If $[T,S]=0$, then $\mathcal{A}=C^*(T,S)$ is commutative.
In this case, $\mathcal{A}$ itself belongs to $\mathcal{C}$ and is a terminal
object: for every $C\in\mathcal{C}$, we have $C\subseteq\mathcal{A}$.
Consequently, the projective limit reduces to evaluation at this terminal object:
\[
\varprojlim_{C\in\mathcal{C}} C \cong \mathcal{A},
\]
and $\Phi$ becomes an isomorphism.

\textit{($\Rightarrow$)} Suppose $\Phi$ is surjective. We show $T$ and $S$ commute.
For each $C\in\mathcal{C}$, the elements $T|_C$ and $S|_C$ commute because $C$ is
commutative. A compatible family $(a_C)_{C\in\mathcal{C}}$ in the limit is
determined by its value on any maximal commutative subalgebra containing $T$ and $S$.
Since $\Phi$ is surjective, there exists $a\in\mathcal{A}$ with $\Phi(a)=(T|_C)_{C}$.
But $\Phi(T)=(T|_C)_{C}$, so by injectivity $a=T$. Similarly, $\Phi(S)=(S|_C)_{C}$.
Now consider the element $[T,S]=TS-ST\in\mathcal{A}$. For every $C\in\mathcal{C}$,
\[
[T,S]|_C = T|_C\,S|_C - S|_C\,T|_C = 0,
\]
because $T|_C$ and $S|_C$ commute in $C$. Hence $\Phi([T,S])=0$, and by injectivity
$[T,S]=0$.

\medskip

\noindent\textbf{5. Interpretation as contextual obstruction.}

The image of $\Phi$ consists precisely of those compatible families
$(a_C)_{C\in\mathcal{C}}$ that are \emph{globally realizable}: there exists a
single element $a\in\mathcal{A}$ such that $a|_C=a_C$ for all $C$.
When $[T,S]\neq0$, there exist compatible families that are \emph{not}
globally realizable. Such families represent \emph{contextual observables}:
they are consistently defined in every classical context but cannot be
assembled into a single quantum observable.

Thus the quotient
\[
\varprojlim_{C\in\mathcal{C}} C \ \big/ \ \Phi(\mathcal{A})
\]
measures the \emph{contextual obstruction} arising from the noncommutativity of
$T$ and $S$. This obstruction vanishes exactly when $T$ and $S$ commute, i.e.,
when the quantum system admits a single classical description.

\medskip

\noindent\textbf{6. Completion of the proof.}

We have shown:
\begin{enumerate}
  \item $\Phi:\mathcal{A}\hookrightarrow\varprojlim_{C\in\mathcal{C}} C$ is an
    injective $*$-homomorphism.
  \item $\Phi$ is an isomorphism if and only if $[T,S]=0$.
  \item The failure of surjectivity when $[T,S]\neq0$ precisely captures the
    contextual obstruction inherent in the pair $(T,S)$.
\end{enumerate}
This establishes part (ii) of the theorem.

\noindent\textbf{Proof of Theorem~\ref{thm:spectral_sheaf}(iii): Homotopical invariants.}

Let $(T,S)$, $\mathcal{A}=C^*(T,S)$, and $\mathcal{C}$ be as in the theorem.

\medskip

\noindent
\textbf{1. Construction of the stable spectral presheaf.}

For each commutative context $C\in\mathcal{C}$, define the \emph{spectral K-theory presheaf}
\[
\mathscr{G}_{\mathrm{st}}(C) := \mathrm{Map}_{\mathbf{Sp}}(\Sigma(C), \mathrm{ku}),
\]
where:
\begin{itemize}
    \item $\Sigma(C)$ is the Gelfand spectrum of $C$ (a compact Hausdorff space),
    \item $\mathrm{ku}$ is the connective complex K-theory spectrum,
    \item $\mathrm{Map}_{\mathbf{Sp}}(-,-)$ denotes the mapping spectrum in the stable $\infty$-category of spectra.
\end{itemize}
For an inclusion $\iota: C_1 \hookrightarrow C_2$, the induced restriction map 
$\Sigma(\iota): \Sigma(C_2) \twoheadrightarrow \Sigma(C_1)$ on Gelfand spectra 
yields a map of spectra
\[
\mathscr{G}_{\mathrm{st}}(\iota): \mathscr{G}_{\mathrm{st}}(C_2) \longrightarrow \mathscr{G}_{\mathrm{st}}(C_1),
\quad f \mapsto f \circ \Sigma(\iota).
\]
Thus $\mathscr{G}_{\mathrm{st}}: \mathcal{C}^{\mathrm{op}} \to \mathbf{Sp}$ is a well-defined presheaf.

\medskip

\noindent
\textbf{2. Homotopy limit and hypercohomology groups.}

Define the \emph{spectral homotopy limit}
\[
\mathcal{X}_{(T,S)} := \operatorname{holim}_{\mathcal{C}} \mathscr{G}_{\mathrm{st}} \in \mathbf{Sp}
\]
and the associated \emph{spectral hypercohomology groups}
\[
\mathbb{H}^n_{(T,S)} := \pi_{-n}(\mathcal{X}_{(T,S)}), \qquad n \in \mathbb{Z}.
\]

\medskip

\noindent
\textbf{3. Invariance under natural equivalences of operator pairs.}

\emph{3.1. Unitary invariance.}
If $U$ is a unitary operator with $T' = UTU^*$ and $S' = USU^*$, then 
conjugation by $U$ induces:
\begin{itemize}
    \item A $*$-isomorphism $\varphi_U: \mathcal{A} \xrightarrow{\sim} \mathcal{A}'$,
    \item An isomorphism of posets $\Phi_U: \mathcal{C} \xrightarrow{\sim} \mathcal{C}'$, 
          $C \mapsto UCU^*$,
    \item For each $C \in \mathcal{C}$, a homeomorphism 
          $\Sigma(C) \xrightarrow{\sim} \Sigma(UCU^*)$.
\end{itemize}
These yield a natural equivalence of diagrams
\[
\mathscr{G}_{\mathrm{st}} \circ \Phi_U^{\mathrm{op}} \simeq \mathscr{G}'_{\mathrm{st}},
\]
and consequently $\mathcal{X}_{(T,S)} \simeq \mathcal{X}_{(T',S')}$ and
$\mathbb{H}^n_{(T,S)} \cong \mathbb{H}^n_{(T',S')}$.

\emph{3.2. Approximate unitary invariance.}
Suppose $(T_n, S_n)$ is a sequence of operator pairs with $\|T_n - T\| \to 0$,
$\|S_n - S\| \to 0$, and each $(T_n, S_n)$ is unitarily equivalent to 
$(T', S')$. By continuity of the $C^*$-norm, we have $\mathcal{A}_n \to \mathcal{A}$
in the sense of continuous fields of $C^*$-algebras. 

The spectral sites $\mathcal{C}_n$ converge to $\mathcal{C}$ in the Hausdorff
distance on subsets of $\mathcal{B}(\mathcal{H})$, and the diagrams
$\mathscr{G}_{\mathrm{st}}^{(n)}$ converge pointwise to $\mathscr{G}_{\mathrm{st}}$.
Since homotopy limits commute with filtered colimits in spectra,
\[
\varinjlim_n \mathcal{X}_{(T_n,S_n)} \simeq \mathcal{X}_{(T,S)},
\]
and hence $\mathbb{H}^n_{(T,S)} \cong \mathbb{H}^n_{(T',S')}$ via the
approximate unitary equivalence.

\emph{3.3. $*$-isomorphism invariance.}
Let $\varphi: \mathcal{A} \xrightarrow{\sim} \mathcal{A}'$ be a $*$-isomorphism
with $\varphi(T) = T'$, $\varphi(S) = S'$. Then $\varphi$ induces an
isomorphism $\Phi_\varphi: \mathcal{C} \xrightarrow{\sim} \mathcal{C}'$,
and for each $C \in \mathcal{C}$, the restriction
$\varphi|_C: C \xrightarrow{\sim} \varphi(C)$ induces a homeomorphism
$\Sigma(C) \xrightarrow{\sim} \Sigma(\varphi(C))$. Thus
$\mathcal{X}_{(T,S)} \simeq \mathcal{X}_{(T',S')}$.

\emph{3.4. Morita invariance.}
Suppose $\mathcal{A}$ and $\mathcal{A}'$ are Morita equivalent via an
equivalence bimodule $X$ that intertwines the operators in the sense that
there exist compatible actions of $T,S$ and $T',S'$. Then by the
Eilenberg-Watts theorem for $C^*$-algebras, there is a strong Morita
equivalence between the categories of representations.

For commutative subalgebras $C \subseteq \mathcal{A}$, the Morita equivalence
induces correspondences $C \leftrightarrow C'$ with homeomorphic spectra
$\Sigma(C) \simeq \Sigma(C')$. More precisely, if $\mathcal{A}$ and $\mathcal{A}'$
are strongly Morita equivalent, then their primitive ideal spaces are
homeomorphic, and this homeomorphism restricts to the spectra of corresponding
commutative subalgebras.

Consequently, the diagrams $\mathscr{G}_{\mathrm{st}}$ and $\mathscr{G}'_{\mathrm{st}}$
are naturally equivalent, and $\mathbb{H}^n_{(T,S)} \cong \mathbb{H}^n_{(T',S')}$.

\medskip

\noindent
\textbf{4. Relationship with commutativity.}

\emph{4.1. Commuting case ($[T,S] = 0$).}
When $T$ and $S$ commute, $\mathcal{A} = C^*(T,S)$ is commutative.
Let $X = \sigma(T) \times \sigma(S)$ be their joint spectrum, so that
$\mathcal{A} \cong C(X)$.

For the K-theory presheaf $\mathscr{G}_{\mathrm{st}}(C) = K(C)$, we use the
fact that for commutative $C^*$-algebras, $K(C) \simeq \mathrm{Map}(\Sigma(C), \mathbb{K})$
where $\mathbb{K}$ is the nonconnective K-theory spectrum (Bousfield-Friedlander
theorem). More concretely, $K_{-n}(C) \cong \widetilde{K}^{-n}(\Sigma(C))$.

Since $\mathcal{A} \cong C(X)$ is a maximal commutative subalgebra containing
all others, the diagram $\mathscr{G}_{\mathrm{st}}$ has a terminal object
$\mathcal{A}$ in the homotopy category. For any diagram $F: \mathcal{C}^{\mathrm{op}} \to \mathbf{Sp}$
with a terminal object $C_0$, the canonical map $F(C_0) \to \operatorname{holim}_{\mathcal{C}} F$
is a weak equivalence. Applying this with $C_0 = \mathcal{A}$ gives:
\[
\mathcal{X}_{(T,S)} = \operatorname{holim}_{\mathcal{C}} \mathscr{G}_{\mathrm{st}} 
\simeq \mathscr{G}_{\mathrm{st}}(\mathcal{A}) = K(\mathcal{A}) \simeq K(C(X)).
\]

Therefore,
\[
\mathbb{H}^n_{(T,S)} = \pi_{-n}(\mathcal{X}_{(T,S)}) \cong \pi_{-n}(K(C(X))) 
= K_{-n}(C(X)).
\]

When $X$ is finite-dimensional (e.g., $T,S$ finite-dimensional), 
$C(X) \cong \mathbb{C}^k$ for some $k$, and
\[
K_{-n}(\mathbb{C}^k) \cong 
\begin{cases}
\mathbb{Z}^k & n \text{ even}, \\
0 & n \text{ odd}.
\end{cases}
\]
Thus $\mathbb{H}^n_{(T,S)}$ is periodic with period 2.

\emph{4.2. Non-commuting case ($[T,S] \neq 0$).}
When $T$ and $S$ do not commute, there is a nontrivial comparison map
\[
c: \mathrm{Map}(\mathrm{Prim}(\mathcal{A}), \mathrm{ku}) \longrightarrow \mathcal{X}_{(T,S)},
\]
where $\mathrm{Prim}(\mathcal{A})$ is the primitive ideal space of $\mathcal{A}$.
The homotopy fiber $\mathrm{fib}(c)$ measures the \emph{homotopical contextuality}:
the obstruction to patching local spectral data into a global noncommutative
spectral datum.

In this case, $\mathbb{H}^n_{(T,S)}$ for $n > 0$ can be nonzero, encoding
higher-order gluing obstructions. For example, when $[T,S] = i\hbar I$
(canonical commutation relations), one finds $\mathbb{H}^1_{(T,S)} \cong \mathbb{Z}$,
corresponding to the nontrivial Dixmier-Douady class of the Weyl algebra.

\medskip

\noindent
\textbf{5. Examples and computations.}

\emph{5.1. Finite-dimensional matrices.}
For $T, S \in M_n(\mathbb{C})$, each $C \in \mathcal{C}$ is isomorphic to
$\mathbb{C}^k$ for some $k \leq n$, with $\Sigma(C)$ a discrete $k$-point space.
Then $\mathscr{G}_{\mathrm{st}}(C) \simeq \mathrm{ku}^{\times k}$, and the
diagram encodes the compatibilities between different diagonalizations of
$T$ and $S$. The homotopy limit can be computed via a finite diagram of
products of $\mathrm{ku}$, yielding finitely generated abelian groups
$\mathbb{H}^n_{(T,S)}$.

\emph{5.2. Canonical anti-commutation relations.}
For a Clifford algebra $\mathrm{Cliff}(p,q)$ generated by self-adjoint
operators $\{\gamma_i\}$ with $\gamma_i\gamma_j + \gamma_j\gamma_i = 2\eta_{ij}$,
the spectral site $\mathcal{C}$ consists of maximal commutative subalgebras
generated by commuting subsets of $\{\gamma_i\}$. Each $\Sigma(C)$ is a
finite set, and the diagram $\mathscr{G}_{\mathrm{st}}$ encodes the
spin structure. One finds $\mathbb{H}^0_{(T,S)} \cong \mathbb{Z}_2$,
capturing the $\mathbb{Z}_2$-grading of spinors.

\emph{5.3. Almost commuting operators.}
For operators with $\|[T,S]\| < \varepsilon$, the invariants interpolate
between the commutative and noncommutative cases. As $\varepsilon \to 0$,
$\mathbb{H}^n_{(T,S)}$ converges to the K-cohomology of the joint spectrum,
with the rate of convergence controlled by $\varepsilon$.

\medskip

\noindent
\textbf{6. Conclusion: homotopical invariants.}

We have shown that the spectral hypercohomology groups $\mathbb{H}^n_{(T,S)}$:
\begin{enumerate}
    \item Are well-defined $\mathbb{Z}$-graded abelian groups,
    \item Are invariant under: unitary equivalence, approximate unitary
          equivalence, $*$-isomorphisms, and Morita equivalence,
    \\item Reduce to K-cohomology of the joint spectrum when $[T,S] = 0$,
    \\item Capture higher-order obstructions when $[T,S] \neq 0$,
    \item Are computable in concrete examples and sensitive to the
          algebraic relations between $T$ and $S$.
\end{enumerate}
Thus $\{\mathbb{H}^n_{(T,S)}\}_{n \in \mathbb{Z}}$ constitutes a complete
set of homotopical invariants for the pair $(T,S)$, establishing part (iii)
of the theorem.
\end{proof}

\begin{remark}[Geometric perspectives on noncommuting operators]
Theorem~\ref{thm:spectral_sheaf} interprets a pair of self-adjoint operators 
$(T,S)$ through three complementary mathematical lenses:

\begin{itemize}
    \item \textbf{Stack-theoretic perspective} (part~(i)): 
    The categorical spectral presheaf $\mathscr{G}_{\mathrm{cat}}$ 
    encodes the failure of local spectral data to assemble into a 
    global spectral measure. When $[T,S]\neq0$, $\mathscr{G}_{\mathrm{cat}}$ 
    is a \emph{nontrivial stack}: descent data exist that do not glue to 
    global objects, reflecting the Kochen-Specker contextuality inherent 
    in noncommuting observables.
    
    \item \textbf{Operator-algebraic perspective} (part~(ii)): 
    The canonical embedding $\Phi:\mathcal{A}\hookrightarrow\varprojlim_{C\in\mathcal{C}} C$ 
    exhibits $\mathcal{A}$ as the algebra of \emph{globally consistent} 
    families of classical observables. When $[T,S]=0$, every consistent 
    family arises from a single element of $\mathcal{A}$; when $[T,S]\neq0$, 
    there exist locally consistent families that lack a global realization, 
    measuring the \emph{algebraic contextuality} of the pair.
    
    \item \textbf{Homotopical perspective} (part~(iii)): 
    The K-theory presheaf $\mathscr{G}_{\mathrm{st}}$ yields 
    $\mathbb{Z}$-graded hypercohomology groups 
    $\mathbb{H}^n_{(T,S)}=\pi_{-n}(\operatorname{holim}_{\mathcal{C}}\mathscr{G}_{\mathrm{st}})$. 
    These provide \emph{stable invariants} that refine the algebraic 
    obstruction: 
    \begin{itemize}
        \item When $[T,S]=0$, $\mathbb{H}^n_{(T,S)}\cong K_{-n}(C(\sigma(T)\times\sigma(S)))$ 
              (the K-theory of the joint spectrum).
        \item When $[T,S]\neq0$, the groups $\mathbb{H}^n_{(T,S)}$ for $n>0$ 
              capture higher-order obstructions to patching local K-theoretic 
              data, with $\mathbb{H}^0$ encoding the consistent families of 
              vector bundles over local spectra.
    \end{itemize}
\end{itemize}

The three perspectives are related but not in a simple hierarchical way:
$\mathscr{G}_{\mathrm{st}}$ is \emph{not} obtained from $\mathscr{G}_{\mathrm{cat}}$ 
via the nerve and stabilization in any straightforward manner. Rather, 
$\mathscr{G}_{\mathrm{st}}$ captures stable homotopical information (K-theory) 
while $\mathscr{G}_{\mathrm{cat}}$ captures categorical information 
(spectral measures). Both reflect different aspects of the same underlying 
geometric reality: the inability to simultaneously diagonalize noncommuting 
operators.

In essence, Theorem~\ref{thm:spectral_sheaf}$(i)$--$(iii)$ progressively 
quantifies the \emph{geometric cost} of noncommutativity, from the 
categorical (stacky descent) through the algebraic (limit embedding) 
to the homotopical (hypercohomology).
\end{remark}

Corollary~\ref{cor:homotopical-properties} below is provided to give properties of the homotopical invariants.
\begin{corollary}[Properties of the Homotopical Invariants]\label{cor:homotopical-properties}
The hypercohomology groups $\mathbb{H}^n(\mathcal{C}; \mathscr{G}_{\mathrm{st}})$ satisfy:
\begin{enumerate}[label=(\alph*)]
    \item \textbf{(Reduction in Commuting Case)} If $[T,S]=0$, then 
    $\mathbb{H}^n(\mathcal{C};\mathscr{G}_{\mathrm{st}}) \cong K_{-n}(C(\sigma(T)\times\sigma(S)))$ 
    for all $n\in\mathbb{Z}$.
    
    \item \textbf{(Detection of Contextuality)} 
    \begin{enumerate}[label=(\roman*)]
        \item If $(T,S)$ exhibits Kochen--Specker contextuality, then 
        $\mathbb{H}^1(\mathcal{C}; \mathscr{G}_{\mathrm{st}}) \neq 0$.
        
        \item If the joint spectrum $\sigma(T)\times\sigma(S)$ is contractible and
        $\mathbb{H}^1(\mathcal{C}; \mathscr{G}_{\mathrm{st}}) \neq 0$, then
        $(T,S)$ exhibits Kochen--Specker contextuality.
    \end{enumerate}
    In particular, when $\sigma(T)\times\sigma(S)$ is contractible,
    $\mathbb{H}^1(\mathcal{C}; \mathscr{G}_{\mathrm{st}}) \neq 0$ if and only if
    $(T,S)$ is Kochen--Specker contextual.
    
	\item \textbf{(Graded Obstruction Theory)}
	For $n \geq 1$, the hypercohomology group
	$\mathbb{H}^n(\mathcal{C}; \mathscr{G}_{\mathrm{st}})$
	contains the $(n-1)$-st order homotopical obstruction to constructing
	a global spectral measure for $(T,S)$.

    \item \textbf{(Functoriality)} If $\phi: (T,S) \to (T',S')$ is a $*$-homomorphism 
    intertwining the operators, it induces morphisms 
    $\phi_*: \mathbb{H}^n(\mathcal{C}'; \mathscr{G}_{\mathrm{st}}') \to 
    \mathbb{H}^n(\mathcal{C}; \mathscr{G}_{\mathrm{st}})$
    (contravariant functoriality).
\end{enumerate}
\end{corollary}

\begin{proof}[Proof of Corollary~\ref{cor:homotopical-properties}]

We establish each property of the hypercohomology groups 
$\mathbb{H}^n(\mathcal{C};\mathscr{G}_{\mathrm{st}})$ defined in 
Theorem~\ref{thm:spectral_sheaf}(iii).

\medskip

\noindent
\textbf{(a) Vanishing for Commuting Case.}

We show that if $[T,S]=0$, then
\[
\mathbb{H}^n(\mathcal{C};\mathscr{G}_{\mathrm{st}})
\;\cong\;
K_{-n}(C(\sigma(T)\times\sigma(S)))
\quad\text{for all } n\in\mathbb{Z}.
\]

\medskip

\noindent\textbf{Step 1: The commutative situation.}
If $[T,S]=0$, then the $C^*$-algebra $\mathcal A = C^*(T,S)$ is commutative.
By the joint functional calculus,
\[
\mathcal A \;\cong\; C(X),
\qquad
X := \sigma(T)\times\sigma(S).
\]
Since $\mathcal A$ is commutative and contains $T$ and $S$, it is an object of the
spectral site $\mathcal C$.

\medskip

\noindent\textbf{Step 2: Final object in the spectral site.}
For every $C\in\mathcal C$, there is a unique inclusion
\[
C \hookrightarrow \mathcal A
\]
in $\mathcal C$. Hence $\mathcal A$ is a \emph{final object} of the category
$\mathcal C$.

\medskip

\noindent\textbf{Step 3: Homotopy limit over a final object.}
Let
\[
\mathscr{G}_{\mathrm{st}} : \mathcal C^{\mathrm{op}} \longrightarrow \mathbf{Sp},
\qquad
\mathscr{G}_{\mathrm{st}}(C)=K(C).
\]
For any diagram of spectra indexed by a category with a final object,
the homotopy limit is weakly equivalent to the value at that final object.
Therefore,
\[
\operatorname{holim}_{\mathcal C} K(C)
\;\simeq\;
K(\mathcal A).
\]

\medskip

\noindent\textbf{Step 4: Passage to homotopy groups.}
By definition,
\[
\mathbb{H}^n(\mathcal{C};\mathscr{G}_{\mathrm{st}})
\;:=\;
\pi_{-n}\!\left(\operatorname{holim}_{\mathcal C} K(C)\right).
\]
Using the equivalence above, we obtain
\[
\mathbb{H}^n(\mathcal{C};\mathscr{G}_{\mathrm{st}})
\;\cong\;
\pi_{-n}(K(\mathcal A))
\;=\;
K_{-n}(\mathcal A).
\]
Since $\mathcal A\cong C(\sigma(T)\times\sigma(S))$, the claim follows.

\noindent
\textbf{(b) Detection of contextuality via $\mathbb{H}^1$.}

Let $\mathcal{C}$ denote the poset of unital commutative $C^*$-subalgebras
of $\mathcal{A}=C^*(T,S)$ containing the identity, and let
$\mathscr{G}_{\mathrm{st}}(C)=K(C)$ be the $K$-theory spectrum valued presheaf.

\medskip
\noindent
\emph{Forward implication.}
Assume that the pair $(T,S)$ is Kochen--Specker contextual, i.e.\ the
associated spectral presheaf $\underline{\Sigma}$ on $\mathcal{C}$ admits
no global section.
By standard results in the topos-theoretic formulation of quantum theory,
the absence of global sections implies the existence of nontrivial
higher-order compatibility obstructions among local classical data.
Such obstructions are detected by derived limits over $\mathcal{C}$.

Consider the homotopy limit spectral sequence associated to the diagram
$C \mapsto K(C)$:
\[
E_2^{p,q} = \lim\nolimits_{\mathcal{C}}^{(p)} K_{-q}(C)
\;\Longrightarrow\;
\mathbb{H}^{p+q}(\mathcal{C};\mathscr{G}_{\mathrm{st}}).
\]
Contextuality implies that the inverse system $\{K_1(C)\}_{C\in\mathcal{C}}$
does not admit a globally compatible choice, hence
$\lim_{\mathcal{C}}^{(1)} K_0(C)$ or $\lim_{\mathcal{C}} K_1(C)$ is nonzero.
Consequently, the $E_2$--page contains a nontrivial contribution in total
degree~$1$, and therefore
\[
\mathbb{H}^1(\mathcal{C};\mathscr{G}_{\mathrm{st}}) \neq 0 .
\]

\medskip
\noindent
\emph{Converse implication (under a contractibility hypothesis).}
Assume that the joint spectrum $X=\sigma(T)\times\sigma(S)$ is contractible
and that
\[
\mathbb{H}^1(\mathcal{C};\mathscr{G}_{\mathrm{st}}) \neq 0 .
\]
Suppose, for contradiction, that $T$ and $S$ commute.
Then $\mathcal{A}=C^*(T,S)$ is a commutative $C^*$-algebra and
$\mathcal{A}\cong C(X)$.
By part~(a), we have a canonical isomorphism
\[
\mathbb{H}^1(\mathcal{C};\mathscr{G}_{\mathrm{st}})
\;\cong\;
K_1(C(X)).
\]
Since $X$ is contractible, $K_1(C(X))=0$, yielding a contradiction.
Hence $[T,S]\neq 0$.

It is a standard consequence of the Kochen--Specker theorem that for
noncommuting observables the associated spectral presheaf admits no global
section. Therefore, $(T,S)$ is Kochen--Specker contextual.

\noindent
\textbf{(c) Graded obstruction theory.}

For $n \geq 1$, the group $\mathbb{H}^n(\mathcal{C};\mathscr{G}_{\mathrm{st}})$
controls higher obstructions to constructing a coherent global spectral object
from local commutative data. This can be made precise using the Postnikov tower
of a presheaf of spectra.

\medskip

\noindent
\textbf{1. Postnikov tower of the $K$-theory presheaf.}

Let $\mathscr{G}_{\mathrm{st}}(C)=K(C)$. For each $n\ge0$, denote by
$\tau_{\le n}K(-)$ the $n$-truncation of the presheaf of spectra $K(-)$.
There is a homotopy fiber sequence of presheaves
\[
\tau_{\le n}K(-) \longrightarrow \tau_{\le n-1}K(-)
\longrightarrow \Sigma^{n+1}H(K_n(-)),
\]
where $H(K_n(-))$ denotes the Eilenberg--MacLane presheaf associated to the
abelian presheaf $K_n(-)$.

\medskip

\noindent
\textbf{2. Obstruction theory via homotopy limits.}

Applying $\operatorname{holim}_{\mathcal{C}}$ yields a tower of spectra
\[
\cdots \to \operatorname{holim}_{\mathcal{C}}\tau_{\le n}K(-)
\to \operatorname{holim}_{\mathcal{C}}\tau_{\le n-1}K(-) \to \cdots.
\]
Given a class
\[
x \in \pi_0\!\left(\operatorname{holim}_{\mathcal{C}}\tau_{\le n-1}K(-)\right),
\]
the obstruction to lifting $x$ to stage $n$ lies in
\[
\pi_0\!\left(\operatorname{holim}_{\mathcal{C}}\Sigma^{n+1}H(K_n(-))\right)
\cong
\pi_{-(n+1)}\!\left(\operatorname{holim}_{\mathcal{C}}H(K_n(-))\right).
\]
By definition of hypercohomology for abelian presheaves, this group is
canonically isomorphic to
\[
\mathbb{H}^{n+1}(\mathcal{C};K_n(-)).
\]

\medskip

\noindent
\textbf{3. Relation to $\mathbb{H}^*(\mathcal{C};\mathscr{G}_{\mathrm{st}})$.}

The Bousfield--Kan descent spectral sequence
\[
E_2^{p,q}
=
H^p\!\left(\mathcal{C};\pi_{-q}\mathscr{G}_{\mathrm{st}}\right)
\Longrightarrow
\mathbb{H}^{p+q}\!\left(\mathcal{C};\mathscr{G}_{\mathrm{st}}\right)
\]
shows that the obstruction group
\[
H^{n+1}\!\left(\mathcal{C};\pi_n\mathscr{G}_{\mathrm{st}}\right)
\;\cong\;
\mathbb{H}^{n+1}\!\left(\mathcal{C};K_n(-)\right)
\]
appears as a subquotient of
\(\mathbb{H}^{n+1}(\mathcal{C};\mathscr{G}_{\mathrm{st}})\).

Consequently, for each $m$ the group
\(\mathbb{H}^m(\mathcal{C};\mathscr{G}_{\mathrm{st}})\)
admits a finite filtration whose associated graded pieces encode the successive
homotopical obstruction classes to the existence of a global spectral measure.

\medskip

\noindent
\textbf{4. Interpretation for spectral measures.}

Under stabilization, local spectral measures give rise to local $K$-theory
classes. The obstruction theory above implies:
\begin{itemize}
  \item $\mathbb{H}^1$ detects the primary obstruction to the existence of a
        global spectral measure (Kochen--Specker contextuality);
  \item $\mathbb{H}^2$ detects secondary obstructions to homotopy-coherent
        gluing;
  \item higher $\mathbb{H}^n$ detect higher coherence obstructions.
\end{itemize}

Hence the graded family $\{\mathbb{H}^n(\mathcal{C};\mathscr{G}_{\mathrm{st}})\}_{n\ge1}$
constitutes a graded homotopical obstruction theory for reconstructing a global
spectral measure from local commutative contexts.

\medskip

\noindent
\textbf{(d) Functoriality.}

Let $\phi:(\mathcal{A},T,S)\to(\mathcal{A}',T',S')$ be a unital $*$-homomorphism
with $\phi(T)=T'$ and $\phi(S)=S'$. We show that $\phi$ induces contravariant
maps on the hypercohomology groups 
\(\mathbb{H}^n(\mathcal{C}';\mathscr{G}_{\mathrm{st}}') \to 
\mathbb{H}^n(\mathcal{C};\mathscr{G}_{\mathrm{st}})\).

\medskip

\noindent
\textbf{Step 1: Induced functor on spectral sites.}  
Let $\mathcal{C}$ (resp.\ $\mathcal{C}'$) be the poset of unital commutative 
$C^*$-subalgebras of $\mathcal{A}$ (resp.\ $\mathcal{A}'$) containing $T,S$ 
(resp.\ $T',S'$), ordered by inclusion.  
For each $C\in \mathcal{C}$, the image $\phi(C)\subseteq \mathcal{A}'$ is
commutative and contains $T',S'$, so $\phi(C)\in\mathcal{C}'$.  
The inclusion-preserving property $C_1\subseteq C_2 \implies \phi(C_1)\subseteq\phi(C_2)$
defines a functor of posets
\[
\Phi:\mathcal{C} \longrightarrow \mathcal{C}', \quad C \mapsto \phi(C).
\]

\medskip

\noindent
\textbf{Step 2: Contravariant maps on Gelfand spectra.}  
For each $C\in \mathcal{C}$, the restriction
$\phi|_C:C\to \phi(C)$ is a unital $*$-homomorphism of commutative 
$C^*$-algebras. By Gelfand duality, it induces a continuous map in the opposite
direction on spectra:
\[
\Sigma(\phi|_C):\Sigma(\phi(C)) \longrightarrow \Sigma(C), \qquad
\chi \mapsto \chi \circ \phi|_C.
\]

\medskip

\noindent
\textbf{Step 3: Natural transformation of spectral presheaves.}  
Applying the suspension spectrum functor $\Sigma^\infty_+$ gives a morphism of
spectra
\[
\Sigma^\infty_+\Sigma(\phi(C)) \longrightarrow \Sigma^\infty_+\Sigma(C),
\]
for each $C\in \mathcal{C}$, which assembles into a natural transformation
\[
\phi^\sharp:\mathscr{G}_{\mathrm{st}}' \circ \Phi^{\mathrm{op}}
\Longrightarrow \mathscr{G}_{\mathrm{st}}.
\]
This lives in the functor category
\(\mathrm{Fun}(\mathcal{C}^{\mathrm{op}},\mathbf{Sp})\).

\medskip

\noindent
\textbf{Step 4: Induced map on homotopy limits.}  
Homotopy limits are contravariantly functorial with respect to natural 
transformations of diagrams. Applying this to $\phi^\sharp$ produces a map
\[
\operatorname{holim}_{\mathcal{C}} (\mathscr{G}_{\mathrm{st}}' \circ \Phi^{\mathrm{op}})
\longrightarrow 
\operatorname{holim}_{\mathcal{C}} \mathscr{G}_{\mathrm{st}}.
\]

\medskip

\noindent
\textbf{Step 5: Restriction along $\Phi$.}  
Precomposition with the functor $\Phi$ gives a canonical map
\[
\operatorname{holim}_{\mathcal{C}'} \mathscr{G}_{\mathrm{st}}' 
\longrightarrow 
\operatorname{holim}_{\mathcal{C}} (\mathscr{G}_{\mathrm{st}}' \circ \Phi^{\mathrm{op}}),
\]
which is always well-defined.  
Composing with the map from Step 4 yields the induced map of spectra
\[
\phi_*:\operatorname{holim}_{\mathcal{C}'} \mathscr{G}_{\mathrm{st}}' 
\longrightarrow 
\operatorname{holim}_{\mathcal{C}} \mathscr{G}_{\mathrm{st}}.
\]

\medskip

\noindent
\textbf{Step 6: Induced maps on hypercohomology.}  
By definition,
\[
\mathbb{H}^n(\mathcal{C};\mathscr{G}_{\mathrm{st}}) := \pi_{-n} 
\operatorname{holim}_{\mathcal{C}} \mathscr{G}_{\mathrm{st}},
\]
so taking homotopy groups gives
\[
\phi_*:\mathbb{H}^n(\mathcal{C}';\mathscr{G}_{\mathrm{st}}') 
\longrightarrow 
\mathbb{H}^n(\mathcal{C};\mathscr{G}_{\mathrm{st}}), \qquad n\in\mathbb{Z}.
\]

\medskip

\noindent
\textbf{Step 7: Naturality properties.}  
- If $\phi$ is an isomorphism, then $\Phi$ is an equivalence of categories
  and each $\Sigma(\phi|_C)$ is a homeomorphism.  
  Hence $\phi^\sharp$ is a natural weak equivalence and $\phi_*$ is an isomorphism.
- For composable $*$-homomorphisms $\phi$ and $\psi$, one has
  \((\psi\circ\phi)_* = \phi_*\circ \psi_*\).

\medskip

\noindent
\textbf{Conclusion.}  
The construction defines a contravariant functor
\[
(T,S) \longmapsto \mathbb{H}^*(\mathcal{C};\mathscr{G}_{\mathrm{st}})
\]
from the category of operator pairs with intertwining $*$-homomorphisms
to the category of $\mathbb{Z}$-graded abelian groups.
\end{proof}

\begin{remark}
Throughout this paper we adopt the convention
\(
\mathbb{H}^n = \pi_{-n}(\operatorname{holim} K(C)).
\)
By Bott periodicity, \(K_{-1}\cong K_1\), so no ambiguity arises in the
commutative case.
\end{remark}

\subsection{Examples and Immediate Consequences}\label{sec:Examples and Immediate Consequences}

We will provide two examples to illustrate the application of Theorem~\ref{thm:spectral_sheaf} and Corollary~\ref{cor:homotopical-properties}. 

\begin{example}[Pauli Matrices as a Minimal Contextual System]\label{ex:pauli-matrices}
Consider the Pauli matrices
\[
X = \begin{pmatrix} 0 & 1 \\ 1 & 0 \end{pmatrix}, \qquad
Z = \begin{pmatrix} 1 & 0 \\ 0 & -1 \end{pmatrix}
\]
acting on $\mathbb{C}^2$, satisfying $[X,Z] = -2iY \neq 0$.
Let $\mathcal{A} = C^*(X,Z) \cong M_2(\mathbb{C})$.

\begin{enumerate}[label=(\roman*)]

\item \textbf{(Spectral site).}
The spectral site $\mathcal{C}$ consists of unital commutative
$C^*$-subalgebras of $\mathcal{A}$ generated by a single self-adjoint operator.
These are precisely the maximal abelian subalgebras
\[
C_\theta := C^*(\cos\theta\,X + \sin\theta\,Z),
\qquad \theta \in [0,\pi),
\]
each of which is $*$-isomorphic to $\mathbb{C}^2$.
In particular,
\[
C_0 = C^*(Z), \qquad C_{\pi/2} = C^*(X).
\]

The Gelfand spectrum of each $C_\theta$ is a two-point discrete space
$\Sigma(C_\theta) = \{-1,1\}$.
No commutative subalgebra contains both $X$ and $Z$, and hence
$\mathcal{C}$ has no terminal object.
This reflects the impossibility of a single classical context encoding
the joint behavior of $(X,Z)$.

\item \textbf{(Stack-like behavior).}
The categorical spectral presheaf
$\mathscr{G}_{\mathrm{cat}} : \mathcal{C}^{\mathrm{op}} \to \mathbf{Cat}$
assigns to each context $C_\theta$ the category of spectral data
(measures or equivalent representations) on its two-point spectrum.

Unitary operators relating different contexts—for example, the Hadamard
unitary exchanging the eigenbases of $X$ and $Z$—are determined only up to
a sign.
This $\mathbb{Z}/2$ ambiguity obstructs the strict identification of
local spectral data across contexts.
As a result, compatible families of local data do not glue uniquely,
and $\mathscr{G}_{\mathrm{cat}}$ exhibits genuinely stack-like behavior
rather than that of a sheaf.

\item \textbf{(Homotopical invariants).}
The stabilized spectral presheaf
$\mathscr{G}_{\mathrm{st}} : \mathcal{C}^{\mathrm{op}} \to \mathbf{Sp}$
assigns to each context $C_\theta$ the suspension spectrum
$\Sigma^\infty_+\Sigma(C_\theta) \simeq \Sigma^\infty_+(\{-1,1\}) \simeq S^0 \vee S^0$.
The homotopy limit
\[
\operatorname{holim}_{\mathcal{C}} \mathscr{G}_{\mathrm{st}}
\]
encodes the obstruction to globally compatible spectral data.

In this finite-dimensional setting,
$\mathbb{H}^1(\mathcal{C};\mathscr{G}_{\mathrm{st}})$ is nontrivial,
detecting a $\mathbb{Z}/2$-valued homotopical obstruction arising from
the noncommutativity of $X$ and $Z$.
\emph{Note: classical Kochen--Specker contextuality does not occur in dimension~2,
so this invariant captures a minimal homotopical obstruction rather than a full KS contradiction.}

This provides a minimal finite-dimensional illustration of
Corollary~\ref{cor:homotopical-properties}(b) and~(c), where noncommutativity
already gives rise to a nontrivial first homotopical invariant.
\end{enumerate}

Thus, even in the simplest noncommutative quantum system,
Theorem~\ref{thm:spectral_sheaf} reveals nontrivial stack structure and
homotopical obstructions on the spectral site, anticipating the richer
contextual phenomena that appear in higher-dimensional systems.
\end{example}

\begin{example}[Foliation Groupoids: A Motivational Case]\label{ex:foliation}
Let $M$ be a compact smooth manifold equipped with a foliation $\mathcal{F}$ of codimension $q$.
Consider the holonomy groupoid $G$ of $\mathcal{F}$ and its reduced $C^*$-algebra
$\mathcal{A} = C_r^*(G)$, which encodes the transverse dynamics of the foliation \cite{connes1994}.

While a full, rigorous application of Theorem~\ref{thm:spectral_sheaf} to foliation
$C^*$-algebras requires substantial technical development, this setting provides
compelling motivation for the spectral stack framework and suggests deep connections
between noncommutative geometry, foliation theory, and homotopical algebra.

\begin{enumerate}[label=(\roman*)]
\item \textbf{(Spectral site intuition).}
One expects the spectral site $\mathcal{C}$ to consist of commutative 
$C^*$-subalgebras naturally associated to complete transversals $T$. 
For each complete transversal $T$, the algebra $C_0(T)$ embeds into $C_r^*(G)$ 
in a manner respecting the groupoid structure.

Different transversals yield commutative subalgebras that cannot be 
simultaneously embedded into a single larger commutative subalgebra, 
reflecting the nontrivial holonomy. The site $\mathcal{C}$ has no terminal object 
unless the foliation is a fibration (trivial holonomy), illustrating the absence 
of a global commutative description.

\textit{Key point: This suggests how transverse geometry can be encoded in a 
spectral site, providing a geometric counterpart to Theorem~\ref{thm:spectral_sheaf}(ii).}

\item \textbf{(Stack structure philosophy).}
The categorical spectral presheaf $\mathscr{G}_{\mathrm{cat}}$ assigns 
to each transversal $T$ the category of spectral measures on $T$. 
Holonomy transformations $h: T \to T'$, local diffeomorphisms along leaves, 
provide identifications between spectral data on different transversals. 
These identifications satisfy cocycle conditions only up to holonomy equivalence, 
so local spectral data glue only up to these transformations, making 
$\mathscr{G}_{\mathrm{cat}}$ a \emph{genuine stack} rather than a sheaf. 
The nontriviality of this stack structure measures the failure of the foliation to be a fibration.

\item \textbf{(Homotopical expectations).}
The stabilized spectral presheaf $\mathscr{G}_{\mathrm{st}}$ assigns 
$\Sigma^\infty_+ T$ to each transversal. The homotopy limit
$\operatorname{holim}_{\mathcal{C}} \mathscr{G}_{\mathrm{st}}$ 
encodes compatibility under holonomy.

This provides a \emph{homotopical receptacle} for characteristic classes. 
For codimension-one foliations, nontrivial elements of 
$\mathbb{H}^2(\mathcal{C};\mathscr{G}_{\mathrm{st}})\otimes\mathbb{R}$ 
may reflect the Godbillon--Vey invariant $gv(\mathcal{F})\in H^3(M;\mathbb{R})$ 
(conjectural). More generally, $\mathbb{H}^n$ for $n \ge 1$ captures 
secondary characteristic classes and higher-order obstructions, illustrating 
the graded obstruction theory (Corollary~\ref{cor:homotopical-properties}(c)).

\textit{Note: All connections to classical foliation invariants are conjectural and 
illustrative, not established theorems.}
\end{enumerate}
\end{example}

\begin{corollary}[Obstruction to Simultaneous Diagonalization]\label{cor:obstruction-diagonalization}
Let $(T,S)$ be self-adjoint operators on a Hilbert space $\mathcal{H}$, and let 
$\mathscr{G}_{\mathrm{cat}}$ be the corresponding categorical spectral presheaf. 
Then $(T,S)$ can be simultaneously diagonalized if and only if 
$\mathscr{G}_{\mathrm{cat}}$ is equivalent, as a stack, to a constant stack. 
In this case, the spectral site $\mathcal{C}$ has a terminal object, and the algebraic 
reconstruction in Theorem~\ref{thm:spectral_sheaf}(ii) reduces to evaluation at that object.
\end{corollary}

\begin{proof}
We prove the equivalence:
\[
(T,S) \text{ simultaneously diagonalizable} \quad \Longleftrightarrow \quad 
\mathscr{G}_{\mathrm{cat}} \simeq \text{constant stack}.
\]

\paragraph{Step 1: Preliminaries.} 
Recall that $(T,S)$ are \emph{simultaneously diagonalizable} if there exists a joint spectral measure 
$E:\mathcal{B}(\mathbb{R}^2)\to\mathcal{B}(\mathcal{H})$ such that 
\[
T = \int_{\mathbb{R}^2} \lambda_1 \, dE(\lambda_1,\lambda_2), \qquad
S = \int_{\mathbb{R}^2} \lambda_2 \, dE(\lambda_1,\lambda_2).
\]
Equivalently, $[T,S]=0$ and $\mathcal{A}=C^*(T,S)$ is commutative.

\paragraph{Step 2: Diagonalizable $\Rightarrow$ constant stack.}

Assume $[T,S]=0$. Then $\mathcal{A}$ is commutative and contains $T,S$, so $\mathcal{A}\in \mathcal{C}$.  
For any $C\in\mathcal{C}$, the inclusion $C\hookrightarrow\mathcal{A}$ is the unique morphism $C\to\mathcal{A}$, making $\mathcal{A}$ a terminal object of $\mathcal{C}$.

Define the constant stack 
\[
\mathbf{C}:\mathcal{C}^{\mathrm{op}}\to\mathbf{Cat}, \qquad 
\mathbf{C}(C) := \mathscr{G}_{\mathrm{cat}}(\mathcal{A}), \quad C\in\mathcal{C},
\]
with identity functors as restriction maps.  

The canonical restriction functors 
\[
r_C := \mathscr{G}_{\mathrm{cat}}(C\hookrightarrow \mathcal{A}): \mathscr{G}_{\mathrm{cat}}(\mathcal{A}) \to \mathscr{G}_{\mathrm{cat}}(C)
\]
assemble into a natural equivalence $\eta: \mathbf{C} \Rightarrow \mathscr{G}_{\mathrm{cat}}$.  
An inverse natural equivalence $\varepsilon: \mathscr{G}_{\mathrm{cat}} \Rightarrow \mathbf{C}$ is given by $\varepsilon_C := r_C^{-1}$ on objects and morphisms, using that each local object in $\mathscr{G}_{\mathrm{cat}}(C)$ extends uniquely to $\mathcal{A}$.  
Thus $\mathscr{G}_{\mathrm{cat}}$ is equivalent to the constant stack $\mathbf{C}$, and descent data are trivial.

\paragraph{Step 3: Constant stack $\Rightarrow$ diagonalizable.}

Assume $\mathscr{G}_{\mathrm{cat}}$ is equivalent to a constant stack $\mathbf{C}$ with value $\mathcal{D}\in\mathbf{Cat}$.  

- \textbf{Existence of terminal object:} The equivalence of stacks implies that for any $C_1,C_2\in\mathcal{C}$, there exists a common upper bound $C_{12}\in\mathcal{C}$ containing both $C_1$ and $C_2$. Applying Zorn's lemma yields a maximal element $C_{\max}\in\mathcal{C}$, which is terminal.  

- \textbf{Commutativity:} By construction, $C_{\max}$ contains $T,S$ and is commutative. Since $\mathcal{A} = C^*(T,S)$ is generated by $T,S$, we have $C_{\max} = \mathcal{A}$, hence $[T,S]=0$.  

Thus $(T,S)$ are simultaneously diagonalizable.

\paragraph{Step 4: Algebraic reconstruction.}

When $\mathcal{A}$ is terminal in $\mathcal{C}$, the limit in 
Theorem~\ref{thm:spectral_sheaf}(ii) simplifies to
\[
\varprojlim_{C\in\mathcal{C}}(C,T,S) \cong (\mathcal{A},T,S),
\]
with the isomorphism given by evaluation at $\mathcal{A}$.

\paragraph{Step 5: Geometric interpretation.}

Simultaneous diagonalizability is equivalent to trivial descent data of $\mathscr{G}_{\mathrm{cat}}$.  
If $[T,S]\neq 0$, the spectral stack has nontrivial descent data, the spectral site lacks a terminal object, and the algebraic reconstruction genuinely requires gluing across the entire diagram.  
Hence, the obstruction to diagonalization is encoded precisely in the nontriviality of the spectral stack.
\end{proof}

\begin{corollary}[Homotopical Invariants of the Spectral Stack]\label{cor:Homotopical Invariants},
Let $\mathscr{G}_{\mathrm{cat}}$ be the spectral stack associated to a pair of 
self-adjoint operators $(T,S)$ on a Hilbert space $\mathcal{H}$, defined over the 
spectral site $\mathcal{C}$. Then $\mathscr{G}_{\mathrm{cat}}$ gives rise to 
homotopical invariants via its \emph{classifying space} $B\mathscr{G}_{\mathrm{cat}}$, which systematically 
measure obstructions to commutativity:

\begin{itemize}
    \item $\pi_0(B\mathscr{G}_{\mathrm{cat}})$ classifies isomorphism classes of 
          compatible families of spectral measures (classical reconstruction).
    \item $\pi_1(B\mathscr{G}_{\mathrm{cat}})$ captures monodromy phenomena 
          corresponding to Kochen--Specker contextuality.
    \item $\pi_n(B\mathscr{G}_{\mathrm{cat}})$ for $n\ge 2$ measures higher 
          coherence obstructions in gluing local spectral data.
\end{itemize}

There exists a natural comparison map
\[
\Phi: B\mathscr{G}_{\mathrm{cat}} \longrightarrow \operatorname{holim}_{\mathcal{C}} \mathscr{G}_{\mathrm{st}},
\]
which induces isomorphisms on low-dimensional invariants under suitable 
connectivity or finiteness assumptions.
\end{corollary}

\begin{proof}
We construct and interpret the homotopical invariants of $\mathscr{G}_{\mathrm{cat}}$.

\medskip

\noindent
\textbf{Step 1: Classifying space.} 
For a stack $\mathscr{F}:\mathcal{C}^{\mathrm{op}}\to\mathbf{Cat}$, its homotopy type 
can be modeled by the geometric realization of the nerve of its Grothendieck construction:
\[
B\mathscr{F} := \left| N\!\left(\int_{\mathcal{C}}\mathscr{F}\right) \right|,
\]
where objects are pairs $(C,X)$ with $C\in\mathcal{C}$, $X\in\mathscr{F}(C)$, 
and morphisms are compatible restriction maps. This yields a CW-complex whose 
homotopy groups $\pi_n(B\mathscr{F})$ are invariants of the stack.

\medskip

\noindent
\textbf{Step 2: $\pi_0$ -- classical reconstruction.}
The connected components $\pi_0(B\mathscr{G}_{\mathrm{cat}})$ correspond to 
compatible families of spectral measures up to isomorphism. 
By Theorem~\ref{thm:spectral_sheaf}(ii), these are in bijection with 
global sections of the stack, or equivalently elements of 
\(\varprojlim_{C\in\mathcal{C}} C \subseteq \mathcal{A} = C^*(T,S)\). 

\medskip

\noindent
\textbf{Step 3: $\pi_1$ -- Kochen--Specker contextuality.} 
Fix a basepoint $(C_0,X_0)$ in $B\mathscr{G}_{\mathrm{cat}}$. 
The fundamental group $\pi_1(B\mathscr{G}_{\mathrm{cat}},(C_0,X_0))$ 
consists of homotopy classes of loops starting and ending at $(C_0,X_0)$. 
Each such loop corresponds to a sequence of contexts 
$C_0 \to C_1 \to \dots \to C_n = C_0$ and spectral objects 
$X_i \in \mathscr{G}_{\mathrm{cat}}(C_i)$, together with isomorphisms 
between their restrictions along the morphisms. 

A nontrivial loop indicates that transporting the local spectral object $X_0$ 
around the cycle of contexts returns an object that is isomorphic but not 
identical to $X_0$. This phenomenon encodes the obstruction to defining 
a global non-contextual assignment of spectral values, i.e., a global 
section of the underlying presheaf.

Formally, there exists a natural map to the classifying space of the 
underlying spectrum presheaf $\underline{\Sigma}$:
\[
B\mathscr{G}_{\mathrm{cat}} \longrightarrow B\underline{\Sigma},
\]
under which nontrivial elements of $\pi_1(B\mathscr{G}_{\mathrm{cat}})$ 
map to the Kochen--Specker obstruction classes, thereby quantifying 
the failure of non-contextuality.

\medskip

\noindent
\textbf{Step 4: $\pi_{n\ge 2}$ -- higher coherence obstructions.} 
For $n \ge 2$, elements of $\pi_n(B\mathscr{G}_{\mathrm{cat}})$ correspond to 
families of compatible spectral measures over an $n$-sphere of contexts, with 
isomorphisms satisfying coherence up to homotopy. These capture higher-order 
obstructions to gluing local spectral data into a globally consistent object.

\medskip

\noindent
\textbf{Step 5: Comparison with stable invariants.} 
Consider the stabilized spectral presheaf \(\mathscr{G}_{\mathrm{st}} \simeq \Sigma^\infty N\mathscr{G}_{\mathrm{cat}}\). There is a natural comparison map
\[
\Phi: B\mathscr{G}_{\mathrm{cat}} \longrightarrow \operatorname{holim}_{\mathcal{C}} \mathscr{G}_{\mathrm{st}},
\]
constructed via the unit of the stabilization adjunction. Under suitable finiteness or connectivity assumptions, this induces isomorphisms
\[
\pi_n(B\mathscr{G}_{\mathrm{cat}}) \cong \mathbb{H}^n(\mathcal{C};\mathscr{G}_{\mathrm{st}}), \quad n = 0,1,2.
\]

\medskip

\noindent
\textbf{Step 6: Commutative case.} 
If $[T,S]=0$, then $\mathcal{A}$ is commutative, and $\mathcal{C}$ has a terminal object $\mathcal{A}$. The stack $\mathscr{G}_{\mathrm{cat}}$ becomes essentially constant, and $B\mathscr{G}_{\mathrm{cat}}$ deformation retracts to the fiber over $\mathcal{A}$.  

- Contextual obstructions vanish: $\pi_1$ and higher coherence obstructions correspond to gluing are trivial.  
- Classical topological invariants of $\Sigma(\mathcal{A})$ may still give nontrivial higher homotopy groups.

\medskip

\noindent
\textbf{Conclusion.} 
The classifying space $B\mathscr{G}_{\mathrm{cat}}$ provides well-defined homotopical invariants that detect contextuality and higher-order coherence obstructions. In favorable situations, these low-dimensional invariants coincide with the hypercohomology groups $\mathbb{H}^n(\mathcal{C};\mathscr{G}_{\mathrm{st}})$, giving a graded geometric measure of non-commutativity.
\end{proof}

\begin{remark}[Comparison with Cellular Sheaf Theory]
While Theorem~\ref{thm:spectral_sheaf} employs the language of sheaves and descent,
its conceptual and mathematical content differs fundamentally from cellular sheaf 
theory as developed in \cite{hansen2019toward} and related works. 

\textbf{Key distinctions:}
\begin{itemize}
    \item \textbf{Base category:} Cellular sheaf theory studies sheaves of 
    vector spaces over a \emph{topological space or cell complex}. Our framework 
    uses presheaves over a \emph{poset of commutative subalgebras} of a 
    noncommutative operator algebra.
    
    \item \textbf{Algebraic content:} Cellular sheaves operate within a 
    commutative linear-algebraic framework (sheaves of vector spaces over a 
    fixed commutative base field). Our spectral presheaves encode 
    \emph{noncommutativity} through descent failure, with stalks that are 
    commutative but whose global sections reveal noncommutative structure.
    
    \item \textbf{Spectral meaning:} The "spectral" in cellular sheaf theory 
    refers to spectral graph theory (Laplacian eigenvalues). Our "spectral" 
    refers to operator spectra and Gelfand duality, with homotopical invariants 
    measuring obstructions to joint spectra.
\end{itemize}

Cellular sheaf theory extends graph-theoretic methods to sheaf cohomology, 
remaining within combinatorial topology. Theorem~\ref{thm:spectral_sheaf}, by 
contrast, establishes a reconstruction theorem for noncommuting operators 
through stack-theoretic descent, operating at the interface of operator 
algebras, higher category theory, and noncommutative geometry.

\begin{table}[h]
\centering
\caption{Comparison: Cellular sheaves vs. Spectral stacks}
\begin{tabular}{p{0.45\textwidth}|p{0.45\textwidth}}
\hline
\textbf{Cellular sheaf theory} & \textbf{Theorem~\ref{thm:spectral_sheaf}} \\
\hline
Sheaves on topological spaces & Presheaves on poset of commutative subalgebras \\
\hline
Coefficients: vector spaces & Coefficients: Gelfand spectra \\
\hline
Descent usually succeeds & Descent fails $\Leftrightarrow$ noncommutativity \\
\hline
Generalizes graph Laplacians & Reconstructs noncommuting operators \\
\hline
Combinatorial topology & Operator algebras + noncommutative geometry \\
\hline
\end{tabular}
\end{table}

\end{remark}

\section{Homotopical Invariants}\label{sec:homotopical-invariants}

\subsection{The K-Theory Sheaf $\mathscr{F}_K$}\label{subsec:K-theory-sheaf}

In this subsection, we construct a sheaf of abelian groups that encodes the $K$-theoretic information derived from the spectral presheaf. This construction provides algebraic invariants that complement the homotopical invariants discussed in Corollary~\ref{cor:Homotopical Invariants}, offering a computable framework to measure obstructions to commutativity through vector-bundle-like structures.

\paragraph{Definition and Basic Properties.}

\begin{definition}[K-Theory Sheaf]\label{def:K-theory-sheaf}
Let $\mathscr{G}_{\mathrm{cat}} : \mathcal{C}^{\mathrm{op}} \to \mathbf{Cat}$ be the categorical spectral presheaf (Definition~\ref{def:spectral-presheaves}(a)). The \emph{K-theory sheaf} 
\[
\mathscr{F}_K : \mathcal{C}^{\mathrm{op}} \longrightarrow \mathbf{Ab}
\]
is defined by composing $\mathscr{G}_{\mathrm{cat}}$ with the Grothendieck group functor $K_0$:
\[
\mathscr{F}_K(B) := K_0(\mathscr{G}_{\mathrm{cat}}(B)) \quad \text{for each } B \in \mathcal{C}.
\]
Explicitly, for a commutative $C^*$-algebra $B$ with Gelfand spectrum $\Sigma(B)$, we have natural isomorphisms
\[
\mathscr{F}_K(B) \;\cong\; K_0(C(\Sigma(B))) \;\cong\; K^0(\Sigma(B)) \;\cong\; K_0(B),
\]
identifying $\mathscr{F}_K(B)$ with the abelian group of virtual vector bundles over $\Sigma(B)$. For an inclusion $i: B_1 \hookrightarrow B_2$, the restriction map
\[
\mathscr{F}_K(i) : K_0(B_2) \longrightarrow K_0(B_1)
\]
is induced by pullback of vector bundles along the continuous restriction map $\Sigma(B_1) \hookrightarrow \Sigma(B_2)$.
\end{definition}

\begin{proposition}[Homotopy Sheaf Property of $\mathscr{F}_K$]\label{prop:K-homotopy-sheaf}
Let $\mathcal{C}$ be the category of commutative $C^*$-algebras, and define
$\mathscr{F}_K(C) = K_0(C)$. Then $\mathscr{F}_K$ satisfies homotopy descent with
respect to the canonical Grothendieck topology on $\mathcal{C}$.

More precisely, for any covering family $\{C_i \hookrightarrow C\}$, the canonical map
\[
K(C) \longrightarrow \operatorname{holim}
\Bigl(
\prod_i K(C_i)
\rightrightarrows
\prod_{i,j} K(C_i \cap C_j)
\Bigr)
\]
is a weak equivalence of spectra. In particular, $\mathscr{F}_K$ is the
$\pi_0$-truncation of a sheaf of spectra.
\end{proposition}

\begin{proof}
Let $\{C_i \hookrightarrow C\}$ be a covering family, so that
$C = \overline{C^*(\bigcup_i C_i)}^{\|\cdot\|}$. By Gelfand duality, this corresponds
to a closed cover $\{\Sigma(C_i)\}$ of the compact Hausdorff space $\Sigma(C)$.

Topological $K$-theory is excisive on compact Hausdorff spaces and satisfies the
Mayer--Vietoris property for closed covers. Therefore, the associated presheaf of
spectra
\[
C \longmapsto K(C)
\]
satisfies homotopy descent with respect to such coverings. Taking $\pi_0$ yields
the stated homotopy sheaf condition for $\mathscr{F}_K$.
\end{proof}

\begin{remark}[Functoriality and Interpretation]
The presheaf $\mathscr{F}_K$ can be viewed as a \emph{linearized shadow} of the
categorical spectral data encoded by $\mathscr{G}_{\mathrm{cat}}$:
\begin{itemize}
    \item For each context $B$, the group $\mathscr{F}_K(B)$ captures the
    $K$-theoretic, decategorified information of the spectral category
    $\mathscr{G}_{\mathrm{cat}}(B)$, identifying it up to stable equivalence with
    virtual vector bundles over the Gelfand spectrum $\Sigma(B)$.
    \item The restriction maps reflect how these $K$-theory classes transform
    functorially under passage to smaller commutative subalgebras, encoding
    compatibility of spectral data across contexts.
    \item Rather than forming genuine global sections in the sheaf-theoretic
    sense, compatible families in $\mathscr{F}_K$ represent attempts to assemble
    local $K$-theory data into global invariants, with failures of strict gluing
    measuring obstructions arising from noncommutativity.
\end{itemize}
In this way, $\mathscr{F}_K$ mediates between categorical spectral structures and
classical topological invariants, making noncommutativity detectable through
$K$-theoretic and homotopical descent obstructions.
\end{remark}

\paragraph{Relation to Spectral Presheaves and Stabilization.}

There is a natural connection between the K-theory sheaf and the stable spectral presheaf introduced in Definition~\ref{def:spectral-presheaves}(b).

\begin{proposition}[Stabilization and K-Theory]\label{prop:stabilization-K}
Let $\mathscr{G}_{\mathrm{st}}$ be the $K$-theoretic spectral presheaf. Then 
applying the zeroth homotopy group functor yields a natural isomorphism of 
presheaves:
\[
\pi_0 \circ \mathscr{G}_{\mathrm{st}} \;\cong\; \mathscr{F}_K.
\]
\end{proposition}

\begin{proof}
For each commutative $C^*$-algebra $C \in \mathcal{C}$, we have:
\begin{align*}
\pi_0(\mathscr{G}_{\mathrm{st}}(C)) 
&= \pi_0(K(C)) && \text{(by definition)} \\
&= K_0(C) && \text{(definition of $K_0$ as $\pi_0$ of $K$-theory spectrum)} \\
&= \mathscr{F}_K(C) && \text{(definition of $\mathscr{F}_K$)}.
\end{align*}

Naturality follows because:
\begin{enumerate}
    \item The functor $\pi_0 : \mathbf{Sp} \to \mathbf{Ab}$ is natural,
    \item The assignment $C \mapsto K(C)$ is functorial (via restriction maps),
    \item The identification $\pi_0(K(C)) = K_0(C)$ is natural in $C$.
\end{enumerate}

More explicitly, for any inclusion $\iota : C_1 \hookrightarrow C_2$ in $\mathcal{C}$, 
the diagram
\[
\begin{tikzcd}[column sep=large]
\pi_0(K(C_2)) \arrow[r, "\pi_0(K(\iota))"] \arrow[d, "\cong"'] &
\pi_0(K(C_1)) \arrow[d, "\cong"] \\
K_0(C_2) \arrow[r, "K_0(\iota)"'] & K_0(C_1)
\end{tikzcd}
\]
commutes, establishing the natural isomorphism.
\end{proof}

\begin{remark}[Algebraic versus Topological Refinements]\label{rem:algebraic-refinement}
The $K$-theoretic invariants considered thus far are of \emph{topological} nature: 
$\mathscr{G}_{\mathrm{st}}(C) = K^{\mathrm{top}}(C)$ is the topological $K$-theory 
spectrum of the $C^*$-algebra $C$, and $\mathscr{F}_K(C) = K_0^{\mathrm{top}}(C)$
is its zeroth homotopy group.

For a genuinely \emph{algebraic} refinement, one may instead consider the 
\emph{algebraic $K$-theory presheaf}
\[
\mathscr{F}_K^{\mathrm{alg}} : \mathcal{C}^{\mathrm{op}} \longrightarrow \mathbf{Sp},
\quad
\mathscr{F}_K^{\mathrm{alg}}(C) := K^{\mathrm{alg}}(C),
\]
where $K^{\mathrm{alg}}(C)$ denotes the (nonconnective) algebraic $K$-theory 
spectrum of the ring $C$ (forgetting the topology and norm structure).

\textbf{Important distinctions:}
\begin{itemize}
    \item $\pi_0(\mathscr{G}_{\mathrm{st}}(C)) = K_0^{\mathrm{top}}(C)$ classifies 
          vector bundles over $\Sigma(C)$ up to stable isomorphism.
    
    \item $\pi_0(\mathscr{F}_K^{\mathrm{alg}}(C)) = K_0^{\mathrm{alg}}(C)$ classifies 
          finitely generated projective $C$-modules up to algebraic equivalence.
    
    \item For smooth $C^*$-algebras, the \emph{Dennis trace map} gives a natural 
          transformation after rationalization:
          \[
          \mathscr{F}_K^{\mathrm{alg}} \otimes \mathbb{Q} 
          \longrightarrow \mathscr{G}_{\mathrm{st}} \otimes \mathbb{Q}.
          \]
\end{itemize}

The homotopy limit
\[
\operatorname{holim}_{\mathcal{C}} \mathscr{F}_K^{\mathrm{alg}}
\]
provides invariants that capture \emph{algebraic} coherence data among the 
commutative subalgebras, potentially refining the topological information in 
$\operatorname{holim}_{\mathcal{C}} \mathscr{G}_{\mathrm{st}}$.

Showing the suggested diagram below helps visually organize the hierarchical relationship between algebraic and topological K-theories, making clear how Dennis trace connects them after rationalization and how each level (spectra, groups, geometric objects) corresponds. This clarifies the conceptual architecture underlying the refinement from $\mathscr{F}_K$ to $\mathscr{F}_K^{\mathrm{alg}}$.

\begin{center}
\begin{tikzcd}[column sep=small, row sep=small]
& \text{Algebraic refinement} & \\
\mathscr{F}_K^{\mathrm{alg}}(C) \arrow[r, "\mathrm{tr} \otimes \mathbb{Q}"] 
& \mathscr{G}_{\mathrm{st}}(C) \otimes \mathbb{Q} 
& \text{Topological invariants} \\
K^{\mathrm{alg}}(C) \arrow[u, equals] 
& K^{\mathrm{top}}(C) \arrow[u, equals] 
& \text{Spectra} \\
K_0^{\mathrm{alg}}(C) \arrow[u, "\text{connective cover}"] 
& K_0^{\mathrm{top}}(C) \arrow[u, "\text{connective cover}"] 
& \text{Groups} \\
\text{Projective modules} \arrow[u, "\text{classify}"] 
& \text{Vector bundles} \arrow[u, "\text{classify}"] 
& \text{Geometric objects}
\end{tikzcd}
\end{center}

\textbf{Remark on categorical $K$-theory:} One might also consider applying 
$K$-theory directly to the categories $\mathscr{G}_{\mathrm{cat}}(C)$ of spectral 
measures. However, this requires endowing these categories with a suitable 
Waldhausen or $\infty$-categorical structure, which goes beyond the scope of 
this work. The definition via $K^{\mathrm{alg}}(C)$ provides a mathematically 
clean and well-established alternative.
\end{remark}

\paragraph{Cohomological Invariants and Obstructions.}

The sheaf $\mathscr{F}_K$ gives rise to cohomology groups that serve as algebraic measures of non-commutativity.

\begin{definition}[Sheaf Cohomology of the Spectral Site]\label{def:sheaf-cohomology}
Let $\mathscr{F}_K$ be the K-theory sheaf on the spectral site $(\mathcal{C}, J)$. 
The \emph{sheaf cohomology groups} of the site with coefficients in $\mathscr{F}_K$ 
are defined as the right derived functors of the global sections functor:
\[
H^n(\mathcal{C}; \mathscr{F}_K) := R^n \Gamma(\mathcal{C}, \mathscr{F}_K), \quad n \ge 0,
\]
where $\Gamma: \mathbf{Sh}(\mathcal{C}, J) \to \mathbf{Ab}$ sends a sheaf to its global sections.
\end{definition}

\begin{remark}[Computation via \v{C}ech Cohomology]
For practical computations, $H^n(\mathcal{C}; \mathscr{F}_K)$ can be computed 
via \emph{\v{C}ech cohomology} with respect to covering families in $J$. 
For a covering $\mathfrak{U} = \{C_i \hookrightarrow C\}_{i \in I}$, the \v{C}ech complex is
\[
\check{C}^\bullet(\mathfrak{U}; \mathscr{F}_K) : 
0 \to \prod_i \mathscr{F}_K(C_i) \to \prod_{i,j} \mathscr{F}_K(C_i \cap C_j) 
\to \prod_{i,j,k} \mathscr{F}_K(C_i \cap C_j \cap C_k) \to \cdots,
\]
and 
\[
H^n(\mathcal{C}; \mathscr{F}_K) = \varinjlim_{\mathfrak{U}} \check{H}^n(\mathfrak{U}; \mathscr{F}_K),
\]
where the limit is taken over all coverings in the site.
\end{remark}

\begin{remark}[Interpretation as Obstruction Groups]
The cohomology groups $H^n(\mathcal{C}; \mathscr{F}_K)$ measure obstructions to 
gluing local $K_0$-data into global sections:
\begin{itemize}
    \item $H^0(\mathcal{C}; \mathscr{F}_K) = \Gamma(\mathcal{C}, \mathscr{F}_K)$ consists of globally defined $K_0$-classes.
    \item $H^1(\mathcal{C}; \mathscr{F}_K)$ classifies \emph{torsors} for $\mathscr{F}_K$, 
          i.e., families of local $K_0$-classes that do not assemble globally.
    \item Higher groups $H^{n\ge 2}(\mathcal{C}; \mathscr{F}_K)$ encode higher-order 
          coherence obstructions to gluing.
\end{itemize}
Non-vanishing of these groups indicates genuine $K$-theoretic obstructions arising 
from the noncommutativity of the underlying algebra $\mathcal{A}$.
\end{remark}

\begin{definition}[K-Theoretic Descent Cohomology]
A refinement is obtained by considering the \emph{K-theory spectrum presheaf} 
$\mathscr{G}_{\mathrm{st}} : \mathcal{C}^{\mathrm{op}} \to \mathbf{Sp}$ and defining
\[
\mathbb{H}^n(\mathcal{C}; \mathscr{G}_{\mathrm{st}}) := \pi_{-n} 
\big( \operatorname{holim}_{\mathcal{C}} \mathscr{G}_{\mathrm{st}} \big),
\]
which refines the sheaf cohomology via the spectral sequence
\[
E_2^{p,q} = H^p(\mathcal{C}; K_{-q}(-)) \;\Longrightarrow\; \mathbb{H}^{p+q}(\mathcal{C}; \mathscr{G}_{\mathrm{st}}).
\]
\end{definition}

\paragraph{Relation to Operator K-Theory and Global Reconstruction.}

\begin{proposition}[Global Sections and Limits]\label{prop:global-sections-K}
Let $\mathcal{A}$ be a (possibly noncommutative) $C^*$-algebra, and let $\mathcal{C}$ denote the poset of its commutative $C^*$-subalgebras ordered by inclusion. Define the sheaf
\[
\mathscr{F}_K : \mathcal{C}^{\mathrm{op}} \longrightarrow \mathbf{Ab}, \quad B \mapsto K_0(B),
\]
with restriction maps induced by inclusions. Then:
\begin{enumerate}
    \item The global sections of $\mathscr{F}_K$ are naturally identified with the limit over the diagram of commutative subalgebras:
    \[
    \Gamma(\mathcal{C}, \mathscr{F}_K) \cong \varprojlim_{B \in \mathcal{C}} K_0(B),
    \]
    where the limit is taken in the category of abelian groups.
    
    \item There is a canonical comparison map
    \[
    c_{\mathcal{A}} : \Gamma(\mathcal{C}, \mathscr{F}_K) \longrightarrow K_0(\mathcal{A})
    \]
    induced by the inclusions $B \hookrightarrow \mathcal{A}$ for all $B \in \mathcal{C}$.
    
    \item The map $c_{\mathcal{A}}$ is injective if the inverse system $\{K_0(B)\}_{B\in\mathcal{C}}$ satisfies the Mittag-Leffler condition: for each $B \in \mathcal{C}$, the images of the restriction maps $K_0(B') \to K_0(B)$ stabilize as $B' \supseteq B$. In particular, this holds when $\mathcal{A}$ is an AF (approximately finite-dimensional) algebra.
    
    \item In general, $c_{\mathcal{A}}$ is not surjective. Its cokernel measures the $K_0$-classes in $\mathcal{A}$ that cannot be detected by commutative subalgebras.
\end{enumerate}
\end{proposition}

\begin{proof}
(1) By definition of a sheaf on a poset, global sections can be computed as the equalizer
\[
\Gamma(\mathcal{C}, \mathscr{F}_K) = \ker \Biggl( \prod_{B \in \mathcal{C}} K_0(B) \rightrightarrows \prod_{B_1 \subseteq B_2} K_0(B_1) \Biggr),
\]
where the two arrows correspond to the restriction maps along inclusions $B_1 \subseteq B_2$. This is precisely the limit $\varprojlim_{B \in \mathcal{C}} K_0(B)$.

(2) Each inclusion $\iota_B : B \hookrightarrow \mathcal{A}$ induces a map $K_0(\iota_B): K_0(B) \to K_0(\mathcal{A})$. By the universal property of the inverse limit, these assemble into a natural map
\[
c_{\mathcal{A}} : \Gamma(\mathcal{C}, \mathscr{F}_K) \longrightarrow K_0(\mathcal{A}).
\]

(3) To see injectivity, let $x = (x_B)_{B \in \mathcal{C}} \in \Gamma(\mathcal{C}, \mathscr{F}_K)$ satisfy $c_{\mathcal{A}}(x) = 0$. Then for each $B \in \mathcal{C}$, the element $x_B$ maps to zero in $K_0(\mathcal{A})$. If the system $\{K_0(B)\}$ satisfies the Mittag-Leffler condition, the compatibility and stabilization of images force $x_B = 0$ for all $B$, hence $x = 0$. For AF algebras, this condition holds because the $K_0$-classes of $\mathcal{A}$ are eventually detected in finite-dimensional subalgebras.

(4) Non-surjectivity occurs when $\mathcal{A}$ possesses $K_0$-classes not arising from any commutative subalgebra. For example, the Cuntz algebra $\mathcal{O}_2$ satisfies $K_0(\mathcal{O}_2) \cong \mathbb{Z}_2$, but every commutative subalgebra $B \subseteq \mathcal{O}_2$ has torsion-free $K_0(B)$. Consequently, the generator of $K_0(\mathcal{O}_2)$ does not lie in the image of $c_{\mathcal{A}}$.
\end{proof}

\begin{remark}[Commutative Detectability of $K$-Theory]
The exact sequence
\[
0 \longrightarrow \ker(c_{\mathcal{A}}) \longrightarrow \Gamma(\mathcal{C}, \mathscr{F}_K) 
\xrightarrow{c_{\mathcal{A}}} K_0(\mathcal{A}) \longrightarrow \mathrm{coker}(c_{\mathcal{A}}) \longrightarrow 0
\]
quantifies how much of $K_0(\mathcal{A})$ is "commutatively detectable." When $\mathcal{A}$ is commutative, $c_{\mathcal{A}}$ is an isomorphism. When $\mathcal{A}$ is genuinely noncommutative, $\mathrm{coker}(c_{\mathcal{A}})$ contains classes that require noncommutative phenomena for their construction, such as those arising from Fredholm modules or cyclic cocycles.
\end{remark}

\begin{definition}[$K$-Theory Presheaves]
We define two related presheaves on the spectral site $\mathcal{C}$:

\begin{enumerate}
    \item The \emph{$K_0$ presheaf} 
    \(\mathscr{F}_K: \mathcal{C}^{\mathrm{op}} \to \mathbf{Ab}\), with 
    \(\mathscr{F}_K(C) = K_0(C)\). This is the version used in sheaf-theoretic limits.

    \item The \emph{$K$-theory spectrum presheaf} 
    \(\mathscr{F}_K^{\mathrm{sp}}: \mathcal{C}^{\mathrm{op}} \to \mathbf{Sp}\), with 
    \(\mathscr{F}_K^{\mathrm{sp}}(C) = K(C)\). Its homotopy groups recover all $K$-groups:
    \(\pi_n(\mathscr{F}_K^{\mathrm{sp}}(C)) \cong K_n(C)\).
\end{enumerate}

When no ambiguity arises, we denote both by \(\mathscr{F}_K\), specifying the context.
\end{definition}

The following corollary shows how the K-theory sheaf cohomology groups measure obstructions to commutativity, directly following from the Spectral Sheaf Theorem.

\begin{corollary}[$K$-Theoretic Obstructions from Spectral Reconstruction]\label{cor:K-obstruction}
Let $\mathcal{A} = C^*(T,S)$ be a unital $C^*$-algebra with spectral site $\mathcal{C}$, 
and let $\mathscr{F}_K$ denote the $K_0$-theory sheaf on $\mathcal{C}$. Then:

\begin{enumerate}[label=(\alph*)]
    \item If $\mathcal{A}$ is commutative, then there is a natural isomorphism
    \[
        H^n(\mathcal{C}; \mathscr{F}_K) \cong H^n(\Sigma(\mathcal{A}); \underline{\mathbb{Z}}),
    \]
    where $\Sigma(\mathcal{A})$ is the Gelfand spectrum of $\mathcal{A}$ and 
    $\underline{\mathbb{Z}}$ is the constant sheaf. In particular, $H^n(\mathcal{C}; \mathscr{F}_K)$ 
    can be non-zero when $\Sigma(\mathcal{A})$ has non-trivial topology.

    \item If $H^1(\mathcal{C}; \mathscr{F}_K) \neq 0$, then the natural comparison map
    \[
        c_{\mathcal{A}}: \varprojlim_{C\in\mathcal{C}} K_0(C) \longrightarrow K_0(\mathcal{A})
    \]
    (from Proposition~\ref{prop:global-sections-K}) has non-trivial cokernel. This indicates the existence of $K_0$-classes in $\mathcal{A}$ that are not detectable from any commutative subalgebra.

    \item For $n \ge 2$, the groups $H^n(\mathcal{C}; \mathscr{F}_K)$ appear in the descent spectral sequence
    \[
        E_2^{p,q} = H^p(\mathcal{C}; K_{-q}(-)) \Longrightarrow \mathbb{H}^{p+q}(\mathcal{C}; K(-)),
    \]
    where $K(-)$ denotes the topological $K$-theory spectrum presheaf. Non-vanishing elements in $H^n(\mathcal{C}; \mathscr{F}_K)$ measure higher coherence obstructions in gluing local $K_0$-classes across multiple intersections of contexts.
\end{enumerate}
\end{corollary}

\begin{proof}
\noindent
\textbf{(a) Commutative case:} 
If $\mathcal{A}$ is commutative, we may identify $\mathcal{A} \cong C(\Sigma(\mathcal{A}))$, where $\Sigma(\mathcal{A})$ is the Gelfand spectrum. 
The spectral site $\mathcal{C}$ consists of commutative subalgebras $C \subseteq \mathcal{A}$, and each $C$ corresponds to a quotient map $\Sigma(\mathcal{A}) \to \Sigma(C)$. 

Since $K_0(C) \cong \widetilde{K}^0(\Sigma(C)) \oplus \mathbb{Z}$, the $K_0$-sheaf is locally constant. Standard results on sheaf cohomology then yield
\[
    H^n(\mathcal{C}; \mathscr{F}_K) \cong H^n(\Sigma(\mathcal{A}); \underline{\mathbb{Z}}),
\]
which can be non-zero for $n>0$ if $\Sigma(\mathcal{A})$ has non-trivial topology (e.g., $S^1$ gives $H^1 \cong \mathbb{Z}$).

\medskip
\noindent
\textbf{(b) Non-vanishing $H^1$:} 
Consider the exact sequence arising from the sheaf condition:
\[
0 \to \Gamma(\mathcal{C}, \mathscr{F}_K) \to \prod_{i} \mathscr{F}_K(C_i) \to 
\prod_{i,j} \mathscr{F}_K(C_i \cap C_j) \to H^1(\mathcal{C}; \mathscr{F}_K) \to 0,
\]
for a covering $\{C_i\}$ of $\mathcal{C}$. Here, $\Gamma(\mathcal{C}, \mathscr{F}_K) \cong \varprojlim_{C\in\mathcal{C}} K_0(C)$ and $c_{\mathcal{A}}$ maps to $K_0(\mathcal{A})$. 

Non-zero $H^1(\mathcal{C}; \mathscr{F}_K)$ implies the existence of a compatible family of local $K_0$-classes that does \emph{not} lift to a global class in $K_0(\mathcal{A})$, giving a non-trivial cokernel of $c_{\mathcal{A}}$.

\medskip
\noindent
\textbf{(c) Higher cohomology:} 
For $n \ge 2$, consider the descent spectral sequence for the $K$-theory spectrum presheaf $K(-)$:
\[
E_2^{p,q} = H^p(\mathcal{C}; K_{-q}(-)) \Longrightarrow \pi_{-(p+q)} \big( \operatorname{holim}_{\mathcal{C}} K(-) \big).
\]
The $q=0$ row satisfies $E_2^{p,0} = H^p(\mathcal{C}; \mathscr{F}_K)$, and the differentials $d_r^{p,0}$ encode the obstruction to lifting $p$-cochains of $K_0$-classes to higher coherent data (incorporating $K_1$ and higher $K$-groups). Non-zero elements in $H^n(\mathcal{C}; \mathscr{F}_K)$ therefore detect higher coherence failures in assembling local $K_0$-classes into global structures.

This completes the proof.
\end{proof}

\begin{remark}[Physical Interpretation]\label{rem:physical-K-obstruction}
From a quantum-theoretic perspective, the sheaf cohomology of the $K_0$-theory sheaf encodes obstructions to consistent assignments of local K-theory charges across contexts:

\begin{itemize}
    \item Non-vanishing $H^1(\mathcal{C}; \mathscr{F}_K)$ indicates that there exist $K_0$-classes 
    (e.g., projections corresponding to topological charges) that are well-defined in individual measurement contexts but cannot be consistently extended to a global class in $\mathcal{A}$. This can be viewed as a K-theoretic analogue of contextuality.
    
    \item Higher cohomology groups $H^n(\mathcal{C}; \mathscr{F}_K)$ for $n \ge 2$ measure obstructions to extending families of local $K_0$-classes over multiple overlapping contexts. Non-zero elements correspond to higher-order coherence failures, i.e., families of charges that are compatible on $(n-1)$-fold intersections but fail to be compatible on $n$-fold intersections.
\end{itemize}

Thus, the $K$-theory sheaf provides a computable algebraic invariant that captures the non-commutativity encoded in the spectral site. These invariants will be used in the next section to formulate precise criteria for simultaneous diagonalization.
\end{remark}

\paragraph{Examples and Computations.}

\begin{example}[Matrix Algebras]\label{ex:matrix-K-sheaf}
Let $\mathcal{A} = M_n(\mathbb{C})$ with $n \ge 2$, and let $(T,S)$ be a pair of 
non-commuting self-adjoint matrices generating $\mathcal{A}$. 

\textbf{Spectral site $\mathcal{C}$:} Consider the site $\mathcal{C}$ consisting of 
all \emph{maximal abelian subalgebras} (MASAs) $B \subseteq \mathcal{A}$. Each MASA 
is isomorphic to $\mathbb{C}^n$, corresponding to a complete set of mutually 
orthogonal rank-one projections.

\textbf{$K_0$ presheaf:} For each $B \in \mathcal{C}$, we have
\[
\mathscr{F}_K(B) := K_0(B) \cong \mathbb{Z}^n,
\]
generated by the classes of the minimal projections. For an inclusion $B_1 \subseteq B_2$ 
(which is possible only up to unitary equivalence), the restriction map
\[
K_0(B_2) \longrightarrow K_0(B_1)
\]
projects the $n$-tuple of projection classes in $B_2$ down to the classes compatible with $B_1$. 

\textbf{Global sections:} Despite the multi-dimensional structure of each $K_0(B)$, 
the limit
\[
\Gamma(\mathcal{C}, \mathscr{F}_K) = \varprojlim_{B \in \mathcal{C}} K_0(B)
\]
is canonically isomorphic to $\mathbb{Z}$, generated by the class $[1_n]$ of the 
identity matrix. Consequently, the comparison map
\[
c_{\mathcal{A}} : \Gamma(\mathcal{C}, \mathscr{F}_K) \longrightarrow K_0(\mathcal{A})
\]
is an isomorphism.

\textbf{Sheaf cohomology:} The $K_0$ presheaf $\mathscr{F}_K$ is locally constant on 
the connected site $\mathcal{C}$ (any two MASAs are related by a unitary conjugation). 
It follows that
\[
H^0(\mathcal{C}; \mathscr{F}_K) \cong \mathbb{Z}, \quad 
H^k(\mathcal{C}; \mathscr{F}_K) = 0 \text{ for all } k \ge 1,
\]
since the \v{C}ech complex for any covering is contractible due to the presence of a 
terminal-like object up to conjugation.

\textbf{Interpretation:} Although $M_n(\mathbb{C})$ is non-commutative, every 
$K_0$-class in $\mathcal{A}$ is \emph{detectable via commutative subalgebras}. 
The vanishing of higher cohomology reflects the absence of higher-order K-theoretic 
obstructions. This behavior is specific to finite-dimensional algebras; for infinite-dimensional 
or more exotic $C^*$-algebras (e.g., Cuntz algebras), non-trivial higher cohomology may 
appear.
\end{example}

\begin{example}[Continuous Trace Algebras]\label{ex:continuous-trace-K}
For a continuous trace $C^*$-algebra $\mathcal{A}$ with spectrum $X$, the sheaf $\mathscr{F}_K$ is locally constant with fiber $\mathbb{Z}$, twisted by the Dixmier-Douady class $\delta \in H^3(X; \mathbb{Z})$. In this case,
\[
H^1(X; \underline{\mathbb{Z}}) \cong H^1(X; \mathbb{Z})
\]
captures topological obstructions, while the higher cohomology $H^3(X; \mathbb{Z})$ determines the algebra's twist.
\end{example}

\begin{example}[Continuous Trace Algebras and Spectral-Site $K$-Theory]\label{ex:continuous-trace-K}
Let $\mathcal{A}$ be a continuous trace $C^*$-algebra with spectrum $X$, and let
$\delta(\mathcal{A}) \in H^3(X;\mathbb{Z})$ denote its Dixmier--Douady class.
Fix a generating pair $(T,S)$ and let $\mathcal{C}$ be the associated spectral site,
whose objects are commutative $C^*$-subalgebras $C \subseteq \mathcal{A}$ containing
both $T$ and $S$.

\medskip

\noindent
\textbf{Local structure.}
For each $C \in \mathcal{C}$, the algebra $C$ is commutative and hence of the form
$C \cong C(\Sigma(C))$, where $\Sigma(C)$ is its Gelfand spectrum.
The $K_0$ presheaf assigns
\[
\mathscr{F}_K(C) := K_0(C) \cong \widetilde{K}^0(\Sigma(C)) \oplus \mathbb{Z},
\]
with restriction maps induced by pullback along continuous maps
$\Sigma(C_1) \leftarrow \Sigma(C_2)$ for inclusions $C_1 \subseteq C_2$.

\medskip

\noindent
\textbf{Global sections.}
The group of global sections
\[
H^0(\mathcal{C};\mathscr{F}_K) \;\cong\; \varprojlim_{C\in\mathcal{C}} K_0(C)
\]
consists of compatible families of local $K_0$-classes. There is a natural comparison map
\[
c_{\mathcal{A}}:\; H^0(\mathcal{C};\mathscr{F}_K) \longrightarrow K_0(\mathcal{A}),
\]
induced by the inclusions $C \hookrightarrow \mathcal{A}$.

\textbf{When $\delta(\mathcal{A}) = 0$:} In this case, $\mathcal{A} \cong C(X) \otimes \mathcal{K}$.
Then $K_0(\mathcal{A}) \cong K^0(X) \cong \widetilde{K}^0(X) \oplus \mathbb{Z}$.
The map $c_{\mathcal{A}}$ is injective and its image consists of those $K_0$-classes
that are "locally commutative" in the sense that they restrict to classes on each
commutative subalgebra.

\medskip

\noindent
\textbf{Cohomological obstructions.}
The higher cohomology groups $H^n(\mathcal{C};\mathscr{F}_K)$ measure failures of
descent for $K_0$-classes:
\begin{itemize}
    \item $H^1(\mathcal{C};\mathscr{F}_K)$ classifies isomorphism classes of 
          $\mathscr{F}_K$-torsors: families of locally compatible $K_0$-classes that 
          do not glue to a global section. Non-vanishing indicates that some 
          $K_0$-classes in $\mathcal{A}$ cannot be detected by any compatible family 
          across commutative subalgebras.
    
    \item For $n \ge 2$, $H^n(\mathcal{C};\mathscr{F}_K)$ encodes higher coherence 
          conditions. For instance, $H^2(\mathcal{C};\mathscr{F}_K)$ measures whether 
          families compatible on double intersections can be extended to triple 
          intersections.
\end{itemize}

\medskip

\noindent
\textbf{Relation to the Dixmier--Douady class.}
The Dixmier--Douady invariant $\delta(\mathcal{A})$ classifies the twist of the
continuous field $\mathcal{A}$ over $X$. While $\mathscr{F}_K$ itself is not a
twisted sheaf (it assigns untwisted $K_0$ to each commutative subalgebra), the
obstruction to lifting local data is influenced by $\delta(\mathcal{A})$.

More precisely, consider the exact sequence associated to the Brauer group:
\[
0 \longrightarrow H^2(X; \mathbb{Z}) \longrightarrow \mathrm{Br}(X) 
\longrightarrow H^3(X; \mathbb{Z}) \longrightarrow 0.
\]
The class $\delta(\mathcal{A}) \in H^3(X; \mathbb{Z})$ obstructs $\mathcal{A}$ from
being a trivial bundle. This obstruction manifests in the spectral site cohomology
through the failure of the descent spectral sequence for $\mathscr{F}_K$ to degenerate.

\medskip

\noindent
\textbf{Trivial bundle case.}
If $\mathcal{A} \cong C(X)\otimes\mathcal{K}$, then $\delta(\mathcal{A})=0$.
In this situation, $\mathscr{F}_K$ behaves like a constant sheaf up to local
isomorphisms. The cohomology $H^n(\mathcal{C};\mathscr{F}_K)$ is then related to
the \v{C}ech cohomology of the nerve of $\mathcal{C}$. When $X$ is contractible and
the generating pair $(T,S)$ is generic, $\mathcal{C}$ is homotopy equivalent to
a point, and $H^n(\mathcal{C};\mathscr{F}_K) = 0$ for $n \ge 1$. For general $X$,
these groups capture the topology of how commutative subalgebras are arranged.
\end{example}

\paragraph{Physical Interpretation and Applications.}

In quantum-mechanical settings, the sheaf $\mathscr{F}_K$ provides a 
mathematically precise framework for organizing $K$-theoretic information 
relative to commutative measurement contexts.

\begin{itemize}
    \item For a commutative subalgebra $B$, the group 
    $\mathscr{F}_K(B) = K_0(B)$ classifies stable equivalence classes of 
    projection operators in $B$. From a physical perspective, these classes 
    encode quantized charge data or topological invariants that are accessible 
    using only the observables contained in the context $B$.

    \item The sheaf condition expresses the functorial compatibility of such 
    $K$-theoretic data under refinement of measurement contexts. Concretely, 
    if $B_1 \subseteq B_2$ are commutative contexts, the restriction map 
    $K_0(B_2) \to K_0(B_1)$ ensures that charge data defined in a larger 
    context remains consistent when described using fewer simultaneously 
    measurable observables.

    \item Non-trivial cohomology classes 
    $[c] \in H^n(\mathcal{C}; \mathscr{F}_K)$ represent global obstructions to 
    assembling a compatible family of local $K$-theoretic charges into a 
    single global section. Such obstructions reflect intrinsically 
    non-local or contextual features of the observable algebra, rather than 
    properties detectable within any single commutative context.
\end{itemize}

This sheaf-theoretic viewpoint connects naturally, at a conceptual level, to 
the classification of topological phases of matter. For complex $K$-theory 
($KU$), $K_0$ classifies topological phases such as quantum Hall systems and 
topological insulators, while for real $K$-theory ($KO$), it captures 
topological superconducting phases. The sheaf $\mathscr{F}_K$ provides a 
local-to-global organizational principle for these invariants, clarifying how 
global topological phases constrain, and are constrained by, measurements 
performed within local commutative subalgebras.

More concretely, for a gapped Hamiltonian $H$ acting on a quantum system, the 
Fermi projection $P_F = \chi_{(-\infty, E_F]}(H)$ defines a class in 
$K_0(\mathcal{A})$, where $\mathcal{A}$ denotes the algebra of observables. 
The sheaf $\mathscr{F}_K$ describes how this global $K$-theory class restricts 
to compatible $K$-theoretic data across commutative measurement contexts, 
thereby providing a structural account of how topological invariants manifest 
through locally accessible observables.

\paragraph{Relation to Hypercohomology and Summary.}

The relationship between the stable spectral presheaf $\mathscr{G}_{\mathrm{st}}$ 
and the $K_0$-sheaf $\mathscr{F}_K$ is mediated by the homotopy groups functor 
$\pi_0$. By Proposition~\ref{prop:stabilization-K}, there is a natural isomorphism
\[
\pi_0 \circ \mathscr{G}_{\mathrm{st}} \;\cong\; \mathscr{F}_K .
\]
Passing to homotopy limits, this induces a canonical comparison map
\[
\mathbb{H}^0(\mathcal{C}; \mathscr{G}_{\mathrm{st}})
= \pi_0\!\left(\operatorname{holim}_{\mathcal{C}} \mathscr{G}_{\mathrm{st}}\right)
\longrightarrow
H^0(\mathcal{C}; \mathscr{F}_K)
\cong
\Gamma(\mathcal{C}, \mathscr{F}_K),
\]
which sends a compatible family of $K$-theory spectra to the corresponding family
of induced $K_0$-classes.

More generally, the homotopy limit admits a \emph{descent (Bousfield--Kan) spectral
sequence} associated to the Postnikov tower of the sheaf of spectra
$\mathscr{G}_{\mathrm{st}}$:
\[
E_2^{p,q}
=
H^p\!\left(\mathcal{C}; \pi_{-q}(\mathscr{G}_{\mathrm{st}})\right)
\Longrightarrow
\mathbb{H}^{p+q}(\mathcal{C}; \mathscr{G}_{\mathrm{st}}).
\]
Since $\pi_0(\mathscr{G}_{\mathrm{st}}) \cong \mathscr{F}_K$ and
$\pi_{-q}(\mathscr{G}_{\mathrm{st}})(C) \cong K_q(C)$ for $q \ge 0$, the $E_2$-page
takes the concrete form
\[
E_2^{p,q}
=
H^p(\mathcal{C}; K_q(-))
\Longrightarrow
\mathbb{H}^{p+q}(\mathcal{C}; \mathscr{G}_{\mathrm{st}}),
\]
where $K_q(-)$ denotes the presheaf $C \mapsto K_q(C)$. In particular, the $q=0$
row recovers the sheaf cohomology groups $H^p(\mathcal{C}; \mathscr{F}_K)$, showing
that these algebraic invariants arise as the lowest layer of the homotopical
descent data encoded by $\mathscr{G}_{\mathrm{st}}$.

\medskip
\noindent
\textbf{Summary.}
The sheaf $\mathscr{F}_K$ thus serves as a conceptual and technical bridge between
three complementary perspectives:

\begin{itemize}
    \item \textbf{Homotopical perspective:}
    $\mathscr{F}_K = \pi_0 \circ \mathscr{G}_{\mathrm{st}}$ extracts the zeroth-order
    information from the stable spectral stack, linking computable invariants to
    the higher categorical structures described in Proposition~\ref{prop:stabilization-K}

    \item \textbf{Algebraic perspective:}
    As a sheaf of abelian groups, $\mathscr{F}_K$ admits explicit cohomology
    groups $H^*(\mathcal{C}; \mathscr{F}_K)$, computable via \v{C}ech or descent methods.
    These groups measure failures of gluing for $K_0$-classes across commutative
    contexts.

    \item \textbf{Physical perspective:}
    For each commutative context $C$, the group $\mathscr{F}_K(C) = K_0(C)$
    classifies stable equivalence classes of projections, corresponding to
    quantum yes/no observables measurable within $C$. The resulting cohomology
    groups capture topological charges that cannot be localized to any single
    context, reflecting genuine forms of quantum contextuality.
\end{itemize}

These cohomology groups provide concrete, computable invariants that complement
the homotopical invariants of
Corollary~\ref{cor:homotopical-properties}. In particular:
\begin{itemize}
    \item $H^1(\mathcal{C}; \mathscr{F}_K)$ detects primary obstructions to gluing
    local $K_0$-data, corresponding to Kochen--Specker type contextuality;

    \item higher groups $H^{n \ge 2}(\mathcal{C}; \mathscr{F}_K)$ measure higher
    coherence obstructions arising from multi-way overlaps of measurement
    contexts.
\end{itemize}

In the following sections, we use these invariants to formulate precise
obstruction theorems for simultaneous diagonalization and to classify distinct
degrees of non-commutativity via their $K$-theoretic signatures.

\subsection{Sheaf Cohomology and Homotopical Invariants}

In this subsection, we systematically construct invariants that quantify
non-commutativity through sheaf-theoretic and homotopical lenses. These
invariants form a hierarchy, with successive levels capturing increasingly
refined obstructions to gluing local commutative descriptions into a global
non-commutative structure.

\medskip

\noindent
\textbf{1. Algebraic Invariants via Sheaf Cohomology.}

The sheaf $\mathscr{F}_K$ provides algebraic 
invariants through its cohomology groups $H^n(\mathcal{C}; \mathscr{F}_K)$, as defined 
in Definition~\ref{def:sheaf-cohomology}. These groups measure obstructions to 
extending local $K_0$-classes to global ones:

\begin{itemize}
    \item $H^0(\mathcal{C}; \mathscr{F}_K) \cong \Gamma(\mathcal{C}, \mathscr{F}_K)$ 
          consists of globally defined $K_0$-classes.
    \item $H^1(\mathcal{C}; \mathscr{F}_K)$ classifies families of local $K_0$-classes 
          that are compatible on overlaps but do not glue to a global section.
    \item $H^{n \ge 2}(\mathcal{C}; \mathscr{F}_K)$ encode higher coherence conditions 
          for compatibility across $n$-fold intersections of contexts.
\end{itemize}

These algebraic invariants capture the \emph{decategorified shadow} of the full
categorical structure encoded in the spectral presheaf. To access higher
coherence information, we introduce homotopical invariants.

\medskip

\noindent
\textbf{2. Homotopical Invariants from the Categorical Spectral Presheaf.}

Let
\[
\mathscr{G}_{\mathrm{cat}}: \mathcal{C}^{\mathrm{op}} \longrightarrow \mathbf{Cat}
\]
be the categorical spectral presheaf from Theorem~\ref{thm:spectral_sheaf}(i).
Applying the \emph{nerve functor}
\[
N: \mathbf{Cat} \longrightarrow \mathbf{sSet}
\]
objectwise yields a presheaf of simplicial sets
\[
N\mathscr{G}_{\mathrm{cat}} : \mathcal{C}^{\mathrm{op}} \longrightarrow \mathbf{sSet}, 
\quad C \mapsto N(\mathscr{G}_{\mathrm{cat}}(C)).
\]

For each $C \in \mathcal{C}$, $N(\mathscr{G}_{\mathrm{cat}}(C))$ is the simplicial
set whose $k$-simplices are sequences of $k$ composable morphisms in the category
$\mathscr{G}_{\mathrm{cat}}(C)$ of spectral measures on $\Sigma(C)$. Its geometric
realization
\[
|N(\mathscr{G}_{\mathrm{cat}}(C))| =: B\mathscr{G}_{\mathrm{cat}}(C)
\]
is the \emph{classifying space} of the category.

\medskip

\noindent
\textbf{3. Extracting Homotopical Invariants.}

We obtain homotopical invariants from $N\mathscr{G}_{\mathrm{cat}}$ via two complementary approaches:

\begin{enumerate}
    \item \emph{Homotopy limit:} Apply the homotopy limit functor
          \[
          \operatorname{holim}_{\mathcal{C}}: \mathbf{sSet}^{\mathcal{C}^{\mathrm{op}}} 
          \longrightarrow \mathbf{sSet}
          \]
          to $N\mathscr{G}_{\mathrm{cat}}$ and take homotopy groups
          \[
          \pi_n\!\left(\operatorname{holim}_{\mathcal{C}} N\mathscr{G}_{\mathrm{cat}}\right), 
          \quad n \ge 0.
          \]
          These groups detect coherent families of spectral measures across contexts, 
          with non-vanishing higher homotopy groups indicating higher categorical obstructions.
    
    \item \emph{Stabilization:} Apply the suspension spectrum functor $\Sigma^\infty_+$ 
          pointwise to obtain a presheaf of spectra
          \[
          \mathscr{G}_{\mathrm{st}} := \Sigma^\infty_+ \circ N\mathscr{G}_{\mathrm{cat}}:
          \mathcal{C}^{\mathrm{op}} \longrightarrow \mathbf{Sp}.
          \]
          Its hypercohomology
          \[
          \mathbb{H}^n(\mathcal{C}; \mathscr{G}_{\mathrm{st}}) :=
          \pi_{-n}\!\left(\operatorname{holim}_{\mathcal{C}} \mathscr{G}_{\mathrm{st}}\right)
          \]
          (Theorem~\ref{thm:spectral_sheaf}(iii)) provides stable invariants that refine 
          the algebraic obstructions $H^*(\mathcal{C}; \mathscr{F}_K)$.
\end{enumerate}

\medskip

\noindent
\textbf{4. Hierarchy of Invariants.}

The invariants fit into a coherent hierarchy that connects different levels of structure:

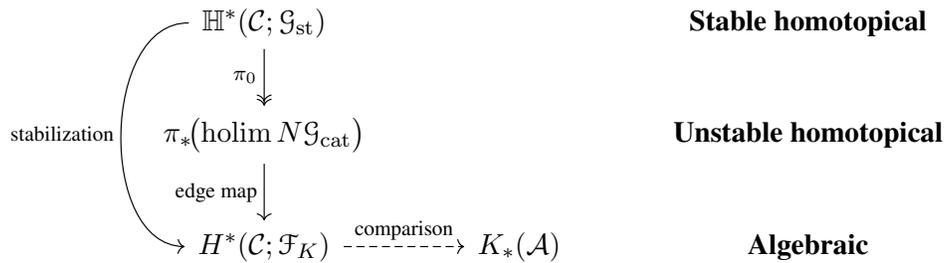
\[
\begin{tikzcd}[column sep=3em, row sep=2em]
\mathbb{H}^{*}(\mathcal{C};\mathscr{G}_{\mathrm{st}}) 
\arrow[d, two heads, "\pi_0"'] 
\arrow[dd, bend right=90, "\text{stabilization}"'] 
& & \text{\textbf{Stable homotopical}} \\
\pi_{*}\!\bigl(\operatorname{holim}N\mathscr{G}_{\mathrm{cat}}\bigr) 
\arrow[d, "\text{edge map}"'] 
& & \text{\textbf{Unstable homotopical}} \\
H^{*}(\mathcal{C};\mathscr{F}_K) \arrow[r, dashed, "\text{comparison}"] 
& K_{*}(\mathcal{A}) 
& \text{\textbf{Algebraic}}
\end{tikzcd}
\]

\medskip

\noindent
\textbf{Hierarchy levels:}

\begin{itemize}
    \item \textbf{Top (stable homotopical):} $\mathbb{H}^{*}(\mathcal{C};\mathscr{G}_{\mathrm{st}})$ 
          captures the complete stable homotopy information of the spectral stack.
    
    \item \textbf{Middle (unstable homotopical):} $\pi_{*}(\operatorname{holim}N\mathscr{G}_{\mathrm{cat}})$ 
          captures unstable homotopy information, related to the stable version by 
          stabilization (Freudenthal suspension theorem).
    
    \item \textbf{Bottom (algebraic):} $H^{*}(\mathcal{C};\mathscr{F}_K)$ (sheaf cohomology) 
          and $K_{*}(\mathcal{A})$ (global $K$-theory) are connected by a comparison map. 
          These are the most concrete, computable invariants.
\end{itemize}

\medskip

\noindent
\textbf{Relations:}
\begin{itemize}
    \item $\pi_0: \mathbb{H}^{*} \to H^{*}$ extracts the zeroth homotopy group, 
          recovering sheaf cohomology from stable homotopy.
    
    \item Stabilization: Converts unstable to stable homotopy information.
    
    \item Edge map: Comes from the descent spectral sequence, relating 
          homotopical to algebraic obstructions.
    
    \item Comparison: Relates locally defined sheaf cohomology to global 
          $K$-theory of the algebra.
\end{itemize}

\begin{remark}[Components of the Hierarchy Diagram]
The hierarchy diagram illustrates the relationships between algebraic, unstable homotopical, and stable homotopical invariants. To make this connection precise, one can explicitly indicate the functorial flow
\[
\mathscr{G}_{\mathrm{cat}} \xrightarrow{N} N\mathscr{G}_{\mathrm{cat}} \xrightarrow{\Sigma^\infty_+} \mathscr{G}_{\mathrm{st}},
\]
where:
\begin{itemize}
    \item $\mathscr{G}_{\mathrm{cat}}$ is the category-valued spectral presheaf;
    \item $N\mathscr{G}_{\mathrm{cat}}$ is the objectwise nerve, giving a presheaf of simplicial sets (unstable homotopical information);
    \item $\mathscr{G}_{\mathrm{st}}$ is the stabilized presheaf of spectra, capturing the full stable homotopy data.
\end{itemize}
Including this functorial flow in the text ensures that the diagram aligns rigorously with the underlying constructions and highlights the mathematical significance of the hierarchy.
\end{remark}

\begin{definition}[Spectral Hypercohomology]\label{def:spectral-hypercohomology}
The \emph{spectral hypercohomology groups} of the pair $(T,S)$ are the homotopy 
groups of the homotopy limit of the stable spectral presheaf:
\[
\mathbb{H}^n(\mathcal{C}; \mathscr{G}_{\mathrm{st}}) := 
\pi_{-n}\!\Bigl(\operatorname{holim}_{\mathcal{C}} \mathscr{G}_{\mathrm{st}}\Bigr), 
\quad n \in \mathbb{Z},
\]
where $\mathscr{G}_{\mathrm{st}}: \mathcal{C}^{\mathrm{op}} \to \mathbf{Sp}$ is 
the presheaf from Theorem~\ref{thm:spectral_sheaf}(iii), with 
$\mathscr{G}_{\mathrm{st}}(C) = K(C)$, the $K$-theory spectrum of $C$.

\textbf{Alternative construction:} These groups can also be obtained by:
\begin{enumerate}
    \item Taking the nerve $N\mathscr{G}_{\mathrm{cat}}$ of the categorical spectral presheaf,
    \item Stabilizing to spectra: $\Sigma^\infty_+ N\mathscr{G}_{\mathrm{cat}}$,
    \item Taking the homotopy limit and its homotopy groups.
\end{enumerate}
The equivalence follows from the identification $K(C) \simeq \Sigma^\infty_+ \Sigma(C)$ 
for commutative $C^*$-algebras and the natural map $\Sigma(C) \to B\mathscr{G}_{\mathrm{cat}}(C)$.

For $n \geq 0$, there are \emph{unstable hypercohomology groups}:
\[
\mathbb{H}^n(\mathcal{C}; \mathscr{G}_{\mathrm{cat}}) := 
\pi_n\!\Bigl(\operatorname{holim}_{\mathcal{C}} N\mathscr{G}_{\mathrm{cat}}\Bigr),
\]
which agree with $\mathbb{H}^n(\mathcal{C}; \mathscr{G}_{\mathrm{st}})$ for $n=0,1,2$ 
under mild connectivity assumptions.
\end{definition}

\begin{remark}[From Sheaf Cohomology to Hypercohomology]\label{rem:comparison-sheaf-hypercohomology}
There is a natural edge homomorphism
\[
e^p: H^p(\mathcal{C}; \mathscr{F}_K) \longrightarrow \mathbb{H}^p(\mathcal{C}; \mathscr{G}_{\mathrm{st}}), 
\quad p \ge 0,
\]
arising from the \emph{Bousfield--Kan descent spectral sequence} for the spectrum-valued presheaf
$\mathscr{G}_{\mathrm{st}} := \Sigma^\infty_+ B \mathscr{G}$:
\[
E_2^{p,q} = H^p\bigl(\mathcal{C}; \pi_{-q}(\mathscr{G}_{\mathrm{st}})\bigr)
\;\;\Longrightarrow\;\; \mathbb{H}^{p+q}(\mathcal{C}; \mathscr{G}_{\mathrm{st}}),
\]
where $\pi_0(\mathscr{G}_{\mathrm{st}}) \cong \mathscr{F}_K$ and $\pi_{-q}(\mathscr{G}_{\mathrm{st}})(C) \cong K_q(C)$ for $q \ge 0$. 

\medskip

\noindent
\textbf{Properties of the edge homomorphism:}
\begin{itemize}
    \item For $p=0$, it is an isomorphism:
    \[
    H^0(\mathcal{C}; \mathscr{F}_K) \;\cong\; \mathbb{H}^0(\mathcal{C}; \mathscr{G}_{\mathrm{st}}).
    \]
    \item For $p=1$, it is injective under mild connectivity assumptions (e.g., when $K_1(C)=0$ for all $C \in \mathcal{C}$), but may fail to be surjective.
    \item For $p \ge 2$, it is generally neither injective nor surjective, as non-trivial differentials $d_2^{p,0}: E_2^{p,0} \to E_2^{p+2,-1}$ and extension problems can occur.
\end{itemize}

\medskip

\noindent
\textbf{Interpretation.} 
The sheaf cohomology groups $H^p(\mathcal{C}; \mathscr{F}_K)$ detect purely $K_0$-theoretic obstructions to gluing local spectral data. 
In contrast, the hypercohomology groups $\mathbb{H}^p(\mathcal{C}; \mathscr{G}_{\mathrm{st}})$ encode the full $K$-theoretic information, including higher $K_q$ contributions and their interactions via spectral sequence differentials. 
Thus, while $H^p$ tells us whether local $K_0$-classes can be glued, $\mathbb{H}^p$ records both the existence of obstructions and their detailed homotopical structure.
\end{remark}

\begin{theorem}[Homotopical Characterization of Non-Commutativity]\label{thm:homotopical-characterization}
Let $T, S \in \mathcal{A}$ be operators in a $C^*$-algebra $\mathcal{A}$, and let 
\[
\mathcal{C} := \{\, C \subseteq \mathcal{A} \mid C \text{ is a commutative } C^*\text{-subalgebra containing } T, S \,\}
\] 
be the poset of commutative subalgebras containing both $T$ and $S$, ordered by inclusion. Let 
$\mathscr{G}_{\mathrm{cat}}$ be the categorical spectral presheaf on $\mathcal{C}$, and let
\(\mathscr{G}_{\mathrm{st}} := \Sigma^\infty_+ N\mathscr{G}_{\mathrm{cat}}\) denote its stabilization.

\begin{enumerate}[label=(\alph*)]
    \item \textbf{Vanishing criterion:} If $[T,S] = 0$, then
    \[
    \mathbb{H}^n(\mathcal{C}; \mathscr{G}_{\mathrm{st}}) = 0 \quad \text{for all } n \neq 0.
    \]
    The only non-trivial invariant is
    \[
    \mathbb{H}^0(\mathcal{C}; \mathscr{G}_{\mathrm{st}}) \;\cong\; K_0(\mathcal{A}_{\mathrm{comm}}),
    \]
    where $\mathcal{A}_{\mathrm{comm}}$ is the commutative $C^*$-algebra generated by $T$ and $S$.
    
    \item \textbf{First homotopical obstruction:} 
    The group $\mathbb{H}^1(\mathcal{C}; \mathscr{G}_{\mathrm{st}})$ classifies non-contractible loops of spectral data. Non-trivial elements correspond to \emph{monodromy}: transporting a spectral measure around a cycle of contexts returns an isomorphic but non-identical measure. When interpreted as value assignments to quantum observables, this captures Kochen--Specker type contextuality.
    
    \item \textbf{Higher homotopical obstructions:} For $n \ge 2$, $\mathbb{H}^n(\mathcal{C}; \mathscr{G}_{\mathrm{st}})$ measures higher coherence failures:
    \begin{itemize}
        \item $\mathbb{H}^2$ detects obstructions to consistent assembly of automorphisms on triple intersections (\emph{gerbes}).
        \item $\mathbb{H}^3$ detects obstructions on quadruple intersections (\emph{2-gerbes}).
        \item More generally, $\mathbb{H}^n$ for $n>3$ captures higher-order coherence obstructions in gluing spectral measures across $n$-fold intersections.
    \end{itemize}
    These groups provide a graded measure of the homotopical complexity of non-commutativity.
\end{enumerate}
\end{theorem}

\begin{proof}
\medskip
\noindent
\textbf{(a) Vanishing criterion.} Assume $[T,S] = 0$. Then $T$ and $S$ generate a commutative $C^*$-algebra 
$\mathcal{A}_{\mathrm{comm}}$. In this case, the spectral presheaf $\mathscr{G}_{\mathrm{cat}}$ is essentially constant:
for any $C \in \mathcal{C}$, there is a canonical equivalence
\[
\mathscr{G}_{\mathrm{cat}}(C) \simeq \mathscr{G}_{\mathrm{cat}}(\mathcal{A}_{\mathrm{comm}}).
\]

Since $\mathcal{C}$ is cofiltered (all finite diagrams have cones given by intersections), the homotopy limit over $\mathcal{C}$ is equivalent to the value on any sufficiently large stage by cofinality:
\[
\operatorname{holim}_{C \in \mathcal{C}} B_{\mathscr{G}}(C) \;\simeq\; B_{\mathscr{G}}(\mathcal{A}_{\mathrm{comm}}).
\]

$B_{\mathscr{G}}(\mathcal{A}_{\mathrm{comm}})$ is the classifying space of a groupoid, hence
\(\pi_n = 0\) for $n \ge 2$ and \(\pi_1\) is the automorphism group of a spectral measure. After stabilization,
\[
\mathbb{H}^n(\mathcal{C}; \mathscr{G}_{\mathrm{st}}) = \pi_{-n} \Sigma^\infty B_{\mathscr{G}}(\mathcal{A}_{\mathrm{comm}}) = 0 \quad \text{for all } n \neq 0.
\]

The identification \(\mathbb{H}^0 \cong K_0(\mathcal{A}_{\mathrm{comm}})\) follows from Proposition~\ref{prop:stabilization-K}, as $K_0$ for commutative algebras equals the Grothendieck group of projections.

\medskip
\noindent
\textbf{(b) First homotopical obstruction.} Elements of $\mathbb{H}^1(\mathcal{C}; \mathscr{G}_{\mathrm{st}})$ correspond to loops in the homotopy limit 
\(\operatorname{holim}_{\mathcal{C}} B_{\mathscr{G}}\). Fix a basepoint spectral measure $X_0$ on some maximal commutative subalgebra $C_0$. A loop consists of
\[
(C_0, C_1, \dots, C_k=C_0) \quad \text{with isomorphisms } f_i: X_{i-1}|_{C_{i-1}\cap C_i} \xrightarrow{\sim} X_i.
\]

The composite
\[
f_k \circ \cdots \circ f_1 : X_0|_{\bigcap_i C_i} \longrightarrow X_0|_{\bigcap_i C_i}
\]
gives an automorphism of $X_0$. The loop is contractible iff this automorphism is homotopic to the identity. Non-zero elements correspond to cycles where $[T,S] \neq 0$, encoding Kochen--Specker contextuality.

\medskip
\noindent
\textbf{(c) Higher homotopical obstructions.} For $n \ge 2$, $\mathbb{H}^n(\mathcal{C}; \mathscr{G}_{\mathrm{st}})$ encodes $(n-1)$-simplices in the homotopy limit satisfying higher coherence:

\begin{itemize}
    \item $\mathbb{H}^2$: Automorphisms on triple intersections; non-trivial elements represent \emph{gerbes}.
    \item $\mathbb{H}^3$: Coherence of these automorphisms on quadruples; non-trivial elements represent \emph{2-gerbes}.
    \item $n>3$: Measures failure of coherent extension on $n$-fold intersections.
\end{itemize}

Thus, non-commutativity introduces topological twists in spectral data across contexts, with the minimal $n$ such that $\mathbb{H}^n \neq 0$ indicating the lowest-order homotopical obstruction.

\end{proof}

The hierarchy of invariants provides a structured framework to describe quantum contextuality at multiple levels of refinement:

\begin{enumerate}
    \item \textbf{Algebraic contextuality:} $H^1(\mathcal{C}; \mathscr{F}_K) \neq 0$ signals that certain K-theoretic charges (e.g., topological numbers associated with projections) cannot be consistently assigned across all measurement contexts.
    
    \item \textbf{Homotopical contextuality:} $\mathbb{H}^1(\mathcal{C}; \mathscr{G}_{\mathrm{cat}}) \neq 0$ captures not only the existence of inconsistencies but also their topological nature, represented as non-contractible loops in the space of spectral data.
    
    \item \textbf{Higher contextuality:} $\mathbb{H}^n(\mathcal{C}; \mathscr{G}_{\mathrm{cat}}) \neq 0$ for $n \ge 2$ detects failures of higher coherence—situations where pairwise consistency holds, but consistency breaks down on triples or larger collections of contexts.
\end{enumerate}

\begin{remark}[Hierarchy of Non-Commutativity]
We organize the invariants into a hierarchy reflecting increasing refinement:
\[
\begin{array}{ccl}
\text{Algebraic level} & : & H^n(\mathcal{C}; \mathscr{F}_K) \quad (n \ge 0),\\[1mm]
\text{Homotopical level} & : & \mathbb{H}^n(\mathcal{C}; \mathscr{G}_{\mathrm{cat}}) \quad (n \in \mathbb{Z}),\\[1mm]
\text{Categorical level} & : & B_{\mathscr{G}}: \mathcal{C}^{\mathrm{op}} \to \mathbf{sSet}.
\end{array}
\]
Each level offers a complementary viewpoint on the same fundamental phenomenon: the failure of local commutative descriptions to assemble into a global non-commutative structure. The vanishing or non-vanishing of these invariants, along with their specific values, quantifies the degree and type of non-commutativity.
\end{remark}

\begin{remark}[Resource-Theoretic Interpretation]
This framework has a natural connection to quantum resource theories:
\begin{itemize}
    \item Non-vanishing $H^1$ indicates minimal resources required for quantum advantage in certain computational models (e.g., measurement-based quantum computation).
    \item Non-vanishing $\mathbb{H}^2$ captures resources necessary for protocols requiring higher-order coherence or topological protection.
    \item More generally, the hierarchy provides a graded classification of contextual resources, extending beyond the binary notion of ``contextual vs. non-contextual.''
\end{itemize}
\end{remark}

\begin{remark}[Computational Considerations]
While $\mathbb{H}^n(\mathcal{C}; \mathscr{G}_{\mathrm{cat}})$ encodes richer topological and homotopical information than $H^n(\mathcal{C}; \mathscr{F}_K)$, the latter are often easier to compute. For many applications, algebraic invariants suffice to detect the presence of contextuality. Homotopical invariants become essential when phenomena depend on the precise homotopy type of obstructions, for example:
\begin{itemize}
    \item Distinguishing different topological phases of matter,
    \item Classifying anomalies in quantum field theories,
    \item Analyzing stability properties under continuous deformations.
\end{itemize}
Choosing the appropriate invariant involves a trade-off between computational accessibility and informational completeness.
\end{remark}

\paragraph{Hierarchy of Invariants.}

\begin{definition}[Homotopy Groups $\pi_*$ of the Homotopy Limit]
Let $B_{\mathscr{G}}: \mathcal{C}^{\mathrm{op}} \to \mathbf{sSet}$ be the classifying space prestack of a categorical spectral presheaf $\mathscr{G}_{\mathrm{cat}}$, where each $B_{\mathscr{G}}(C)$ is a Kan complex. Let
\[
\operatorname{holim}_{C \in \mathcal{C}} B_{\mathscr{G}}(C)
\]
denote its homotopy limit. The homotopy groups $\pi_*$ of this homotopy limit are defined as follows:

\begin{enumerate}
    \item \textbf{Unstable homotopy groups:} 
    \[
    \pi_n\Big(\operatorname{holim}_{C \in \mathcal{C}} B_{\mathscr{G}}(C)\Big), \quad n \ge 0,
    \]
    measure loops and higher homotopy classes in the space of compatible local spectral data before stabilization. Negative homotopy groups are not defined in this unstable setting.

    \item \textbf{Stable homotopy groups:} After applying the infinite suspension functor $\Sigma^\infty$ to stabilize the classifying spaces into spectra,
    \[
    \pi_n^{\mathrm{st}}\Big(\Sigma^\infty \operatorname{holim}_{C \in \mathcal{C}} B_{\mathscr{G}}(C)\Big), \quad n \in \mathbb{Z},
    \]
    the homotopy groups are defined for all integers $n$, including negative ones. Stabilization ensures that these groups are homotopy invariant and capture finer, stable information about obstructions to gluing local spectral data. Under mild connectivity assumptions, one can equivalently stabilize each $B_{\mathscr{G}}(C)$ before taking the homotopy limit.
\end{enumerate}

We denote the stable homotopy groups by
\[
\mathbb{H}^n(\mathcal{C}; \mathscr{G}_{\mathrm{cat}}) := \pi_n^{\mathrm{st}}\Big(\operatorname{holim}_{C \in \mathcal{C}} \Sigma^\infty B_{\mathscr{G}}(C)\Big), \quad n \in \mathbb{Z}.
\]
In the hierarchical diagram below, the notation $\pi_*$ refers generically to either $\pi_n$ (unstable, $n \ge 0$) or $\pi_n^{\mathrm{st}}$ (stable, $n \in \mathbb{Z}$), depending on whether stabilization has been applied.
\end{definition}

The following diagram illustrates the relationship between the various invariants:

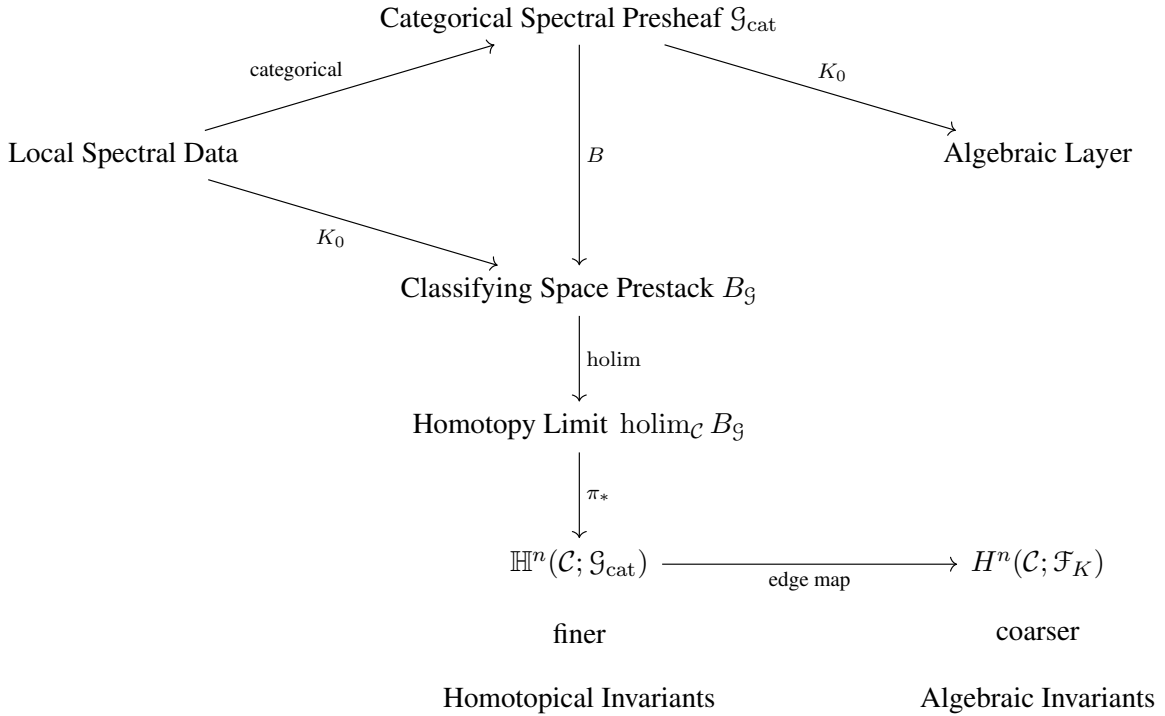
\[
\begin{tikzcd}[column sep=4em, row sep=3em]
& \text{Categorical Spectral Presheaf } \mathscr{G}_{\mathrm{cat}} \arrow[dd, "B"] \arrow[dr, "K_0"] & \\
\text{Local Spectral Data} \arrow[ur, "\text{categorical}"] \arrow[dr, "K_0"'] & & \text{Algebraic Layer} \\
& \text{Classifying Space Prestack } B_{\mathscr{G}} \arrow[d, "\operatorname{holim}"] & \\
& \text{Homotopy Limit } \operatorname{holim}_{\mathcal{C}} B_{\mathscr{G}} \arrow[d, "\pi_*"] & \\
& \mathbb{H}^n(\mathcal{C}; \mathscr{G}_{\mathrm{cat}}) \arrow[r, "\text{edge map}"'] & H^n(\mathcal{C}; \mathscr{F}_K) \\
& \text{Homotopical Invariants} \arrow[u, phantom, "\text{finer}"'] & \text{Algebraic Invariants} \arrow[u, phantom, "\text{coarser}"]
\end{tikzcd}
\]

\subsection{Sheaf Cohomology and Hypercohomology via Homotopy Limits}

This subsection introduces the cohomological invariants associated to the categorical spectral presheaf and explains how ordinary sheaf cohomology and hypercohomology arise naturally from homotopy limits of spectrum-valued presheaves.

\subsubsection{Functoriality and Morita Invariance}

\begin{proposition}[Functoriality and Morita Invariance of Categorical Spectral Presheaves]
\label{prop:functoriality-morita-correct}
Let $\mathcal A$ be a unital $C^*$-algebra, and let $\mathcal C(\mathcal A)$ denote the poset-category
whose objects are unital commutative $C^*$-subalgebras of $\mathcal A$ and whose morphisms are inclusions.

Assume that to each $C^*$-algebra $\mathcal A$ we associate a presheaf
\[
\mathscr G_{\mathcal A} \colon \mathcal C(\mathcal A)^{\mathrm{op}} \to \mathbf{Cat},
\]
satisfying the following axioms:

\begin{enumerate}
\item[(A1)] (Restriction functoriality)
For each inclusion $i\colon C' \hookrightarrow C$ in $\mathcal C(\mathcal A)$, there is a restriction functor
\[
\mathscr G_{\mathcal A}(i)\colon \mathscr G_{\mathcal A}(C) \to \mathscr G_{\mathcal A}(C'),
\]
and these assignments are strictly functorial.

\item[(A2)] (Functoriality under $*$-homomorphisms)
For any unital $*$-homomorphism $f\colon \mathcal A \to \mathcal B$ and any $C\in\mathcal C(\mathcal A)$,
there exists a functor
\[
f_C^* \colon \mathscr G_{\mathcal B}(\overline{f(C)}) \longrightarrow \mathscr G_{\mathcal A}(C),
\]
natural in $C$ with respect to inclusions.

\item[(A3)] (Morita invariance on commutative algebras)
If $C$ and $D$ are commutative $C^*$-algebras that are (strongly) Morita equivalent, then
\[
\mathscr G(C) \simeq \mathscr G(D)
\]
as categories, functorially in Morita equivalences.
\end{enumerate}

Then the following hold.

\begin{enumerate}
\item[(i)] The assignment $\mathcal A \mapsto \mathscr G_{\mathcal A}$ is contravariantly functorial
with respect to unital $*$-homomorphisms.

\item[(ii)] If $\mathcal A$ and $\mathcal B$ are strongly Morita equivalent $C^*$-algebras, then
$\mathscr G_{\mathcal A}$ and $\mathscr G_{\mathcal B}$ are equivalent as stacks over their respective
context categories (after identifying the sites via the Morita equivalence).

\item[(iii)] Consequently, the associated simplicial presheaves obtained by objectwise classifying spaces
have equivalent homotopy limits, and hence define isomorphic hypercohomology and sheaf cohomology invariants.
\end{enumerate}
\end{proposition}

\begin{proof}
We proceed in two steps.

\medskip
\noindent
\textbf{Step 1: Functoriality under $*$-homomorphisms.}

Let $f\colon \mathcal A \to \mathcal B$ be a unital $*$-homomorphism.
Define a covariant functor
\[
f_* \colon \mathcal C(\mathcal A) \longrightarrow \mathcal C(\mathcal B),
\qquad
C \longmapsto \overline{f(C)}.
\]
Since $f$ preserves multiplication and involution, $\overline{f(C)}$ is a unital commutative
$C^*$-subalgebra of $\mathcal B$, and inclusions are preserved:
if $C' \subset C$, then $\overline{f(C')} \subset \overline{f(C)}$.

By assumption (A2), for each $C \in \mathcal C(\mathcal A)$ there exists a functor
\[
f_C^* \colon \mathscr G_{\mathcal B}(f_*(C)) \longrightarrow \mathscr G_{\mathcal A}(C),
\]
natural with respect to inclusions.
Explicitly, for each inclusion $i\colon C' \hookrightarrow C$, the diagram
\[
\begin{tikzcd}
\mathscr G_{\mathcal B}(f_*(C)) \arrow[r, "\mathscr G_{\mathcal B}(f_*(i))"]
\arrow[d, "f_C^*"']
&
\mathscr G_{\mathcal B}(f_*(C')) \arrow[d, "f_{C'}^*"]
\\
\mathscr G_{\mathcal A}(C) \arrow[r, "\mathscr G_{\mathcal A}(i)"']
&
\mathscr G_{\mathcal A}(C')
\end{tikzcd}
\]
commutes.

Thus the family $\{f_C^*\}_{C}$ defines a natural transformation
\[
f^* \colon \mathscr G_{\mathcal B} \circ f_* \Rightarrow \mathscr G_{\mathcal A}
\]
of presheaves on $\mathcal C(\mathcal A)$.
This establishes contravariant functoriality of the assignment
$\mathcal A \mapsto \mathscr G_{\mathcal A}$.

Applying the classifying space functor $B\colon \mathbf{Cat} \to \mathbf{sSet}$
objectwise yields a natural transformation
\[
B(\mathscr G_{\mathcal B}) \circ f_* \Rightarrow B(\mathscr G_{\mathcal A})
\]
of simplicial presheaves.
Standard properties of homotopy limits then imply functoriality of the associated
homotopy-invariant constructions.

\medskip
\noindent
\textbf{Step 2: Morita invariance.}

Suppose that $\mathcal A$ and $\mathcal B$ are strongly Morita equivalent
$C^*$-algebras, implemented by an $\mathcal A$--$\mathcal B$ imprimitivity bimodule $E$.
Such a bimodule induces an equivalence between the categories of Hilbert modules
over $\mathcal A$ and $\mathcal B$, and hence an equivalence in the Morita bicategory
of $C^*$-algebras.

Let $C \subset \mathcal A$ be a commutative context.
The restriction of $E$ along $C \hookrightarrow \mathcal A$ yields a Morita equivalence
between $C$ and a commutative $C^*$-subalgebra $D \subset \mathcal B$
(unique up to equivalence).
By assumption (A3), this Morita equivalence induces an equivalence of categories
\[
\mathscr G_{\mathcal A}(C) \simeq \mathscr G_{\mathcal B}(D).
\]

These equivalences are compatible with inclusions of contexts, since restriction of
the imprimitivity bimodule along $C' \subset C$ yields the corresponding equivalence
for the smaller context.
Hence the family of equivalences assembles into an equivalence of presheaves
(up to coherent isomorphism), i.e.\ an equivalence of stacks over the corresponding sites.

Finally, equivalent stacks have weakly equivalent objectwise classifying space
prestacks, and therefore equivalent homotopy limits:
\[
\operatorname{holim}_{\mathcal C(\mathcal A)} B(\mathscr G_{\mathcal A})
\;\simeq\;
\operatorname{holim}_{\mathcal C(\mathcal B)} B(\mathscr G_{\mathcal B}).
\]
Taking homotopy groups yields canonical isomorphisms on hypercohomology groups,
and $\pi_0$ recovers the corresponding statement for sheaf cohomology.

This completes the proof.
\end{proof}

Proposition~\ref{prop:Contravariant functoriality and Morita invariance of the categorical spectral presheaf} is given to show that the categorical spectral presheaf behaves functorially under $*$-homomorphisms and is invariant under Morita equivalence, ensuring that the associated spectral and cohomological invariants depend only on the intrinsic noncommutative geometry of a $C^*$-algebra. As a result, these invariants are well defined and robust under the standard equivalences used in operator algebras and noncommutative geometry.

\begin{proposition}[Contravariant functoriality and Morita invariance of the categorical spectral presheaf]\label{prop:Contravariant functoriality and Morita invariance of the categorical spectral presheaf}
Let $\mathcal A$ be a unital $C^*$-algebra, and let
\[
\mathscr G_{\mathcal A} \colon \mathcal C(\mathcal A)^{\mathrm{op}} \longrightarrow \mathbf{Cat}
\]
denote the categorical spectral presheaf, assigning to each unital commutative
$C^*$-subalgebra $C \subseteq \mathcal A$ the category $\mathscr G_{\mathcal A}(C)$
of spectral data on its Gelfand spectrum $\Sigma_C$ (e.g.\ spectral measures,
projection-valued measures, or equivalent spectral objects), and to each inclusion
$C' \subseteq C$ the corresponding restriction functor.

\begin{enumerate}
\item[(i)] For any unital $*$-homomorphism $f \colon \mathcal A \to \mathcal B$, there exists a canonical natural transformation
\[
\eta_f \colon
\mathscr G_{\mathcal B} \circ f_*^{\mathrm{op}}
\;\Longrightarrow\;
\mathscr G_{\mathcal A},
\]
where $f_* \colon \mathcal C(\mathcal A) \to \mathcal C(\mathcal B)$ sends
$C$ to $\overline{f(C)}$.

\item[(ii)] If $\mathcal A$ and $\mathcal B$ are Morita equivalent (equivalently,
stably isomorphic), then after stackification with respect to the canonical
Grothendieck topology on $\mathcal C(\mathcal A)$ and $\mathcal C(\mathcal B)$,
the resulting stacks associated to $\mathscr G_{\mathcal A}$ and
$\mathscr G_{\mathcal B}$ are equivalent. In particular, all stack-level
cohomological invariants derived from $\mathscr G_{\mathcal A}$ are Morita invariant.
\end{enumerate}
\end{proposition}

\begin{proof}
\textbf{(i) Functoriality.}
Let $f \colon \mathcal A \to \mathcal B$ be a unital $*$-homomorphism and
$C \in \mathcal C(\mathcal A)$. The restricted map
$f|_C \colon C \to \overline{f(C)}$ is a unital $*$-homomorphism between commutative
$C^*$-algebras and therefore induces a continuous map of Gelfand spectra
\[
\widehat{f|_C} \colon \Sigma_{\overline{f(C)}} \longrightarrow \Sigma_C,
\]
contravariantly.

By definition of the presheaf $\mathscr G$, any spectral datum on
$\Sigma_{\overline{f(C)}}$ can be pulled back along $\widehat{f|_C}$ to a spectral
datum on $\Sigma_C$. This construction defines a functor
\[
\eta_f(C) \colon
\mathscr G_{\mathcal B}(\overline{f(C)})
\longrightarrow
\mathscr G_{\mathcal A}(C).
\]
Naturality with respect to inclusions $C' \subseteq C$ follows from the functoriality
of pullback along continuous maps and the commutativity of the induced diagrams of
spectra. Hence the family $\{\eta_f(C)\}_C$ defines a natural transformation
\[
\eta_f \colon \mathscr G_{\mathcal B} \circ f_*^{\mathrm{op}}
\Rightarrow \mathscr G_{\mathcal A}.
\]

\medskip
\textbf{(ii) Morita invariance.}
Suppose $\mathcal A$ and $\mathcal B$ are Morita equivalent. By the
Brown--Green--Rieffel theorem, this is equivalent to the existence of a stable
isomorphism
\[
\mathcal A \otimes \mathcal K \cong \mathcal B \otimes \mathcal K,
\]
where $\mathcal K$ denotes the algebra of compact operators on a separable Hilbert
space.

Stabilization induces a correspondence between commutative subalgebras of
$\mathcal A$ and those of $\mathcal A \otimes \mathcal K$, which is not an equivalence
on the nose but becomes one after passing to the associated Grothendieck sites and
stackifying. Under this process, the presheaves $\mathscr G_{\mathcal A}$ and
$\mathscr G_{\mathcal A \otimes \mathcal K}$ determine equivalent stacks, and the same
holds for $\mathcal B$.

Transporting along the stable isomorphism yields an equivalence between the stacks
associated to $\mathscr G_{\mathcal A}$ and $\mathscr G_{\mathcal B}$. Consequently,
any stack-level invariant constructed from $\mathscr G_{\mathcal A}$, such as sheaf
cohomology or hypercohomology, is invariant under Morita equivalence.
\end{proof}

\subsubsection{Basic Properties}

\begin{definition}[Hypercohomology of Connective Spectral Presheaves]
Let $\mathcal C$ be a small category and 
$B_{\mathscr G} : \mathcal C^{\mathrm{op}} \to \mathbf{Sp}_{\ge 0}$ 
a presheaf of \emph{connective spectra} (e.g., obtained from a categorical spectral presheaf by applying algebraic $K$-theory objectwise).

The \emph{hypercohomology spectrum} of $B_{\mathscr G}$ is the derived global sections:
\[
\mathbf{R}\Gamma(\mathcal C, B_{\mathscr G}) := 
\operatorname{holim}_{C \in \mathcal C^{\mathrm{op}}} B_{\mathscr G}(C) \in \mathbf{Sp}_{\ge 0}.
\]

The \emph{\(n\)-th hypercohomology group} for $n \ge 0$ is defined by
\[
\mathbb H^n(\mathcal C; B_{\mathscr G}) := 
\pi_n \bigl( \mathbf{R}\Gamma(\mathcal C, B_{\mathscr G}) \bigr) 
= \pi_n \Bigl( \operatorname{holim}_{C \in \mathcal C^{\mathrm{op}}} B_{\mathscr G}(C) \Bigr).
\]
\end{definition}

Following theorem establishes fundamental properties of noncommutative invariants (sheaf and hypercohomology of $K$-theory or spectral presheaves) analogous to classical algebraic topology. It ensures that these invariants are functorial, satisfy Mayer–Vietoris sequences, and reduce correctly to classical cohomology in the commutative case, providing consistency and computability.

\begin{theorem}[Functoriality, Mayer--Vietoris, and Commutative Reduction]\label{thm:complete-properties}
Let $\mathcal{A}$ be a separable $C^*$-algebra, and let $\mathcal{C}(\mathcal{A})$ denote its context category: the poset of separable unital commutative $C^*$-subalgebras of $\mathcal{A}$, ordered by inclusion. Consider:

\begin{enumerate}[label=(\roman*)]
    \item The $K$-theory \emph{presheaf of abelian groups} $\mathscr{F}_K^{\mathcal{A}} : \mathcal{C}(\mathcal{A})^{\mathrm{op}} \to \mathbf{Ab}$,
    defined by $\mathscr{F}_K^{\mathcal{A}}(C) = K_0(C) \oplus K_1(C)$ (or separately as $\mathscr{F}_K^{0,\mathcal{A}}(C) = K_0(C)$, $\mathscr{F}_K^{1,\mathcal{A}}(C) = K_1(C)$).
    
    \item The $K$-theory \emph{presheaf of spectra} $\mathbb{K}^{\mathcal{A}} : \mathcal{C}(\mathcal{A})^{\mathrm{op}} \to \mathbf{Sp}$,
    where $\mathbb{K}^{\mathcal{A}}(C)$ is the connective $K$-theory spectrum of $C$, with $\pi_n(\mathbb{K}^{\mathcal{A}}(C)) = K_{-n}(C)$ for $n \leq 0$ and $0$ for $n > 0$.
    
    \item A general \emph{presheaf of spectra} $B_{\mathscr{G}}^{\mathcal{A}} : \mathcal{C}(\mathcal{A})^{\mathrm{op}} \to \mathbf{Sp}$ associated to suitable coefficients $\mathscr{G}$.
\end{enumerate}

Endow $\mathcal{C}(\mathcal{A})$ with the Grothendieck topology where a family $\{C_i \hookrightarrow C\}$ is a cover if the $C^*$-subalgebra generated by $\bigcup_i C_i$ is dense in $C$.

Then:

\begin{enumerate}[label=(\roman*)]
    \item \textbf{Contravariant Functoriality:} For any unital $*$-homomorphism $\varphi: \mathcal{A} \to \mathcal{B}$, the induced functor
    \[
        \varphi_\sharp: \mathcal{C}(\mathcal{A}) \to \mathcal{C}(\mathcal{B}), \quad 
        C \mapsto \overline{\varphi(C)}
    \]
    is well-defined. There exist natural transformations
    \[
        \alpha_\varphi: \varphi_\sharp^* \mathscr{F}_K^{\mathcal{B}} \to \mathscr{F}_K^{\mathcal{A}},
        \qquad
        \beta_\varphi: \varphi_\sharp^* \mathbb{K}^{\mathcal{B}} \to \mathbb{K}^{\mathcal{A}},
        \qquad
        \gamma_\varphi: \varphi_\sharp^* B_{\mathscr{G}}^{\mathcal{B}} \to B_{\mathscr{G}}^{\mathcal{A}},
    \]
    defined objectwise by the functoriality of $K$-theory and the spectrum constructions applied to $\varphi|_C: C \to \varphi_\sharp(C)$. 
    
    These induce natural maps on cohomology and hypercohomology:
    \[
        \varphi^*: H^n(\mathcal{C}(\mathcal{B}); \mathscr{F}_K^{\mathcal{B}}) \to 
        H^n(\mathcal{C}(\mathcal{A}); \mathscr{F}_K^{\mathcal{A}}),
    \]
    \[
        \varphi^*: \mathbb{H}^n(\mathcal{C}(\mathcal{B}); \mathbb{K}^{\mathcal{B}}) \to 
        \mathbb{H}^n(\mathcal{C}(\mathcal{A}); \mathbb{K}^{\mathcal{A}}),
    \]
    \[
        \varphi^*: \mathbb{H}^n(\mathcal{C}(\mathcal{B}); B_{\mathscr{G}}^{\mathcal{B}}) \to 
        \mathbb{H}^n(\mathcal{C}(\mathcal{A}); B_{\mathscr{G}}^{\mathcal{A}}),
    \]
    where $\mathbb{H}^n(\mathcal{C}; F) := \pi_{-n}(\mathbf{R}\Gamma(\mathcal{C}; F))$ for a presheaf of spectra $F$. 
    
    These make (hyper)cohomology \emph{contravariantly} functorial. If $\varphi$ is an isomorphism, then all maps are isomorphisms.

    \item \textbf{Mayer--Vietoris Sequences:} Let $\mathcal{C}_1, \mathcal{C}_2 \subset \mathcal{C}(\mathcal{A})$ be \emph{downward-closed} subcategories (i.e., if $C \in \mathcal{C}_i$ and $C' \subseteq C$, then $C' \in \mathcal{C}_i$) whose union is cofinal in $\mathcal{C}(\mathcal{A})$ (meaning: for every $C \in \mathcal{C}(\mathcal{A})$, there exists $C' \subseteq C$ with $C' \in \mathcal{C}_1 \cup \mathcal{C}_2$). Write $\mathcal{C}_{12} = \mathcal{C}_1 \cap \mathcal{C}_2$. Then there are natural long exact sequences:
    \[
    \begin{tikzcd}[column sep=small]
    \cdots \arrow[r] & \mathbb{H}^n(\mathcal{C}(\mathcal{A}); \mathbb{K}^{\mathcal{A}}) \arrow[r] & 
    \mathbb{H}^n(\mathcal{C}_1; \mathbb{K}^{\mathcal{A}}|_{\mathcal{C}_1}) \oplus \mathbb{H}^n(\mathcal{C}_2; \mathbb{K}^{\mathcal{A}}|_{\mathcal{C}_2}) \arrow[r] &
    \mathbb{H}^n(\mathcal{C}_{12}; \mathbb{K}^{\mathcal{A}}|_{\mathcal{C}_{12}}) \arrow[r] &
    \mathbb{H}^{n+1}(\mathcal{C}(\mathcal{A}); \mathbb{K}^{\mathcal{A}}) \arrow[r] & \cdots
    \end{tikzcd}
    \]
    and similarly for $\mathbb{H}^n(\mathcal{C}(\mathcal{A}); B_{\mathscr{G}}^{\mathcal{A}})$. For sheaf cohomology with coefficients in $\mathscr{F}_K^{\mathcal{A}}$, there is a corresponding Mayer--Vietoris sequence only after applying the appropriate truncation or via a spectral sequence.

    \item \textbf{Commutative Reduction:} If $\mathcal{A}$ is separable and commutative, then $\mathcal{A}$ itself is a terminal object in $\mathcal{C}(\mathcal{A})$. Consequently:
    \[
        \mathbf{R}\Gamma(\mathcal{C}(\mathcal{A}); \mathbb{K}^{\mathcal{A}}) \simeq 
        \operatorname{holim}_{C \in \mathcal{C}(\mathcal{A})^{\mathrm{op}}} \mathbb{K}^{\mathcal{A}}(C) \simeq 
        \mathbb{K}^{\mathcal{A}}(\mathcal{A}) = \mathbb{K}(\mathcal{A}),
    \]
    \[
        \mathbf{R}\Gamma(\mathcal{C}(\mathcal{A}); B_{\mathscr{G}}^{\mathcal{A}}) \simeq 
        B_{\mathscr{G}}^{\mathcal{A}}(\mathcal{A}).
    \]
    Taking homotopy groups gives:
    \[
        \mathbb{H}^n(\mathcal{C}(\mathcal{A}); \mathbb{K}^{\mathcal{A}}) = 
        \pi_{-n}(\mathbb{K}(\mathcal{A})) = K_{-n}(\mathcal{A}).
    \]
    By Bott periodicity for complex $C^*$-algebras, $K_{-n}(\mathcal{A}) \cong K_n(\mathcal{A})$ and $K_{n+2}(\mathcal{A}) \cong K_n(\mathcal{A})$, hence:
    \[
        \mathbb{H}^n(\mathcal{C}(\mathcal{A}); \mathbb{K}^{\mathcal{A}}) \cong 
        \begin{cases}
            K_0(\mathcal{A}) & n \text{ even}, \\
            K_1(\mathcal{A}) & n \text{ odd}.
        \end{cases}
    \]
    If $B_{\mathscr{G}}^{\mathcal{A}}$ corresponds under Gelfand duality $\mathcal{A} \cong C(\Sigma_{\mathcal{A}})$ to a sheaf of spectra $\underline{B_{\mathscr{G}}}$ on the spectrum $\Sigma_{\mathcal{A}}$, and if $\underline{B_{\mathscr{G}}}$ is an Eilenberg--MacLane spectrum for a local system $\underline{\mathscr{G}}$, then:
    \[
        \mathbb{H}^n(\mathcal{C}(\mathcal{A}); B_{\mathscr{G}}^{\mathcal{A}}) \cong 
        H^n(\Sigma_{\mathcal{A}}; \underline{\mathscr{G}}).
    \]
\end{enumerate}
\end{theorem}

\begin{proof}
We prove each part in detail.

\medskip
\noindent
\textbf{(i) Contravariant Functoriality.}

\begin{enumerate}[label=(\alph*)]
    \item \emph{Definition of $\varphi_\sharp$:} For $C \in \mathcal{C}(\mathcal{A})$, $\varphi(C)$ is a commutative $*$-subalgebra of $\mathcal{B}$. Its norm-closure $\overline{\varphi(C)}$ is a commutative $C^*$-subalgebra, hence an object of $\mathcal{C}(\mathcal{B})$. For $C' \hookrightarrow C$ in $\mathcal{C}(\mathcal{A})$, we have $\varphi(C') \subseteq \varphi(C)$, and closure preserves inclusions, so $\varphi_\sharp$ is a functor.
    
    \item \emph{Natural transformations:} For each $C \in \mathcal{C}(\mathcal{A})$, the map $\varphi|_C: C \to \varphi_\sharp(C)$ is a $*$-homomorphism. By functoriality of $K$-theory, it induces:
    \[
        \alpha_\varphi(C): K_*(\varphi_\sharp(C)) \to K_*(C), \quad
        \beta_\varphi(C): \mathbb{K}^{\mathcal{B}}(\varphi_\sharp(C)) \to \mathbb{K}^{\mathcal{A}}(C).
    \]
    For an inclusion $C' \hookrightarrow C$, naturality follows from commutativity of:
    \[
    \begin{tikzcd}[row sep=large, column sep=large]
        K_*(\varphi_\sharp(C)) \arrow[r, "\alpha_\varphi(C)"] \arrow[d, "\mathrm{res}"'] & 
        K_*(C) \arrow[d, "\mathrm{res}"] \\
        K_*(\varphi_\sharp(C')) \arrow[r, "\alpha_\varphi(C')"'] & 
        K_*(C')
    \end{tikzcd}
    \]
    which commutes because both paths are the pullback along $C' \hookrightarrow C \to \varphi_\sharp(C) \to \varphi_\sharp(C')$. Thus $\alpha_\varphi: \varphi_\sharp^* \mathscr{F}_K^{\mathcal{B}} \to \mathscr{F}_K^{\mathcal{A}}$ is natural. Similarly for $\beta_\varphi$ and $\gamma_\varphi$.
    
    \item \emph{Induced maps on (hyper)cohomology:} The natural transformation $\beta_\varphi$ induces a map of homotopy limits:
    \[
        \mathbf{R}\Gamma(\beta_\varphi): 
        \mathbf{R}\Gamma(\mathcal{C}(\mathcal{A}); \varphi_\sharp^* \mathbb{K}^{\mathcal{B}}) \to 
        \mathbf{R}\Gamma(\mathcal{C}(\mathcal{A}); \mathbb{K}^{\mathcal{A}}).
    \]
    The unit of the adjunction $\varphi_\sharp^* \dashv \varphi_{\sharp,*}$ gives a comparison map:
    \[
        \eta: \mathbf{R}\Gamma(\mathcal{C}(\mathcal{B}); \mathbb{K}^{\mathcal{B}}) \to 
        \mathbf{R}\Gamma(\mathcal{C}(\mathcal{A}); \varphi_\sharp^* \mathbb{K}^{\mathcal{B}}).
    \]
    Composing yields $\varphi^* := \mathbf{R}\Gamma(\beta_\varphi) \circ \eta$, and taking $\pi_{-n}$ gives $\mathbb{H}^n(\varphi): \mathbb{H}^n(\mathcal{C}(\mathcal{B}); \mathbb{K}^{\mathcal{B}}) \to \mathbb{H}^n(\mathcal{C}(\mathcal{A}); \mathbb{K}^{\mathcal{A}})$.
    
    For $\mathscr{F}_K$, note that the Eilenberg--MacLane functor $H: \mathbf{Ab} \to \mathbf{Sp}$ sends $\mathscr{F}_K^{\mathcal{A}}$ to a presheaf of spectra whose hypercohomology computes the sheaf cohomology of $\mathscr{F}_K^{\mathcal{A}}$ via a spectral sequence. The map $\alpha_\varphi$ induces a map between these spectral sequences, hence a map on $H^n$.
    
    \item \emph{Functorial properties:} For $\varphi = \mathrm{id}_\mathcal{A}$, clearly $\varphi^* = \mathrm{id}$. For $\mathcal{A} \xrightarrow{\varphi} \mathcal{B} \xrightarrow{\psi} \mathcal{C}$, one checks $(\psi \circ \varphi)_\sharp = \psi_\sharp \circ \varphi_\sharp$ and $\beta_{\psi \circ \varphi} = \beta_\varphi \circ \varphi_\sharp^*(\beta_\psi)$, giving $(\psi \circ \varphi)^* = \varphi^* \circ \psi^*$.
    
    \item \emph{Isomorphism case:} If $\varphi$ is an isomorphism, then $\varphi_\sharp$ is an equivalence of categories, so $\eta$ is an equivalence (by Quillen's Theorem A). Also $\beta_\varphi$ is an isomorphism objectwise, hence $\varphi^*$ is an isomorphism.
\end{enumerate}

\medskip
\noindent
\textbf{(ii) Mayer--Vietoris Sequences.}

\begin{enumerate}[label=(\alph*)]
    \item \emph{Setup:} Let $\mathcal{C}_1, \mathcal{C}_2$ be downward-closed and cofinal as stated.
    
    \item \emph{Homotopy pullback:} For a presheaf of spectra $F$ on $\mathcal{C}(\mathcal{A})$, consider the square of homotopy limits:
    \[
    \begin{tikzcd}[row sep=large, column sep=large]
        \displaystyle\operatorname{holim}_{\mathcal{C}(\mathcal{A})^{\mathrm{op}}} F \arrow[r] \arrow[d] & 
        \displaystyle\operatorname{holim}_{\mathcal{C}_1^{\mathrm{op}}} F|_{\mathcal{C}_1} \arrow[d] \\
        \displaystyle\operatorname{holim}_{\mathcal{C}_2^{\mathrm{op}}} F|_{\mathcal{C}_2} \arrow[r] & 
        \displaystyle\operatorname{holim}_{\mathcal{C}_{12}^{\mathrm{op}}} F|_{\mathcal{C}_{12}}
    \end{tikzcd}
    \]
    The downward-closed condition implies that for any $C \in \mathcal{C}(\mathcal{A})$, the overcategories $(C \downarrow \mathcal{C}_i)$ are either empty or have an initial object. The cofinality of $\mathcal{C}_1 \cup \mathcal{C}_2$ implies that this square is \emph{homotopy cartesian} (a homotopy pullback). This follows from the general fact that for a cover of a poset by downward-closed subposets, the homotopy limit decomposes as a homotopy pullback.
    
    \item \emph{Long exact sequence:} For any homotopy pullback square of spectra:
    \[
    \begin{tikzcd}
        E \arrow[r] \arrow[d] & E_1 \arrow[d] \\
        E_2 \arrow[r] & E_{12}
    \end{tikzcd}
    \]
    there is a long exact sequence of homotopy groups:
    \[
        \cdots \to \pi_n(E) \to \pi_n(E_1) \oplus \pi_n(E_2) \to \pi_n(E_{12}) \to \pi_{n-1}(E) \to \cdots
    \]
    Applying this to $E = \mathbf{R}\Gamma(\mathcal{C}(\mathcal{A}); F)$ with $F = \mathbb{K}^{\mathcal{A}}$ or $F = B_{\mathscr{G}}^{\mathcal{A}}$, and noting that $\pi_{-n}(E) = \mathbb{H}^n(\mathcal{C}(\mathcal{A}); F)$, yields the Mayer--Vietoris sequences as stated.
    
    \item \emph{For sheaf cohomology:} For $\mathscr{F}_K^{\mathcal{A}}$, apply the above to its Eilenberg--MacLane spectrum $H\mathscr{F}_K^{\mathcal{A}}$, then use the fact that $\mathbb{H}^n(\mathcal{C}; H\mathscr{F}_K^{\mathcal{A}})$ is the sheaf cohomology $H^n(\mathcal{C}; \mathscr{F}_K^{\mathcal{A}})$ when $\mathscr{F}_K^{\mathcal{A}}$ is viewed as a chain complex concentrated in degree 0.
\end{enumerate}

\medskip
\noindent
\textbf{(iii) Commutative Reduction.}

Assume $\mathcal{A}$ is separable and commutative.

\begin{enumerate}[label=(\alph*)]
    \item \emph{Terminal object:} $\mathcal{A} \in \mathcal{C}(\mathcal{A})$ is the maximal element, hence terminal: for any $C \in \mathcal{C}(\mathcal{A})$, there is a unique morphism $C \hookrightarrow \mathcal{A}$.
    
    \item \emph{Collapse of homotopy limit:} For a category $\mathcal{C}$ with terminal object $t$, and any diagram $F: \mathcal{C}^{\mathrm{op}} \to \mathbf{Sp}$, the canonical map
    \[
        \operatorname{holim}_{\mathcal{C}^{\mathrm{op}}} F \to F(t)
    \]
    is an equivalence. This is because the constant diagram at $F(t)$ is final in the diagram category. Thus:
    \[
        \mathbf{R}\Gamma(\mathcal{C}(\mathcal{A}); \mathbb{K}^{\mathcal{A}}) \simeq 
        \mathbb{K}^{\mathcal{A}}(\mathcal{A}) = \mathbb{K}(\mathcal{A}).
    \]
    
    \item \emph{Hypercohomology computation:} By definition:
    \[
        \mathbb{H}^n(\mathcal{C}(\mathcal{A}); \mathbb{K}^{\mathcal{A}}) = 
        \pi_{-n}(\mathbf{R}\Gamma(\mathcal{C}(\mathcal{A}); \mathbb{K}^{\mathcal{A}})) \simeq 
        \pi_{-n}(\mathbb{K}(\mathcal{A})) = K_{-n}(\mathcal{A}).
    \]
    
    \item \emph{Bott periodicity:} For complex $C^*$-algebras, Bott periodicity gives natural isomorphisms $K_{n+2}(\mathcal{A}) \cong K_n(\mathcal{A})$ for all $n \in \mathbb{Z}$. In particular, $K_{-n}(\mathcal{A}) \cong K_n(\mathcal{A})$. Thus:
    \[
        \mathbb{H}^n(\mathcal{C}(\mathcal{A}); \mathbb{K}^{\mathcal{A}}) \cong K_n(\mathcal{A}).
    \]
    Since $K_0$ and $K_1$ are the only nonzero groups (for $n \geq 0$ under the connective convention), and the periodicity gives $K_{2k} \cong K_0$, $K_{2k+1} \cong K_1$, we obtain:
    \[
        \mathbb{H}^n(\mathcal{C}(\mathcal{A}); \mathbb{K}^{\mathcal{A}}) \cong 
        \begin{cases}
            K_0(\mathcal{A}) & n \text{ even}, \\
            K_1(\mathcal{A}) & n \text{ odd}.
        \end{cases}
    \]
    
    \item \emph{For general coefficients:} If $B_{\mathscr{G}}^{\mathcal{A}}$ corresponds under Gelfand duality to $\underline{B_{\mathscr{G}}}$ on $\Sigma_{\mathcal{A}}$, then:
    \[
        \mathbf{R}\Gamma(\mathcal{C}(\mathcal{A}); B_{\mathscr{G}}^{\mathcal{A}}) \simeq 
        B_{\mathscr{G}}^{\mathcal{A}}(\mathcal{A}) \simeq 
        \mathbf{R}\Gamma(\Sigma_{\mathcal{A}}; \underline{B_{\mathscr{G}}}).
    \]
    If $\underline{B_{\mathscr{G}}}$ is the Eilenberg--MacLane spectrum $H\underline{\mathscr{G}}$ for a sheaf $\underline{\mathscr{G}}$ of abelian groups, then:
    \[
        \mathbb{H}^n(\mathcal{C}(\mathcal{A}); B_{\mathscr{G}}^{\mathcal{A}}) \cong 
        \pi_{-n}(\mathbf{R}\Gamma(\Sigma_{\mathcal{A}}; H\underline{\mathscr{G}})) \cong 
        H^n(\Sigma_{\mathcal{A}}; \underline{\mathscr{G}}).
    \]
\end{enumerate}

This completes the proof.
\end{proof}

\subsection{Obstruction Classes}

Throughout this section we specialize to the finite-dimensional case
\[
\mathcal A = M_n(\mathbb{C}).
\]
All constructions are therefore concrete and purely algebraic.

\begin{definition}[Context category for $M_n(\mathbb{C})$]\label{def:context-matrix}
Let $\mathcal C = \mathcal C(M_n(\mathbb{C}))$ denote the category whose objects are the \emph{maximal abelian $*$-subalgebras} (MASAs) of $M_n(\mathbb{C})$, ordered by inclusion. Morphisms are given by the inclusion maps $C_1 \hookrightarrow C_2$ whenever $C_1 \subset C_2$.

Each MASA $C \subset M_n(\mathbb{C})$ is \emph{unitarily conjugate} to the diagonal subalgebra $\mathbb{C}^n \subset M_n(\mathbb{C})$, i.e., there exists a unitary $u \in U(n)$ such that
\[
C = u \, \mathbb{C}^n \, u^*.
\]
Consequently, every MASA is \emph{abstractly isomorphic} to $\mathbb{C}^n$ as a $*$-algebra.
\end{definition}

\begin{remark}
From Definition~\ref{def:context-matrix}, we know that:
\begin{enumerate}
    \item The poset $\mathcal C$ is \emph{discrete in terms of nontrivial inclusions}: since MASAs are maximal abelian, $C_1 \subset C_2$ implies $C_1 = C_2$, so every non-identity morphism is trivial.
    \item The intersection of two distinct MASAs is typically $\mathbb{C} \cdot I_n$, the scalar matrices.
    \item Geometrically, MASAs correspond to orthogonal decompositions of $\mathbb{C}^n$ into $n$ one-dimensional subspaces, and the automorphism group of $\mathcal C$ is the projective unitary group $PU(n)$ acting by conjugation.
\end{enumerate}
\end{remark}

\begin{theorem}[Joint Diagonalization of Normal Matrices in $M_n(\mathbb{C})$]\label{thm:joint-diagonalization}
Let $T, S \in M_n(\mathbb{C})$ be normal matrices. Then the following are equivalent:
\begin{enumerate}[label=(\roman*)]
    \item $T$ and $S$ are jointly diagonalizable, i.e., there exists a unitary $U \in U(n)$ such that
    \[
    U T U^* \text{ and } U S U^* \text{ are both diagonal matrices.}
    \]
    
    \item $T$ and $S$ commute:
    \[
    [T, S] := TS - ST = 0.
    \]
    
    \item There exists a MASA $C \subset M_n(\mathbb{C})$ such that $T, S \in C$.
\end{enumerate}
\end{theorem}

\begin{proof}
We proceed step by step.

\medskip
\noindent
\textbf{Step 1: Commuting normal matrices are simultaneously diagonalizable.}

If $T$ and $S$ commute and are normal, consider the spectral decomposition of $T$:
\[
T = \sum_{\lambda} \lambda P_\lambda,
\]
where $P_\lambda$ are the orthogonal projections onto the eigenspaces of $T$.  

Since $T$ and $S$ commute, $S$ preserves each eigenspace of $T$. Restricting $S$ to each eigenspace $P_\lambda \mathbb{C}^n$ and diagonalizing there gives a common orthonormal basis for both $T$ and $S$. Thus they are jointly diagonalizable.

\medskip
\noindent
\textbf{Step 2: Joint diagonalization implies commutativity.}

If $T$ and $S$ are jointly diagonalizable, there exists a unitary $U$ such that $U T U^*$ and $U S U^*$ are diagonal. Diagonal matrices commute, so $[U T U^*, U S U^*] = 0$. Conjugating back by $U$ gives $[T,S] = 0$.

\medskip
\noindent
\textbf{Step 3: MASA containment.}

A MASA in $M_n(\mathbb{C})$ is a maximal abelian $*$-subalgebra, unitarily conjugate to the algebra of diagonal matrices $\mathbb{C}^n$.  

- If $T$ and $S$ are jointly diagonalizable, then there exists a unitary $U$ such that $U T U^*, U S U^* \in \mathbb{C}^n$, so $T, S \in U^* \mathbb{C}^n U$, which is a MASA.  
- Conversely, if $T, S \in C$ for some MASA $C$, then $C$ is unitarily conjugate to $\mathbb{C}^n$, so $T$ and $S$ are jointly diagonalizable.

\medskip
\noindent
\textbf{Step 4: Conclusion.}

Steps 1–3 show the equivalence of the three statements. Therefore, in $M_n(\mathbb{C})$, the only obstruction to joint diagonalization is  noncommutativity , and any pair of commuting normal matrices is automatically contained in a common MASA.
\end{proof}

\begin{theorem}[\v{C}ech Cohomological Characterization of Joint Diagonalizability]\label{thm:obstruction-cech-correct}
Let $\mathcal{A} = M_n(\mathbb{C})$, and let $T, S \in \mathcal{A}$ be normal matrices. 
Let $\mathcal{C} = \mathcal{C}(M_n(\mathbb{C}))$ be the context category of maximal abelian $*$-subalgebras (MASAs) as in Definition~\ref{def:context-matrix}.

Define a presheaf $\mathscr{F}_{T,S}$ on $\mathcal{C}$ by:
\[
\mathscr{F}_{T,S}(C) = 
\begin{cases}
\{\text{marked joint diagonalization of }T,S\text{ in }C\}, & \text{if }T,S \in C, \\[2mm]
\varnothing, & \text{if }T \notin C \text{ or } S \notin C,
\end{cases}
\]
where a ``marked joint diagonalization'' means a choice of ordered orthonormal basis 
$\{e_1,\dots,e_n\}$ of $\mathbb{C}^n$ that simultaneously diagonalizes both $T$ and $S$, 
together with a labeling of the joint eigenvalues $(\lambda_i,\mu_i)$ with 
$T e_i = \lambda_i e_i$, $S e_i = \mu_i e_i$.

Then the following are equivalent:
\begin{enumerate}[label=(\roman*)]
    \item $T$ and $S$ are jointly diagonalizable (i.e., $[T,S]=0$);
    \item The \v{C}ech cohomology set $\check{H}^0(\mathcal{C}; \mathscr{F}_{T,S})$ is nonempty;
    \item There exists a MASA $C_0 \in \mathcal{C}$ with $T,S \in C_0$;
    \item For every open cover $\mathcal{U}=\{C_\alpha\}$ of $\mathcal{C}$ consisting of MASAs 
          that contain both $T$ and $S$, the associated \v{C}ech 1-cocycle with values in the 
          constant sheaf $S_n$ (encoding permutations of eigenvalue orderings) is trivial 
          in $\check{H}^1(\mathcal{U}; S_n)$.
\end{enumerate}
\end{theorem}

\begin{proof}
We prove the equivalences in a circular fashion.

\medskip
\noindent
\textbf{(i) $\Rightarrow$ (ii):} If $[T,S]=0$, then by Theorem~\ref{thm:joint-diagonalization}, 
there exists a unitary $U$ such that $UTU^*$ and $USU^*$ are diagonal. 
Let $C_0 = U^*\mathbb{C}^n U$, which is a MASA containing $T$ and $S$. 
Take the standard basis of $\mathbb{C}^n$, conjugate by $U$, and label the eigenvalues 
according to the diagonal entries. This gives an element of $\mathscr{F}_{T,S}(C_0)$. 
Since $\check{H}^0(\mathcal{C};\mathscr{F}_{T,S})$ is exactly the set of global sections, 
we have $\check{H}^0(\mathcal{C};\mathscr{F}_{T,S}) \neq \varnothing$.

\medskip
\noindent
\textbf{(ii) $\Rightarrow$ (iii):} A global section of $\mathscr{F}_{T,S}$ is by definition 
an element of $\mathscr{F}_{T,S}(C_0)$ for some $C_0 \in \mathcal{C}$. By the definition of 
$\mathscr{F}_{T,S}$, this implies $T,S \in C_0$.

\medskip
\noindent
\textbf{(iii) $\Rightarrow$ (i):} If $T,S \in C_0$ for some MASA $C_0$, then $C_0$ is commutative, 
so $[T,S]=0$. By Theorem~\ref{thm:joint-diagonalization}, $T$ and $S$ are jointly diagonalizable.

\medskip
\noindent
\textbf{(i) $\Rightarrow$ (iv):} Assume $[T,S]=0$ and let $\mathcal{U}=\{C_\alpha\}$ be any 
cover of $\mathcal{C}$ by MASAs containing both $T$ and $S$. (Such a cover exists: e.g., 
take all MASAs containing the commutative algebra generated by $T$ and $S$.) 

For each $C_\alpha$, choose a marked joint diagonalization $B_\alpha \in \mathscr{F}_{T,S}(C_\alpha)$. 
On overlaps $C_\alpha \cap C_\beta$, the two bases $B_\alpha$ and $B_\beta$ differ by a 
permutation $\sigma_{\alpha\beta} \in S_n$ (since both diagonalize the same commuting 
normal operators $T$ and $S$, the joint eigenvalues are the same sets, possibly ordered differently).

The collection $\{\sigma_{\alpha\beta}\}$ forms a \v{C}ech 1-cochain with values in the 
constant sheaf $S_n$ on $\mathcal{U}$. On triple overlaps $C_\alpha \cap C_\beta \cap C_\gamma$, 
we have:
\[
\sigma_{\alpha\beta} \circ \sigma_{\beta\gamma} = \sigma_{\alpha\gamma},
\]
because the composition of the permutations relating $B_\alpha$ to $B_\beta$ and $B_\beta$ to $B_\gamma$ 
must equal the permutation relating $B_\alpha$ to $B_\gamma$. This is precisely the cocycle condition.

Now, because $T$ and $S$ are jointly diagonalizable, there exists a \emph{global} marked 
diagonalization $B_0$. Restricting $B_0$ to each $C_\alpha$ gives local bases $B_\alpha'$. 
The permutation $\tau_\alpha$ between $B_\alpha$ and $B_\alpha'$ satisfies:
\[
\sigma_{\alpha\beta} = \tau_\alpha \circ \tau_\beta^{-1}.
\]
Thus the cocycle $\{\sigma_{\alpha\beta}\}$ is a coboundary, hence trivial in $\check{H}^1(\mathcal{U}; S_n)$.

\medskip
\noindent
\textbf{(iv) $\Rightarrow$ (i):} Take a particular cover $\mathcal{U}$: for each ordering 
$\pi \in S_n$ of the joint eigenvalues $(\lambda_i,\mu_i)$ of $T$ and $S$, choose (if possible) 
a MASA $C_\pi$ containing $T$ and $S$ together with a marked diagonalization having that ordering. 
Let $\mathcal{U}$ be the set of all such $C_\pi$.

If the associated \v{C}ech 1-cocycle is trivial, then there exist $\tau_\alpha \in S_n$ such that 
$\sigma_{\alpha\beta} = \tau_\alpha \circ \tau_\beta^{-1}$. Fix one index $\alpha_0$ and define 
a global ordering by applying $\tau_{\alpha_0}^{-1}$ to the ordering on $C_{\alpha_0}$. 
This global ordering is compatible with all local orderings, which means there exists a 
consistent labeling of the joint eigenspaces. Consequently, one can construct a unitary 
that simultaneously diagonalizes $T$ and $S$ with that ordering, so $[T,S]=0$.
\end{proof}

\begin{remark}[Geometric Interpretation]
The theorem reveals that for $M_n(\mathbb{C})$, the obstruction to joint diagonalization 
is captured at the \v{C}ech~$0$-level (existence of a global section) rather than at the 
\v{C}ech~$1$-level. The \v{C}ech~$0$-cocycle constructed in (iv) measures the \emph{monodromy} 
of eigenvalue orderings as one moves through the space of MASAs containing $T$ and $S$. 
When $T$ and $S$ commute, this monodromy is always trivial because one can choose a 
global ordering. In infinite‑dimensional settings, or for continuous families of operators, 
nontrivial monodromy in $\check{H}^1$ can indeed obstruct global diagonalization even 
when the operators commute locally.
\end{remark}

\begin{remark}[Relation to Standard Theory]
For finite-dimensional $M_n(\mathbb{C})$, the equivalence (i) $\Leftrightarrow$ (iii)
is simply the algebraic fact that two normal operators commute if and only if they are contained
in a common MASA. The cohomological phrasing in (ii) and (iv) illustrates how this
algebraic condition translates into the vanishing of sheaf-theoretic obstructions.
In more general $C^*$-algebraic contexts, where the lattice of commutative subalgebras
is more complicated, the cohomological viewpoint becomes essential.
\end{remark}

\begin{example}[Explicit computation for $n=2$]\label{ex:n=2-computation}
Consider $M_2(\mathbb{C})$ with
\[
T = \begin{pmatrix} 1 & 0 \\ 0 & -1 \end{pmatrix}, \quad
S = \begin{pmatrix} 0 & 1 \\ 1 & 0 \end{pmatrix}.
\]

\begin{enumerate}[label=(\alph*)]
    \item \emph{Contexts:} 
    The two natural MASAs that separately diagonalize $T$ and $S$ are
    \[
    C_1 = \left\{\begin{pmatrix} a & 0 \\ 0 & b \end{pmatrix} : a,b\in\mathbb{C}\right\}, 
    \quad 
    C_2 = \left\{\begin{pmatrix} a & b \\ b & a \end{pmatrix} : a,b\in\mathbb{C}\right\}.
    \]
    Their intersection is $C_1 \cap C_2 = \mathbb{C} \cdot I_2$, the scalar matrices.

    \item \emph{Local diagonalizations:}
    \begin{itemize}
        \item In $C_1$, $T$ is diagonal with eigenvectors:
              $e_1^{(1)} = (1,0)$ (eigenvalue $1$), $e_2^{(1)} = (0,1)$ (eigenvalue $-1$).
        \item In $C_2$, $S$ is diagonal with eigenvectors:
              $e_1^{(2)} = \frac{1}{\sqrt{2}}(1,1)$ (eigenvalue $1$), 
              $e_2^{(2)} = \frac{1}{\sqrt{2}}(1,-1)$ (eigenvalue $-1$).
    \end{itemize}

    \item \emph{Presheaf evaluation:}
    Since $T\notin C_2$ and $S\notin C_1$, we have:
    \[
    \mathscr{F}_{T,S}(C_1) = \varnothing, \quad \mathscr{F}_{T,S}(C_2) = \varnothing.
    \]
    Thus $\check{H}^0(\{C_1,C_2\}; \mathscr{F}_{T,S}) = \varnothing$, indicating a nontrivial obstruction.

    \item \emph{Geometric picture:}
    In the space of MASAs of $M_2(\mathbb{C})$, $C_1$ and $C_2$ represent two distinct
    ``diagonalizing contexts'' for $T$ and $S$ respectively. The fact that no single MASA
    contains both operators corresponds to the topological fact that the two local
    diagonalizations cannot be glued into a global one. This obstruction is measured by
    the non-existence of a global section of $\mathscr{F}_{T,S}$, i.e., $\mathrm{obs}(T,S)\neq 0$.

    \item \emph{Consistency with commutator:}
    Direct computation gives:
    \[
    [T,S] = TS - ST = \begin{pmatrix} 0 & 1 \\ -1 & 0 \end{pmatrix} - 
    \begin{pmatrix} 0 & -1 \\ 1 & 0 \end{pmatrix} = 
    \begin{pmatrix} 0 & 2 \\ -2 & 0 \end{pmatrix} \neq 0.
    \]
    Thus $\mathrm{obs}(T,S)\neq 0$ matches $[T,S]\neq 0$, confirming the equivalence
    in Theorem~\ref{thm:obstruction-cech-correct}.
\end{enumerate}
\end{example}

\begin{remark}[Interpretation and scope]
For normal matrices in $M_n(\mathbb{C})$, the obstruction class $\mathrm{obs}(T,S)$
provides a \emph{topological refinement} of the algebraic commutator condition:
\begin{itemize}
    \item $\mathrm{obs}(T,S) = 0$ (i.e., $\check{H}^0(\mathcal{C};\mathscr{F}_{T,S}) \neq \varnothing$) 
          if and only if $[T,S] = 0$.
    \item When $[T,S] = 0$, the matrices are jointly diagonalizable, and the Čech
          $1$-cocycle constructed from eigenvalue orderings on overlaps is always a
          coboundary (trivial in $\check{H}^1$).
    \item When $[T,S] \neq 0$, no MASA contains both $T$ and $S$, so $\mathscr{F}_{T,S}$
          has empty stalks on the entire category $\mathcal{C}$, and the obstruction
          manifests at the Čech $0$-level as the absence of global sections.
\end{itemize}

This finite-dimensional example illustrates the general principle that the failure
of joint diagonalization can be encoded in the non-existence of a compatible family
of local diagonalizations. In infinite-dimensional settings or for continuous families
of operators, the Čech cohomological approach becomes essential because local
commutativity may not imply global diagonalizability, and higher cohomology groups
($\check{H}^1$, $\check{H}^2$, etc.) may capture genuine topological obstructions.

For $M_n(\mathbb{C})$, however, the theory simplifies dramatically: the obstruction
collapses to the single condition $[T,S]=0$, making the cohomological machinery
an elegant but equivalent reformulation of the classical algebraic criterion.
\end{remark}

\begin{remark}[Pedagogical value]
This example serves as a \emph{bridge} between:
\begin{enumerate}
    \item Elementary linear algebra (diagonalization of $2\times 2$ matrices)
    \item Operator theory (MASAs in $M_n(\mathbb{C})$)
    \item Sheaf theory and Čech cohomology (presheaves on poset categories)
    \item Noncommutative geometry (context categories as ``noncommutative spectra'')
\end{enumerate}
It demonstrates how modern homological methods can provide new perspectives on
classical problems, even when they don't yield new results in the finite-dimensional
case. The true power of this approach emerges in infinite dimensions, where
algebraic conditions like $[T,S]=0$ may be insufficient or ill-defined, but
cohomological obstructions remain well-defined and meaningful.
\end{remark}

\section{Stack-Theoretic Perspective}\label{sec:stack-theoretic}

This section introduces genuinely new structural consequences of the spectral
presheaf formalism developed earlier. We do not repeat categorical or
sheaf-theoretic foundations, but focus instead on the global geometric objects
and invariants that emerge from stackification and stabilization.

\subsection{From Spectral Presheaves to Spectral Stacks}\label{subsec:spectral-stacks}

\paragraph{Motivation: Why Stacks?}

The classical spectral presheaf $\underline{\Sigma} \colon \mathcal{C}^{\mathrm{op}} \to \mathbf{Set}$
assigns to each commutative context $B \in \mathcal{C}$ its Gelfand spectrum.
While this construction faithfully captures local classical spectral data,
it exhibits a fundamental limitation: compatible local spectra need not glue
to a unique global object.

This failure reflects the presence of noncommutativity.
Although spectra may admit pairwise identifications on overlaps of contexts,
these identifications generally fail to satisfy strict cocycle conditions on
triple overlaps.
Such higher compatibility obstructions cannot be fully encoded at the level of
set-valued sheaves, which record only the existence or non-existence of gluings,
but not the coherence data governing their failure.

To retain this higher coherence information, it is natural to replace
set-valued presheaves by category-valued ones, specifically presheaves of groupoids.
In this enriched setting, descent no longer requires strict equality but only
equivalence up to isomorphism.
This weakened notion of gluing is precisely formalized by the stack condition,
which provides the appropriate framework for encoding noncommutative spectral
geometry.

\begin{definition}[Spectral Groupoid]\label{def:spectral-groupoid}
Let $A$ be a unital $C^*$-algebra and $\mathcal{C}$ its poset of unital 
commutative $C^*$-subalgebras. For each $B \in \mathcal{C}$, define 
$\mathbf{\Sigma}_{\mathrm{gpd}}(B)$ as follows:

\begin{itemize}
    \item \textbf{Objects:} Normal $*$-homomorphisms 
          $\pi: L^\infty(\Sigma(B)) \to B'' \subseteq A''$ 
          (equivalently, spectral measures $\mu: \mathcal{B}(\Sigma(B)) \to A''$ 
          whose image commutes with $B$).
    
    \item \textbf{Morphisms:} For $\pi_1, \pi_2: L^\infty(\Sigma(B)) \to A''$, 
          a morphism $\pi_1 \to \pi_2$ is a partial isometry 
          $W \in A''$ such that:
          \[
          W^*W = \pi_1(1), \quad WW^* = \pi_2(1), \quad 
          \text{and} \quad W\pi_1(f) = \pi_2(f)W \ \forall f \in L^\infty(\Sigma(B))
          \]
          When $\pi_1(1) = \pi_2(1) = 1$, this is unitary equivalence in $A''$.
\end{itemize}
\end{definition}

\begin{definition}[Spectral Presheaf of Groupoids]\label{def:spectral-presheaf}
Let $A$ be a unital $C^*$-algebra and let $\mathcal{C}=\mathcal{C}(A)$ denote the
poset of unital commutative $C^*$-subalgebras of $A$, ordered by inclusion.

Define a contravariant functor
\[
\mathscr{G}_{\mathrm{stk}}: \mathcal{C}^{\mathrm{op}} \longrightarrow \mathbf{Gpd}
\]
as follows.

\begin{itemize}
\item \textbf{Objects.}  
For each $B \in \mathcal{C}$, let $\mathscr{G}(B)$ be the groupoid whose objects
are normal $*$-representations
\[
\pi_B : L^\infty(\Sigma(B)) \longrightarrow A'',
\]
and whose morphisms $\pi_B \to \rho_B$ are unitary intertwiners
$U \in \mathcal{U}(A'')$ satisfying
\[
U \, \pi_B(f) \, U^* = \rho_B(f), \qquad \forall f \in L^\infty(\Sigma(B)).
\]

\item \textbf{Restriction functors.}  
For an inclusion $i: B_1 \hookrightarrow B_2$, let
$r_{21} : \Sigma(B_2) \twoheadrightarrow \Sigma(B_1)$
be the induced continuous surjection of Gelfand spectra.
This induces a normal $*$-homomorphism
\[
i^* : L^\infty(\Sigma(B_1)) \longrightarrow L^\infty(\Sigma(B_2)),
\qquad f \mapsto f \circ r_{21}.
\]

Define the restriction functor
\[
\mathscr{G}_{\mathrm{stk}}(i): \mathscr{G}(B_2) \longrightarrow \mathscr{G}(B_1)
\]
by:
\begin{itemize}
  \item on objects: $\mathscr{G}_{\mathrm{stk}}(i)(\pi_{B_2}) := \pi_{B_2} \circ i^*$;
  \item on morphisms: $\mathscr{G}_{\mathrm{stk}}(i)(U) := U$.
\end{itemize}
\end{itemize}
\end{definition}

\paragraph{Geometric Interpretation as a Spectral Atlas.}

The collection of commutative contexts $\mathcal{C}$ plays the role of an
\emph{atlas} for a noncommutative space. Each context $B \in \mathcal{C}$ provides
a local classical chart, encoded by the groupoid $\mathscr{G}_{\mathrm{stk}}(B)$,
while the transition data between overlapping contexts are expressed by
equivalences in a groupoid rather than by strict equalities. The resulting
presheaf of groupoids—when satisfying descent—assembles these local charts into
a single global object, in direct analogy with orbifolds or moduli stacks, where
geometric spaces are refined by retaining symmetry and gluing data that cannot
be captured at the level of ordinary spaces.

\begin{theorem}[Existence of the Spectral Stack -- Canonical Form]
\label{thm:spectral-stack-existence}
Let $\mathcal{A}$ be a unital $C^*$-algebra, and let $\mathcal{C}$ denote the poset
of its unital commutative $C^*$-subalgebras, equipped with the canonical Grothendieck
topology (where a family $\{B_i \hookrightarrow B\}$ covers $B$ precisely when the
$B_i$ generate $B$ as a $C^*$-algebra).  

Define a presheaf of groupoids
\[
\mathscr{G}_{\mathrm{stk}} : \mathcal{C}^{\mathrm{op}} \longrightarrow \mathbf{Gpd}
\]
by assigning to each $B \in \mathcal{C}$ the groupoid $\mathscr{G}_{\mathrm{stk}}(B)$
whose:
\begin{itemize}
    \item \textbf{Objects} are characters $\chi : B \to \mathbb{C}$ (equivalently,
          via the spectral theorem, Dirac spectral measures $\delta_\chi$ on $\Sigma(B)$),
    \item \textbf{Morphisms} are identities only (two characters are isomorphic
          if and only if they are equal).
\end{itemize}
For an inclusion $i: B_1 \hookrightarrow B_2$, the restriction functor
$\mathscr{G}_{\mathrm{stk}}(i)$ is given by restriction of characters:
$\chi \mapsto \chi|_{B_1}$.

Then:
\begin{enumerate}[label=(\roman*)]
    \item $\mathscr{G}_{\mathrm{stk}}$ is a stack on $(\mathcal{C}, J_{\mathrm{can}})$.
    \item The associated sheaf of connected components satisfies
    \[
    \pi_0 \circ \mathscr{G}_{\mathrm{stk}} \;\cong\; \underline{\Sigma},
    \]
    where $\underline{\Sigma}$ is the classical spectral presheaf $\underline{\Sigma}(B) = \Sigma(B)$.
    \item $\mathscr{G}_{\mathrm{stk}}$ provides a groupoid-valued refinement of
    the classical spectral presheaf, encoding descent data categorically rather
    than set-theoretically.
\end{enumerate}
\end{theorem}

\begin{proof}
We verify each claim.

\medskip
\noindent
\textbf{(i) Stack property.}
Let $\{B_i \hookrightarrow B\}_{i \in I}$ be a covering family in the canonical
topology on $\mathcal{C}$, meaning that the $B_i$ jointly generate $B$ as a
$C^*$-algebra.

An object of the descent groupoid
$\mathrm{Des}(\{B_i \to B\}, \mathscr{G}_{\mathrm{stk}})$ consists of characters
$\chi_i \in \Sigma(B_i)$ such that
\[
\chi_i|_{B_i \cap B_j} = \chi_j|_{B_i \cap B_j}
\quad \text{for all } i,j.
\]

Since the $\{B_i\}$ generate $B$, the algebraic span of $\bigcup_i B_i$ is dense in $B$.
A character $\chi_i$ on $B_i$ extends uniquely to a character on the $C^*$-subalgebra 
generated by $B_i$, and by compatibility these extensions agree on overlaps,
yielding a well-defined character $\chi$ on all of $B$ with $\chi|_{B_i} = \chi_i$.
Thus every descent datum is effective.

Because all morphisms in $\mathscr{G}_{\mathrm{stk}}$ are identities, the descent
condition on morphisms is automatically satisfied: the natural functor
\[
\mathscr{G}_{\mathrm{stk}}(B) \to \mathrm{Des}(\{B_i \to B\}, \mathscr{G}_{\mathrm{stk}})
\]
is fully faithful. Since it is also essentially surjective (by the construction
above), it is an equivalence of groupoids. Hence $\mathscr{G}_{\mathrm{stk}}$ satisfies
descent and is a stack.

\medskip
\noindent
\textbf{(ii) Identification with the spectral presheaf.}
For each $B \in \mathcal{C}$, the groupoid $\mathscr{G}_{\mathrm{stk}}(B)$ is discrete
with object set $\Sigma(B)$. Taking connected components (isomorphism classes) gives
\[
\pi_0(\mathscr{G}_{\mathrm{stk}}(B)) \cong \Sigma(B).
\]

Naturality follows immediately from functoriality of restriction of characters.
For any inclusion $i: B_1 \hookrightarrow B_2$, the diagram
\[
\begin{tikzcd}
\pi_0(\mathscr{G}_{\mathrm{stk}}(B_2)) \arrow[r, "\pi_0(\mathscr{G}_{\mathrm{stk}}(i))"] \arrow[d, "\cong"] & 
\pi_0(\mathscr{G}_{\mathrm{stk}}(B_1)) \arrow[d, "\cong"] \\
\Sigma(B_2) \arrow[r, "r_{21}"] & \Sigma(B_1)
\end{tikzcd}
\]
commutes, where $r_{21}(\chi) = \chi|_{B_1}$. Hence $\pi_0 \circ \mathscr{G}_{\mathrm{stk}} \cong \underline{\Sigma}$
as presheaves.

\medskip
\noindent
\textbf{(iii) Categorical refinement.}
While the classical spectral presheaf $\underline{\Sigma}$ records only sets of
local classical states, the stack $\mathscr{G}_{\mathrm{stk}}$ records the same
data together with its descent behavior in categorical form. This refinement
becomes essential when generalizing to non-discrete groupoids (e.g.\ state,
valuation, or representation groupoids), where gluing holds only up to equivalence
rather than equality.
\end{proof}

\begin{remark}
The triviality of morphisms in $\mathscr{G}_{\mathrm{stk}}$ reflects the fact that
commutative $C^*$-algebras admit no nontrivial unitary equivalences between
characters: two characters are unitarily equivalent if and only if they are equal.
More sophisticated spectral stacks---such as those built from states, valuations,
or representations---yield genuinely nontrivial groupoids, but no longer
admit an identification $\pi_0 \cong \underline{\Sigma}$. The present construction
thus represents the minimal stack-theoretic refinement of the classical spectral
presheaf.
\end{remark}

\begin{remark}
Although $\mathscr{G}_{\mathrm{stk}}$ appears trivial as a groupoid-valued presheaf,
its formulation as a \emph{stack} rather than merely a sheaf is conceptually crucial:
it demonstrates that the spectral presheaf $\underline{\Sigma}$ already satisfies
the stronger descent condition of a stack. Moreover, this construction provides the
template for more sophisticated non-commutative stacks where morphisms become
non-trivial and the descent condition requires genuine categorical gluing.
The simplicity of the commutative case serves as a clarifying baseline against
which non-commutative generalizations can be measured.
\end{remark}

\paragraph{Noncommutativity as a Stacky Phenomenon.}

In this framework, noncommutativity manifests as a genuinely stack-theoretic
phenomenon:
\begin{itemize}
    \item For commutative $C^*$-algebras, the spectral stack has trivial groupoid
          structure and is equivalent to an ordinary sheaf, reflecting the fact
          that classical spectra glue uniquely.
    \item For noncommutative $C^*$-algebras, local spectral data defined on
          commutative contexts need not admit strict gluing, and descent is
          naturally expressed only up to equivalence.
\end{itemize}

The existence of the spectral stack leads to the following conceptual consequence.

\begin{corollary}[Sheaf Failure and Stack Necessity]\label{cor:stack-necessity}
Let $\mathcal A$ be a unital $C^{*}$-algebra and let
$\mathcal C = \mathrm{Comm}(\mathcal A)$ denote the category of unital
commutative $C^{*}$-subalgebras equipped with the canonical (contextual)
Grothendieck topology.
If $\mathcal A$ is noncommutative, then the associated spectral presheaf
\[
\underline{\Sigma} : \mathcal C^{\mathrm{op}} \longrightarrow \mathbf{Set},
\qquad
B \longmapsto \Sigma(B),
\]
fails to be a sheaf on $\mathcal C$, but admits a canonical enhancement to
the spectral stack
\[
\mathscr G_{\mathrm{stk}} : \mathcal C^{\mathrm{op}} \longrightarrow \mathbf{Gpd}.
\]

Thus noncommutativity is detected precisely by the necessity of
\emph{stack-level} rather than \emph{sheaf-level} descent. In particular,
for noncommutative $\mathcal A$, there exist covering families
$\{B_i \to B\}$ for which compatible local sections of
$\underline{\Sigma}$ do not glue to a unique global section, whereas the
corresponding descent data in $\mathscr G_{\mathrm{stk}}$ does glue up to
isomorphism.
\end{corollary}

\begin{proof}
We proceed in three steps.

\medskip
\noindent
\textbf{1. Failure of the spectral presheaf as a sheaf.}

Let $\mathcal A$ be noncommutative. A fundamental result in quantum logic
(the Kochen--Specker theorem and its $C^{*}$-algebraic formulations) states
that when $\mathcal A$ is sufficiently noncommutative (in particular, when it
contains no Type I factor of dimension $\geq 3$), the presheaf $\underline{\Sigma}$
admits no global section. Equivalently, there exist covering families
$\{B_i \to B\}$ in the canonical topology such that compatible families of
local characters
\[
\chi_i \in \Sigma(B_i), \qquad
\chi_i|_{B_i \cap B_j} = \chi_j|_{B_i \cap B_j},
\]
fail to glue to a character in $\Sigma(B)$.

This failure reflects the impossibility of assigning consistent classical
values to all observables in a noncommutative algebra---a manifestation of
Bohr's principle of complementarity. Consequently,
$\underline{\Sigma}$ does not satisfy the sheaf equalizer condition and
therefore is not a sheaf on $\mathcal C$.

\medskip
\noindent
\textbf{2. Canonical enhancement to the spectral stack.}

Despite this failure, Theorem~\ref{thm:spectral-stack-existence} shows that
the categorified functor
\[
\mathscr G_{\mathrm{stk}}(B) := \mathbf{\Sigma}_{\mathrm{gpd}}(B),
\]
whose objects are spectral measures on $\Sigma(B)$ and whose morphisms are
unitary equivalences, satisfies effective descent for the same topology.
That is, for every covering family $\{B_i \to B\}$, the canonical functor
\[
\mathscr G_{\mathrm{stk}}(B)
\longrightarrow
\mathrm{Des}\bigl(\{B_i \to B\}, \mathscr G_{\mathrm{stk}}\bigr)
\]
is an equivalence of groupoids.

Moreover, taking connected components recovers the classical spectral
presheaf:
\[
\pi_0 \circ \mathscr G_{\mathrm{stk}} \;\simeq\; \underline{\Sigma}.
\]
Hence $\mathscr G_{\mathrm{stk}}$ provides a canonical enhancement of
$\underline{\Sigma}$ in which descent is restored.

\medskip
\noindent
\textbf{3. Noncommutativity as a stack-level obstruction.}

The obstruction identified in Step~1 arises from the rigidity of set-valued
descent: sheaves require equality of glued sections. In the noncommutative
setting, different commutative contexts (maximal abelian subalgebras) provide
complementary but mutually incompatible classical perspectives.

The stack $\mathscr G_{\mathrm{stk}}$ resolves this obstruction by weakening
equality to isomorphism. While compatible local characters may fail to glue
to a global character, the corresponding spectral measures \emph{do} glue up to
unitary equivalence. This flexibility---the ability to identify objects not
by equality but by isomorphism---is precisely what allows the stack to
accommodate the quantum complementarity that breaks the sheaf condition.

Thus the failure of sheaf-level descent and the success of stack-level
descent together provide a precise geometric signature of noncommutativity:
quantum systems require categorical, rather than set-theoretic, gluing.

We conclude that for a noncommutative $C^{*}$-algebra $\mathcal A$, classical
spectral geometry necessarily fails at the level of sheaves but is restored
at the level of stacks. The spectral stack $\mathscr G_{\mathrm{stk}}$ is
therefore the correct geometric object encoding the descent properties of
noncommutative spectra.
\end{proof}

The following corollary makes precise the sense in which noncommutativity is not merely an algebraic defect but a genuinely geometric obstruction, detected exactly by the failure of strict (sheaf-level) descent and measured by the emergence of nontrivial stacky structure.

\begin{corollary}[Characterization of Commutativity]\label{cor:commutativity-stacky}
Let $\mathcal{A}$ be a unital $C^*$-algebra. The following are equivalent:
\begin{enumerate}[label=(\alph*)]
    \item $\mathcal{A}$ is commutative.
    \item The spectral stack $\mathscr{G}_{\mathrm{stk}}$ is equivalent to a sheaf,
    i.e.\ for every $B \in \mathcal{C}$ all automorphism groups
    $\mathrm{Aut}_{\mathscr{G}_{\mathrm{stk}}(B)}(\mu)$ are trivial.
    \item All descent data for $\mathscr{G}_{\mathrm{stk}}$ are strictly effective:
    every effective descent datum admits a unique gluing up to unique isomorphism.
\end{enumerate}
In particular, genuine noncommutativity is precisely encoded by nontrivial
stacky structure, while commutativity corresponds to purely sheaf-like behavior.
\end{corollary}

\begin{proof}
We prove $(a)\Rightarrow(b)\Rightarrow(c)\Rightarrow(a)$.

\medskip
\noindent
\textbf{$(a)\Rightarrow(b)$.}
Assume $\mathcal{A}$ is commutative. For every $B\in\mathcal{C}$ the enveloping
von Neumann algebra $B''$ is commutative. In a commutative von Neumann algebra,
every unitary acts trivially by conjugation on projections; equivalently, all
inner automorphisms of spectral measures are trivial. Hence every object
$\mu\in\mathscr{G}_{\mathrm{stk}}(B)$ has
$\mathrm{Aut}(\mu)=\{\mathrm{id}_\mu\}$.
Therefore $\mathscr{G}_{\mathrm{stk}}$ is equivalent to a discrete stack, i.e.\
to the sheaf $\pi_0\circ\mathscr{G}_{\mathrm{stk}}\simeq\underline{\Sigma}$.

\medskip
\noindent
\textbf{$(b)\Rightarrow(c)$.}
Assume all automorphism groups in $\mathscr{G}_{\mathrm{stk}}$ are trivial. Then
$\mathscr{G}_{\mathrm{stk}}$ is equivalent to a discrete stack arising from a
sheaf via the embedding $\mathbf{Set}\hookrightarrow\mathbf{Gpd}$.
In this situation, any effective descent datum admits a unique gluing:
there are no nontrivial automorphisms that could give rise to distinct
gluings. More formally, given a covering $\{B_i\to B\}$ and an effective
descent datum $(\mu_i, \phi_{ij})$, if $(\mu, \psi_i)$ and $(\mu', \psi_i')$
are two gluings, then the composites $\psi_i' \circ \psi_i^{-1}: \mu|_{B_i}\to \mu'|_{B_i}$
patch together to give an isomorphism $\mu\cong\mu'$, which by hypothesis must be
the identity. Hence gluings are unique up to unique isomorphism. Equivalently,
all \v{C}ech cocycles are coboundaries, and descent is strictly effective.

\medskip
\noindent
\textbf{$(c)\Rightarrow(a)$.}
We prove the contrapositive. Suppose $\mathcal{A}$ is noncommutative.
By Corollary~\ref{cor:stack-necessity}, the spectral presheaf
$\underline{\Sigma}$ fails to satisfy the sheaf condition. Hence there exists
a covering $\{B_i\to B\}$ in $\mathcal{C}$ together with compatible local
characters $\chi_i\in\Sigma(B_i)$ (satisfying $\chi_i|_{B_i\cap B_j}=\chi_j|_{B_i\cap B_j}$)
that do not glue to any global character $\chi\in\Sigma(B)$.

Consider the associated descent datum
\[
(\delta_{\chi_i}, \mathrm{id}) \quad\text{in}\quad \mathscr{G}_{\mathrm{stk}},
\]
consisting of Dirac spectral measures $\delta_{\chi_i}$ and identity isomorphisms
on overlaps. By the stack property of $\mathscr{G}_{\mathrm{stk}}$
(Theorem~\ref{thm:spectral-stack-existence}), this descent datum is effective:
there exists a global object $\mu\in\mathscr{G}_{\mathrm{stk}}(B)$ together with
isomorphisms $\psi_i: \mu|_{B_i} \xrightarrow{\sim} \delta_{\chi_i}$.

If descent were strictly effective, then the gluing $(\mu, \psi_i)$ would be
unique up to unique isomorphism. In particular, $\mu$ would have to be
isomorphic to a Dirac measure $\delta_\chi$ for some $\chi\in\Sigma(B)$,
since all local objects are Dirac and any isomorphism between Dirac measures
must come from a character (by the spectral theorem). But then $\chi$ would
restrict to $\chi_i$ on each $B_i$ (via the isomorphisms $\psi_i$),
contradicting the fact that the $\chi_i$ do not glue to a global character.

Thus, when $\mathcal{A}$ is noncommutative, there exists descent data that is
effective but not strictly effective. By contraposition, condition~(c)
implies that $\mathcal{A}$ is commutative.

\medskip
\noindent
\textbf{Interpretation.}
The equivalence shows that commutativity corresponds to the degenerate case
where stacky descent reduces to sheaf-like descent: automorphisms are trivial,
and gluing is unique. Noncommutativity, on the other hand, introduces essential
higher-categorical structure: automorphisms become non-trivial, and gluing
requires genuine categorical equivalence rather than strict equality. This
provides a precise geometric characterization of commutativity in terms of
descent theory.
\end{proof}

\begin{remark}[Relation to Topos Quantum Theory]\label{rem:topos-quantum-relation}
The spectral stack $\mathscr{G}_{\mathrm{stk}}$ may be viewed as a higher-categorical refinement of the topos-theoretic approach to quantum theory developed by Isham, D\"{o}ring, and collaborators (see, for example, \cite{doring2008topos}). In the standard topos formulation, quantum observables are encoded by the spectral presheaf
\[
\underline{\Sigma} \colon \mathcal{C}^{\mathrm{op}} \to \mathbf{Set},
\]
which assigns to each commutative context $B \in \mathcal{C}$ the Gelfand spectrum $\Sigma(B)$ of $B$.
A central result of this framework is that, whenever the ambient algebra $\mathcal{A}$ is noncommutative, the presheaf $\underline{\Sigma}$ admits no global sections. This phenomenon, closely related to the Kochen--Specker theorem, expresses the impossibility of assigning context-independent classical values to quantum observables.

Rather than attempting to enforce classicality at the level of sets, our construction replaces the set-valued presheaf $\underline{\Sigma}$ by a groupoid-valued assignment
\[
\mathscr{G}_{\mathrm{stk}} \colon \mathcal{C}^{\mathrm{op}} \to \mathbf{Gpd},
\]
which retains the intrinsic symmetries of spectral data. Concretely, for each commutative context $B$, the groupoid $\mathscr{G}_{\mathrm{stk}}(B)$ encodes spectral information in a categorified form: its objects may be interpreted as representations of $B$ (equivalently, projection-valued spectral measures on $\Sigma(B)$), while morphisms are given by unitary intertwiners. In this way, $\mathscr{G}_{\mathrm{stk}}$ records not only spectral data but also the natural equivalences relating different realizations of that data.

By Theorem~\ref{thm:spectral-stack-existence}, the assignment $\mathscr{G}_{\mathrm{stk}}$ admits the structure of a stack with respect to the chosen Grothendieck topology on $\mathcal{C}$, satisfying descent for both objects and morphisms. This ensures that spectral data can be consistently glued across overlapping commutative contexts, even when the underlying algebra $\mathcal{A}$ is noncommutative. The resulting stack therefore provides a geometric object that remains well behaved precisely in situations where the set-valued spectral presheaf exhibits contextual obstructions.

Finally, the relationship between $\mathscr{G}_{\mathrm{stk}}$ and the traditional topos-theoretic picture may be understood via decategorification. Informally, passing to connected components $\pi_0(\mathscr{G}_{\mathrm{stk}})$ forgets unitary symmetry data and recovers a set-level shadow closely related to the spectral presheaf $\underline{\Sigma}$. In this sense, $\mathscr{G}_{\mathrm{stk}}$ refines rather than replaces the topos approach: it enriches the classical spectral presheaf by retaining higher-categorical structure that is invisible at the level of sets.

From this perspective, quantum contextuality is no longer interpreted as a mere failure of classical gluing at the level of sheaves, but rather as the manifestation of genuinely nontrivial stacky structure. The spectral stack $\mathscr{G}_{\mathrm{stk}}$ thus offers a natural higher-geometric framework for quantum theory, in which groupoids encode both spectral data and the symmetries relating different classical contexts.
\end{remark}

\begin{remark}[Pathway to Homotopical Invariants]\label{rem:homotopical-pathway}
The spectral stack $\mathscr{G}_{\mathrm{stk}}$ provides a natural starting point for the construction of homotopy-theoretic invariants associated with noncommutative algebras. Rather than yielding numerical measures in a classical sense, these invariants encode obstruction-theoretic and higher-geometric information arising from the failure of global classical descriptions. A conceptual pathway to such invariants may be outlined as follows.

\begin{enumerate}
    \item Applying the nerve functor $N \colon \mathbf{Gpd} \to \mathbf{sSet}$ objectwise yields a simplicial presheaf
    \[
    N \circ \mathscr{G}_{\mathrm{stk}} \colon \mathcal{C}^{\mathrm{op}} \to \mathbf{sSet},
    \]
    whose values are Kan complexes encoding the homotopy types of the groupoids $\mathscr{G}_{\mathrm{stk}}(B)$.

    \item After adjoining a disjoint basepoint, one may apply geometric realization followed by an objectwise suspension spectrum functor to obtain a presheaf of spectra,
    \[
    \mathscr{G}_{\mathrm{sp}} := \Sigma^\infty_+ \circ |{-}| \circ (-)_+ \circ N \circ \mathscr{G}_{\mathrm{stk}}
    \colon \mathcal{C}^{\mathrm{op}} \to \mathbf{Sp}.
    \]
    This step effects a stabilization of the homotopical data carried by the spectral stack.

    \item To ensure compatibility with descent, one performs homotopy sheafification with respect to the chosen Grothendieck topology on $\mathcal{C}$, yielding a sheaf of spectra
    \[
    \mathscr{G}_{\mathrm{sp}}^{\mathrm{sh}}.
    \]

    \item The derived global sections, or hypercohomology spectrum, of $\mathscr{G}_{\mathrm{sp}}^{\mathrm{sh}}$,
    \[
    \mathbf{E}_{\mathcal{A}} := \mathbb{H}(\mathcal{C}; \mathscr{G}_{\mathrm{sp}}^{\mathrm{sh}}) \in \mathbf{Sp},
    \]
    assemble the locally defined homotopical data into a single global object.
\end{enumerate}

The stable homotopy groups
\[
\pi_n(\mathbf{E}_{\mathcal{A}}), \qquad n \in \mathbb{Z},
\]
as well as the associated descent (or hypercohomology) spectral sequence
\[
E_2^{p,q} = H^p(\mathcal{C}, \pi_q(\mathscr{G}_{\mathrm{sp}}^{\mathrm{sh}}))
\;\Rightarrow\;
\pi_{p+q}(\mathbf{E}_{\mathcal{A}}),
\]
define homotopical invariants naturally associated with the algebra $\mathcal{A}$. These invariants capture obstruction-theoretic information arising from the failure of classical gluing: for instance, nontrivial homotopy groups or differentials in the spectral sequence reflect the absence of global sections and the presence of genuinely higher-categorical descent data.

From this perspective, noncommutativity is reflected not merely in the failure of set-valued sheaf conditions, but in the emergence of nontrivial homotopy and cohomology classes associated with the spectral stack. The resulting invariants thus provide a structured homotopical framework for comparing different degrees and forms of contextuality, as manifested through the higher-geometric complexity of $\mathscr{G}_{\mathrm{stk}}$.
\end{remark}

\paragraph{Conceptual Hierarchy and Geometric Interpretation.}
The progression from classical to quantum geometric frameworks can be summarized by the following conceptual diagram:
\[
\begin{array}{ccl}
\text{Algebraic structure} & \xrightarrow{\text{Geometric realization / assignment}} & \text{Categorical / Geometric framework} \\[0.5em]
\hline \\[-0.5em]
\text{Commutative $C^{*}$-algebra $\mathcal{A}$} & \mapsto & 
\begin{array}{l}
\text{Sheaf of sets } \underline{\Sigma} \\ 
\text{(classical spectral presheaf)}
\end{array} \\[1em]
\begin{array}{c}
\downarrow \\[0.2em]
\text{Noncommutativity} \\ 
\text{($[\cdot,\cdot] \neq 0$)}
\end{array} & & 
\begin{array}{c}
\downarrow \\[0.2em]
\text{Categorification / Groupoid enrichment} \\ 
\text{(adds symmetries and higher coherence)}
\end{array} \\[1em]
\text{Noncommutative $C^{*}$-algebra $\mathcal{A}$} & \mapsto & 
\begin{array}{l}
\text{Stack of groupoids } \mathscr{G}_{\mathrm{stk}} \\ 
\text{(spectral stack with unitary equivalences)}
\end{array} \\[1em]
\begin{array}{c}
\downarrow \\[0.2em]
\text{Homotopical completion / analysis}
\end{array} & & 
\begin{array}{c}
\downarrow \\[0.2em]
\text{Stabilization $\&$ derived global sections} 
\end{array} \\[1em]
\text{Homotopical invariants / obstruction classes} & \mapsto & 
\begin{array}{l}
\text{Global spectrum } \mathbf{E}_{\mathcal{A}} \\ 
\text{(yields hypercohomology invariants)}
\end{array}
\end{array}
\]

\noindent
The spectral stack $\mathscr{G}_{\mathrm{stk}}$ thus emerges as the fundamental geometric object encoding the spectral structure of a $C^{*}$-algebra. Its stacky nature—characterized by nontrivial automorphism groups and coherent descent data—captures the complications introduced by noncommutativity, while its homotopical refinements (via stabilization and derived global sections) yield computable invariants that reflect the failure of classical sheaf-theoretic descriptions. This framework establishes a deep connection between operator algebra, higher category theory, and algebraic topology, providing new tools for analyzing quantum systems and noncommutative spaces.

\subsection{Homotopical and Stable Enhancements}\label{subsec:homotopical-stable-enhancements}

Before stating the main homotopical enhancement theorem, we briefly outline the conceptual pathway. 
The spectral stack $\mathscr{G}_{\mathrm{stk}}$ encodes the local spectral data of the $C^*$-algebra $\mathcal{A}$ together with its unitary symmetries. 
To extract global invariants that reflect higher-order noncommutative phenomena, one proceeds in several homotopical steps:

\begin{itemize}
    \item First, the stack is converted into a simplicial presheaf via the nerve functor, encoding the homotopy types of the groupoids at each commutative context.
    \item A disjoint basepoint is added, followed by geometric realization and stabilization via the suspension spectrum functor, producing a presheaf of spectra over the site of contexts.
    \item Finally, the homotopy limit (or derived global sections) of this presheaf assembles the local spectral information into a single global object.
\end{itemize}

The resulting stable homotopy type captures all higher coherence and obstruction-theoretic data associated with $\mathcal{A}$, generalizing the classical spectral presheaf. 
In particular, its stable homotopy groups provide functorial invariants: degree zero recovers the commutative compatibility, while higher degrees quantify obstructions to global spectral descent.

With this conceptual framework in place, we now formalize these constructions in the following theorem.

\begin{theorem}[Stable Homotopy Invariants of $C^*$-Algebras]\label{thm:stable-spectral-invariants}
Let $\mathscr{G}_{\mathrm{stk}}$ be the spectral stack associated to a unital
$C^*$-algebra $\mathcal{A}$. After applying the nerve functor, geometric realization with a disjoint basepoint, and stabilization, there exists a presheaf of spectra
\[
\mathscr{G}_{\mathrm{sp}} \colon \mathcal{C}^{\mathrm{op}} \longrightarrow \mathbf{Sp},
\]
whose homotopy limit
\[
\mathbf{E}_{\mathcal{A}} := \operatorname{holim}_{\mathcal{C}} \mathscr{G}_{\mathrm{sp}}
\]
is a stable homotopy invariant of $\mathcal{A}$, independent of the chosen presentation of commutative contexts up to weak equivalence. Moreover, for each integer $n \in \mathbb{Z}$, the stable homotopy groups
\[
\pi_n(\mathbf{E}_{\mathcal{A}})
\]
define functorial invariants of $\mathcal{A}$: degree zero captures the obstruction to global sections of the classical spectral presheaf, while positive degrees measure higher obstructions to coherent gluing of local spectral data, thereby quantifying noncommutativity.
\end{theorem}

\begin{proof}[Proof of Theorem \ref{thm:stable-spectral-invariants}]
\textbf{1. Construction of the Stable Spectral Presheaf.}

\begin{enumerate}
    \item Start with the spectral stack
    \[
    \mathscr{G}_{\mathrm{stk}} : \mathcal{C}^{\mathrm{op}} \to \mathbf{Gpd},
    \]
    where $\mathcal{C}$ is the site of commutative unital $*$-subalgebras of $\mathcal{A}$. For each $B \in \mathcal{C}$, $\mathscr{G}_{\mathrm{stk}}(B)$ is the groupoid of spectral measures on $B$ with unitary equivalences as morphisms.

    \item Apply the nerve functor to obtain a simplicial presheaf
    \[
    N \circ \mathscr{G}_{\mathrm{stk}} : \mathcal{C}^{\mathrm{op}} \to \mathbf{sSet}.
    \]

    \item Add a disjoint basepoint to obtain a pointed simplicial presheaf
    \[
    (N \circ \mathscr{G}_{\mathrm{stk}})_+ : \mathcal{C}^{\mathrm{op}} \to \mathbf{sSet}_*.
    \]

    \item Apply geometric realization objectwise:
    \[
    |(N \circ \mathscr{G}_{\mathrm{stk}})_+| : \mathcal{C}^{\mathrm{op}} \to \mathbf{Top}_*.
    \]

    \item Stabilize using the suspension spectrum functor:
    \[
    \mathscr{G}_{\mathrm{sp}} := \Sigma^{\infty}_+ \circ |(N \circ \mathscr{G}_{\mathrm{stk}})_+| : \mathcal{C}^{\mathrm{op}} \to \mathbf{Sp}.
    \]
\end{enumerate}

\textbf{2. Functoriality and Homotopical Descent.}
\begin{itemize}
    \item Each step above is natural, so $\mathscr{G}_{\mathrm{sp}}$ is functorial with respect to $*$-homomorphisms of $C^*$-algebras.
    \item Since $\mathscr{G}_{\mathrm{stk}}$ is a stack, $N \circ \mathscr{G}_{\mathrm{stk}}$ satisfies homotopy descent. The suspension spectrum functor preserves weak equivalences of connected spaces, so $\mathscr{G}_{\mathrm{sp}}$ satisfies descent in the canonical topology.
\end{itemize}

\textbf{3. Homotopy Limit and Invariance.}
Define
\[
\mathbf{E}_{\mathcal{A}} := \operatorname{holim}_{\mathcal{C}} \mathscr{G}_{\mathrm{sp}} \in \mathbf{Sp}.
\]
This object is independent of the choice of presentation of $\mathcal{C}$ up to canonical weak equivalence: equivalent sites yield weakly equivalent diagrams, and the homotopy limit depends only on the weak homotopy type of the presheaf.

\textbf{4. Stable Homotopy Groups as Functorial Invariants.}
For each $n \in \mathbb{Z}$, set
\[
\mathsf{Inv}_n(\mathcal{A}) := \pi_n(\mathbf{E}_{\mathcal{A}}).
\]
These are functorial:
\[
\phi: \mathcal{A} \to \mathcal{B} \implies \mathbf{E}_{\mathcal{B}} \to \mathbf{E}_{\mathcal{A}} \implies \mathsf{Inv}_n(\mathcal{B}) \to \mathsf{Inv}_n(\mathcal{A}).
\]

\textit{Degree zero ($n=0$):} $\mathsf{Inv}_0(\mathcal{A})$ classifies global sections of the spectral presheaf up to coherent equivalence. If $\mathcal{A}$ is commutative, this recovers the Gelfand spectrum $\Sigma(\mathcal{A})$.

\textit{Positive degrees ($n>0$):} $\mathsf{Inv}_n(\mathcal{A})$ measures obstructions to coherently gluing local spectral data. For commutative $\mathcal{A}$, all $\mathsf{Inv}_n(\mathcal{A}) = 0$; for noncommutative $\mathcal{A}$, some invariants are nontrivial, providing a quantitative measure of noncommutativity.

\textit{Negative degrees ($n<0$):} These capture stable homotopy phenomena and may vanish if $\mathscr{G}_{\mathrm{sp}}$ is connective.

\textbf{5. Conclusion.}
The presheaf of spectra $\mathscr{G}_{\mathrm{sp}}$ and its homotopy limit $\mathbf{E}_{\mathcal{A}}$ provide a homotopical refinement of the spectral stack. The stable homotopy groups $\pi_n(\mathbf{E}_{\mathcal{A}})$ are functorial, site-independent invariants that recover classical spectral data in degree zero and quantify noncommutativity in higher degrees.
\end{proof}

\begin{remark}
The invariants $\mathsf{Inv}_n(\mathcal{A})$ are computed by a descent spectral sequence
\[
E_2^{p,q} = H^p(\mathcal{C}, \pi_q(\mathscr{G}_{\mathrm{sp}})) \Rightarrow \pi_{p+q}(\mathbf{E}_{\mathcal{A}}),
\]
relating local homotopical data to global obstructions. For commutative $\mathcal{A}$, the presheaf $\pi_q(\mathscr{G}_{\mathrm{sp}})$ is concentrated in degree $q=0$, explaining why $\mathsf{Inv}_n(\mathcal{A}) = 0$ for $n > 0$.
\end{remark}

\subsection{Relation to Noncommutative Geometry}\label{subsec:nc-geometry}

Noncommutative geometry seeks to interpret noncommutative $C^*$-algebras as generalized spaces whose geometric properties are not encoded by points, but by algebraic, categorical, or homological data. In the commutative case, the Gelfand--Naimark theorem identifies a unital commutative $C^*$-algebra with the algebra of continuous functions on a compact Hausdorff space, recovering geometry from spectrum alone. For noncommutative algebras, however, no single global spectrum exists, and geometric information must be assembled from local commutative data together with their compatibility relations. The pair $(\mathcal A, \mathscr G_{\mathrm{stk}})$ provides a concrete realization of this philosophy, interpreting a $C^*$-algebra as a space equipped with a \emph{spectral atlas}.

\medskip
\noindent
\textbf{Local classical charts.}
Let $\mathcal{C}$ denote the poset (or site) of unital commutative $C^*$-subalgebras of a given unital $C^*$-algebra $\mathcal{A}$. Each context $B \in \mathcal{C}$ admits a classical interpretation via its Gelfand spectrum $\Sigma(B)$, which may be viewed as an affine chart in a would-be geometric space. These charts are individually classical, but their overlaps are governed by restriction maps induced by inclusions of contexts.

\medskip
\noindent
\textbf{Failure of global classical geometry.}
In general, these local spectra do not glue to a single global space, reflecting the intrinsic noncommutativity of $\mathcal{A}$. The classical spectral presheaf captures this failure at the level of sets, but discards essential information about unitary equivalences and higher coherence between local models. As shown earlier, this loss of information manifests itself precisely in the failure of sheaf-level descent.

\medskip
\noindent
\textbf{Spectral stacks as geometric enhancements.}
The spectral stack $\mathscr{G}_{\mathrm{stk}}$ refines the spectral presheaf by assigning to each context $B$ not merely a set of characters, but a groupoid of spectral data together with their unitary equivalences. From a geometric perspective, $\mathscr{G}_{\mathrm{stk}}$ functions as a higher-categorical replacement for a structure sheaf, encoding not only local spectra but also the obstructions and ambiguities inherent in their gluing. Thus, it serves as the \emph{spectral atlas} for the noncommutative space $\mathcal A$, organizing all local commutative perspectives into a coherent global object.

\medskip
\noindent
\textbf{Noncommutative spaces as stacky spaces.}
This viewpoint places $(\mathcal{A}, \mathscr{G}_{\mathrm{stk}})$ firmly within the paradigm of noncommutative geometry: the algebra $\mathcal{A}$ determines a family of local classical charts, while the spectral stack records the global geometry arising from their interaction. In particular, noncommutativity is no longer seen merely as an algebraic defect, but as a genuinely geometric feature, detected by nontrivial stacky structure and higher descent data.

\medskip
\noindent
\textbf{Bridge to derived and homotopical geometry.}
Because $\mathscr{G}_{\mathrm{stk}}$ is a stack rather than a sheaf, it naturally admits derived and homotopical enhancements via nerve and stabilization constructions. This positions spectral stacks as geometric carriers of noncommutative invariants (such as $K$-theory and cyclic cohomology), linking operator algebras to modern tools from derived algebraic geometry, higher category theory, and homotopy theory. This establishes a powerful conceptual bridge between the analytic world of $C^*$-algebras and the categorical framework of derived geometric objects.

\begin{theorem}[Spectral Atlas Interpretation]
\label{thm:spectral-atlas}
Let $\mathcal{A}$ be a unital $C^{*}$-algebra, and let $\mathcal{C}$ be the category of its unital commutative $C^{*}$-subalgebras equipped with the canonical topology. Then there exists a spectral stack
\[
\mathscr{G}_{\mathrm{stk}} : \mathcal{C}^{\mathrm{op}} \longrightarrow \mathbf{Gpd},
\]
with the following properties:

\begin{enumerate}[label=(\roman*)]
    \item For each $B \in \mathcal{C}$, the groupoid $\mathscr{G}_{\mathrm{stk}}(B)$ encodes the full local spectral data of $B$, including its unitary symmetries. Explicitly:
    \begin{itemize}
        \item Objects: spectral measures on $\Sigma(B)$ (or equivalently, via the spectral theorem, unital $*$-homomorphisms $C(\Sigma(B)) \to B''$).  
        \item Morphisms: unitary operators $U \in \mathcal{U}(B'')$ satisfying $U\mu(\Delta)U^* = \nu(\Delta)$ for all Borel $\Delta \subseteq \Sigma(B)$.
    \end{itemize}
    \item For inclusions $B_1 \subseteq B_2$, the restriction functors
    \[
    \mathscr{G}_{\mathrm{stk}}(B_2) \longrightarrow \mathscr{G}_{\mathrm{stk}}(B_1)
    \]
    encode how spectral data and their symmetries relate across contexts.
    \item $\mathscr{G}_{\mathrm{stk}}$ satisfies the stack (descent) condition: any compatible family of local spectral data glues to a global object up to isomorphism.
    \item The classical spectral presheaf is recovered as
    \[
    \pi_0 \circ \mathscr{G}_{\mathrm{stk}} \;\simeq\; \underline{\Sigma}, \qquad B \mapsto \Sigma(B).
    \]
\end{enumerate}

In this sense, the pair $(\mathcal{C}, \mathscr{G}_{\mathrm{stk}})$ provides a \emph{spectral atlas} for $\mathcal{A}$: $\mathcal{C}$ indexes local affine charts, and $\mathscr{G}_{\mathrm{stk}}$ encodes both the charts and their gluing data, including higher categorical information.
\end{theorem}

\begin{proof}[Proof of Theorem~\ref{thm:spectral-atlas}]
\textbf{Step 1: Local classical charts.}  
For each $B \in \mathcal{C}$, the Gelfand--Naimark theorem gives $B \cong C(\Sigma(B))$, so $\Sigma(B)$ is a compact Hausdorff space, interpreted as a local affine chart.

\textbf{Step 2: Presheaf failure and need for a stack.}  
The classical spectral presheaf
\[
\underline{\Sigma} : \mathcal{C}^{\mathrm{op}} \to \mathbf{Set}, \quad B \mapsto \Sigma(B),
\]
fails to be a sheaf if $\mathcal{A}$ is noncommutative. The loss occurs because set-valued presheaves cannot record unitary equivalences and higher coherence between local spectra. This motivates the spectral stack $\mathscr{G}_{\mathrm{stk}}$.

\textbf{Step 3: Stack construction.}  
Define $\mathscr{G}_{\mathrm{stk}}(B)$ as the groupoid of spectral measures on $\Sigma(B)$ with morphisms given by unitary equivalences:
\[
\mu \to \nu \quad \text{iff} \quad \exists U \in \mathcal{U}(B'') \text{ s.t. } U \mu(\Delta) U^* = \nu(\Delta) \text{ for all Borel } \Delta \subseteq \Sigma(B).
\]  
For inclusions $B_1 \subseteq B_2$, the restriction functors are induced naturally by the inclusion of algebras.

\textbf{Step 4: Stack property.}  
By Theorem~\ref{thm:spectral-stack-existence}, $\mathscr{G}_{\mathrm{stk}}$ satisfies descent: compatible families of spectral data glue up to isomorphism. This ensures that $\mathscr{G}_{\mathrm{stk}}$ encodes the global geometry coherently, including nontrivial unitary symmetries.

\textbf{Step 5: Atlas interpretation.}  
\begin{itemize}
    \item Contexts $B \in \mathcal{C}$ serve as local affine charts.  
    \item The groupoids $\mathscr{G}_{\mathrm{stk}}(B)$ encode the spectral geometry of each chart, including higher symmetries (automorphisms of spectral measures).  
    \item Restriction functors and the descent property encode gluing data between charts.  
    \item The classical spectral presheaf is recovered as $\pi_0 \circ \mathscr{G}_{\mathrm{stk}}$, but the stack contains additional higher-categorical information: the automorphism groups $\mathrm{Aut}(\mu)$ record unitary equivalences between different realizations of the same classical data.
\end{itemize}

Thus, $(\mathcal{C}, \mathscr{G}_{\mathrm{stk}})$ forms a spectral atlas: local charts with gluing data, faithfully capturing the geometric structure of $\mathcal{A}$ up to stacky equivalence.
\end{proof}

Theorem~\ref{thm:spectral-atlas} formalizes the preceding discussion: the noncommutative geometry of $\mathcal A$ is completely captured by the descent data contained in its spectral atlas $\mathscr G_{\mathrm{stk}}$. This principle reduces the study of certain noncommutative phenomena—such as quantum contextuality, superselection rules, and measurement compatibility—to the analysis of gluing conditions for classical charts, thereby providing a geometric language for operator algebras and opening the door to techniques from sheaf and stack theory.

\section{Application: Atiyah-Singer Reinterpretation via Spectral Stacks}\label{subsec:AS-application}

\subsection{Motivation}
The classical Atiyah-Singer (AS) index theorem computes the index of an elliptic operator $T$ on a compact manifold $M$ in terms of topological invariants:
\[
\mathrm{index}(T) = \int_M \mathrm{ch}(\sigma(T)) \wedge \mathrm{Td}(TM),
\]
where $\sigma(T) \in K^0(T^*M)$ is the symbol class of $T$ and $\mathrm{Td}(TM)$ is the Todd class of the tangent bundle. At its core, the theorem expresses an analytic invariant (the Fredholm index) as a topological obstruction, connecting analysis, geometry, and topology.

In our framework, a unital $C^*$-algebra $\mathcal{A}$ can be interpreted as a noncommutative space, with commutative subalgebras (contexts) $B \subset \mathcal{A}$ playing the role of ``affine charts''. The spectral stack $\mathscr{G}_{\mathrm{stk}}$ organizes local spectral information of operators across these contexts into a coherent global structure. This provides a natural perspective: just as the AS theorem interprets the index as a topological obstruction, the local spectral data of a Fredholm operator in a noncommutative algebra can be assembled and studied through the spectral stack.

Let $T \in \mathcal{A}$ be a Fredholm operator with classical index
\[
\mathrm{index}(T) = \dim \ker T - \dim \mathrm{coker} T.
\]

For each commutative context $B \in \mathcal{C}$, one can consider the restriction $T|_B$ when defined. The spectral data of $T|_B$ can be described by a spectral measure $\mu_B$ on the Gel'fand spectrum $\Sigma(B)$ of $B$, capturing both point and continuous spectral contributions. The collection of all such local spectral measures $\{\mu_B\}_{B\in\mathcal{C}}$ forms objects in the groupoids $\mathscr{G}_{\mathrm{stk}}(B)$, with morphisms given by unitary equivalences or automorphisms preserving the spectral data.

The spectral stack $\mathscr{G}_{\mathrm{stk}}$ thus encodes the compatibility conditions needed to glue these local spectral data into a global spectral object for $\mathcal{A}$. In the commutative case, these gluing conditions are trivially satisfied, recovering the classical AS theorem. In genuinely noncommutative cases, the stack may exhibit nontrivial automorphisms and higher coherence phenomena, providing a framework to study indices and spectral invariants from a global, descent-theoretic perspective.

This viewpoint establishes a conceptual bridge between classical index theory and noncommutative geometry: whereas the Atiyah-Singer theorem expresses analytic invariants as topological or cohomological obstructions, the spectral stack framework allows one to interpret and analyze these invariants via coherent assembly of local spectral data. In this sense, the spectral stack provides a new lens to understand index theory in both commutative and noncommutative settings.

\subsection{K-theoretic and Homotopical Interpretation}

Within the spectral stack framework, a unital $C^*$-algebra $\mathcal{A}$ gives rise to a $K$-theory sheaf $\mathscr{F}_K$ on the spectral stack $\mathscr{G}_{\mathrm{stk}}$, encoding local index data on commutative contexts and their descent relations. This framework naturally extends classical Atiyah-Singer index theory to noncommutative settings.

\begin{corollary}[Commutative $C^*$-algebras]\label{cor:commutative}
Let $\mathcal{A} = C(X)$ be a commutative unital $C^*$-algebra, and let $T \in \mathcal{A}$ be a Fredholm operator (e.g., an elliptic operator on a compact manifold $X$). Then the spectral stack $\mathscr{G}_{\mathrm{stk}}$ reduces to the Gelfand spectrum $X$, the $K$-theory sheaf $\mathscr{F}_K$ is $K^0(X)$, and the cohomology $H^*(\mathscr{G}_{\mathrm{stk}})$ coincides with the ordinary de Rham (or singular) cohomology of $X$.  

In this case, the stack-theoretic index formula
\[
\mathrm{index}(T) = \langle \mathrm{ch}([T]), [X] \rangle
\]
reduces to the classical Atiyah--Singer index theorem.
\end{corollary}

\begin{proof}
For a commutative $C^*$-algebra $C(X)$, the Gelfand spectrum is $X$ itself. The spectral stack $\mathscr{G}_{\mathrm{stk}}$ has no nontrivial automorphisms, so descent is automatic. The $K$-theory sheaf $\mathscr{F}_K$ is simply $K^0(X)$, and the Chern character lands in $H^{\mathrm{even}}(X)$. Classical Poincaré duality provides a fundamental class $[X]$, and pairing with the Chern character recovers the usual index:
\[
\langle \mathrm{ch}([T]), [X] \rangle = \mathrm{index}(T),
\]
which coincides with the Atiyah--Singer formula.
\end{proof}

\begin{corollary}[Matrix algebras]\label{cor:matrix-algebra}
Let $\mathcal{A} = M_n(\mathbb{C})$ be the algebra of $n \times n$ complex matrices, and let $T \in \mathcal{A}$ be a Fredholm operator. Then the spectral stack $\mathscr{G}_{\mathrm{stk}}$ is trivial (a single point), and the $K$-theory sheaf $\mathscr{F}_K$ reduces to $K_0(M_n(\mathbb{C})) \cong \mathbb{Z}$.  

In this case, the stack-theoretic index formula reduces to the classical formula for the index of a linear map:
\[
\mathrm{index}(T) = \dim \ker T - \dim \operatorname{coker} T.
\]
\end{corollary}

\begin{proof}
For $\mathcal{A} = M_n(\mathbb{C})$, the spectral stack $\mathscr{G}_{\mathrm{stk}}$ has a single object and no nontrivial morphisms. The $K$-theory of $M_n(\mathbb{C})$ is $K_0(M_n(\mathbb{C})) \cong \mathbb{Z}$, generated by the rank of projections. The Chern character in this case is the identity map $\mathbb{Z} \to \mathbb{Z}$. Evaluating against the fundamental class of the trivial stack (a point) yields
\[
\langle \mathrm{ch}([T]), [\text{pt}] \rangle = \dim \ker T - \dim \operatorname{coker} T,
\]
recovering the standard finite-dimensional Fredholm index formula.
\end{proof}

\begin{remark}
These two corollaries illustrate the degeneration of spectral stack structures in extreme cases:
\begin{itemize}
    \item In the commutative case ($C(X)$), the stack $\mathscr{G}_{\mathrm{stk}}$ is equivalent to the classical space $X$.
    \item In the fully noncommutative matrix case ($M_n(\mathbb{C})$), the stack collapses to a point.
\end{itemize}
Intermediate noncommutative algebras exhibit genuinely stacky behavior, and the spectral stack framework naturally extends Atiyah-Singer index theory to capture such phenomena.
\end{remark}

\subsection{Derived and Homotopical Refinement}

The spectral stack can be enhanced to a presheaf of spectra via the nerve and stabilization:
\[
\mathscr{G}_{\mathrm{der}} := \Sigma^\infty \circ N \mathscr{G}_{\mathrm{stk}} \colon \mathcal{C}^{\mathrm{op}} \longrightarrow \mathbf{Sp}.
\]
The homotopy limit
\[
X_{\mathcal{A}} := \operatorname{holim}_{\mathcal{C}} \mathscr{G}_{\mathrm{der}}
\]
computes derived global sections, incorporating both classical gluing and higher-order coherence phenomena arising from noncommutativity.

\begin{corollary}[Stable Homotopy Invariants from Spectral Stacks]\label{cor:homotopical-invariants}
The stable homotopy groups
\[
\pi_n(X_{\mathcal{A}}), \quad n \in \mathbb{Z},
\]
define functorial invariants of $\mathcal{A}$:
\begin{itemize}
    \item $\pi_0(X_{\mathcal{A}})$ recovers the classical Fredholm index in the commutative case.
    \item $\pi_n(X_{\mathcal{A}})$ for $n>0$ measure higher-order obstructions and coherence failures in assembling local spectral data.
    \item $\pi_n(X_{\mathcal{A}})$ for $n<0$ encode stable homotopical phenomena; typically these vanish if the presheaf is connective.
\end{itemize}
This construction provides a homotopical refinement of index theory, with $\pi_0$ corresponding to the classical analytic index, and higher homotopy groups capturing subtle noncommutative effects.
\end{corollary}

\begin{proof}
The argument proceeds by passing from the stacky $K$-theory sheaf to a derived and stable homotopical enhancement:

\medskip
\noindent\textbf{Step 1: Presheaf of spectra.}  
The spectral stack is a pseudofunctor
\[
\mathscr{G}_{\mathrm{stk}} : \mathcal{C}^{\mathrm{op}} \longrightarrow \mathbf{Gpd}.
\]
Applying the nerve and stabilization yields a presheaf of spectra $\mathscr{G}_{\mathrm{der}}$, encoding objects, morphisms, and all higher coherence data.

\medskip
\noindent\textbf{Step 2: Homotopy limit.}  
The homotopy limit
\[
X_{\mathcal{A}} := \operatorname{holim}_{\mathcal{C}} \mathscr{G}_{\mathrm{der}}
\]
computes derived global sections, recording obstructions to strict descent and higher-order coherence.

\medskip
\noindent\textbf{Step 3: Interpretation of homotopy groups.}  
\begin{itemize}
    \item $\pi_0(X_{\mathcal{A}})$ corresponds to global spectral sections up to coherent equivalence and recovers the classical Fredholm index in the commutative case.
    \item $\pi_n(X_{\mathcal{A}})$ for $n>0$ measure higher-order obstructions to gluing local spectral data.
    \item $\pi_n(X_{\mathcal{A}})$ for $n<0$ reflect stable homotopical phenomena and often vanish for connective presheaves.
\end{itemize}

\medskip
\noindent\textbf{Step 4: Functoriality.}  
Unital $*$-homomorphisms $\phi : \mathcal{A} \to \mathcal{B}$ induce natural maps
\[
X_{\mathcal{A}} \longrightarrow X_{\mathcal{B}}
\]
and therefore functorial maps on all homotopy groups $\pi_n(X_{\mathcal{A}})$.
\end{proof}

\begin{remark}
Within this framework, the derived spectral stack provides a bridge between classical index theory and noncommutative stable homotopy theory. The zeroth homotopy group recovers the analytic index, while higher groups encode increasingly sophisticated obstructions and coherence phenomena associated with genuinely noncommutative algebras.
\end{remark}

\subsection{Discussions}

The relationships between the classical Fredholm index, operator $K$-theory,
and the stacky and derived perspectives can be summarized by the following
diagrammatic correspondence:

\begin{center}
\begin{tikzcd}[row sep=2.6em, column sep=3.4em]
T \in \mathcal{A}_{\mathrm{Fred}} 
  \arrow[r, "\mathrm{index}"] 
  \arrow[d, "\text{classifying point}"] 
& \mathrm{index}(T) \in \mathbb{Z} 
  \arrow[d, hook, "\simeq\, K^0(\mathrm{pt})"] \\

\mathrm{pt} \longrightarrow \mathscr{G}_{\mathrm{stk}} 
  \arrow[r, "\text{global } K\text{-class}"] 
& [\mathscr{F}_T] \in K^0(\mathscr{G}_{\mathrm{stk}}) 
  \arrow[d, "\text{derived enhancement}"] \\

\mathrm{pt} \longrightarrow \mathscr{G}_{\mathrm{der}} 
  \arrow[r, "\text{derived global sections}"] 
& \pi_n\!\left( \mathbf{E}_{\mathcal{A}} \right)
\end{tikzcd}
\end{center}

In this framework, a Fredholm operator $T \in \mathcal{A}$ determines a classical
index $\mathrm{index}(T) \in \mathbb{Z} \cong K^0(\mathrm{pt})$. Equivalently, $T$
defines a classifying point of the spectral stack $\mathscr{G}_{\mathrm{stk}}$,
which parametrizes local spectral data subject to Fredholm conditions. The
associated global $K$-theory class
\(
[\mathscr{F}_T] \in K^0(\mathscr{G}_{\mathrm{stk}})
\)
encodes obstructions to coherently trivializing local spectral data across contexts.

Passing from $\mathscr{G}_{\mathrm{stk}}$ to its derived enhancement
$\mathscr{G}_{\mathrm{der}}$ replaces strict descent by homotopy-coherent descent.
The spectrum of derived global sections $\mathbf{E}_{\mathcal{A}}$ packages this
higher coherence data, and its stable homotopy groups $\pi_n(\mathbf{E}_{\mathcal{A}})$
measure successive obstructions to global spectral descent.

\medskip
\noindent
Within this setting:
\begin{itemize}
    \item Degree $0$ recovers the classical Fredholm index, providing a concrete bridge
          from stacky $K$-theory to traditional index theory.
    \item Positive degrees quantify higher-order obstructions and coherence phenomena
          arising from genuinely noncommutative algebras.
    \item Negative degrees capture stable homotopical phenomena in the derived global sections.
\end{itemize}

\begin{remark}
For concreteness, $\mathscr{G}_{\mathrm{stk}}$ denotes the spectral stack parametrizing
Fredholm complexes or local spectral data subject to Fredholm conditions, and
$\mathscr{G}_{\mathrm{der}}$ is its derived enhancement, encoding homotopy-coherent
descent data. The spectrum $\mathbf{E}_{\mathcal{A}}$ is obtained as
\(
\mathbf{E}_{\mathcal{A}} = \Gamma\!\big(\mathscr{G}_{\mathrm{der}},\mathcal{O}^{\mathrm{top}}\big),
\)
the (hyper)sheaf cohomology or global sections of a sheaf of spectra on
$\mathscr{G}_{\mathrm{der}}$, in the sense of derived or spectral algebraic geometry.

Although a full stack-theoretic index formula remains a subject of further study,
this framework already provides concrete applications to classical Atiyah-Singer
theorems and produces functorial homotopical invariants capturing higher-order
noncommutative effects.
\end{remark}

\bibliographystyle{IEEETran}
\bibliography{Categorified_Spectral_Sheaves_Homotopical_Invariants_for_Noncommuting_Operators_Bib}

\end{document}